\documentclass[11pt,letterpaper]{amsart}
\usepackage[foot]{amsaddr}

\usepackage{mathtools}
\usepackage{amsmath}
\usepackage[T1]{fontenc}
\usepackage[latin9]{inputenc}
\usepackage{geometry}
\usepackage{wrapfig}
\usepackage{multicol}
\usepackage{graphicx}
\usepackage{soul}
\usepackage{xcolor}
\usepackage{amssymb}
\usepackage{placeins}
\usepackage{bbm}
\usepackage{cite}
\usepackage{caption}
\usepackage{enumerate}
\usepackage{afterpage}
\usepackage{enumitem}
\usepackage{bmpsize}
\usepackage{hyperref}
\usepackage{tabu}
\usepackage{enumitem}
\numberwithin{equation}{section}
\usepackage{stmaryrd}
\usepackage{tikz}
\usetikzlibrary{matrix,graphs,arrows,positioning,calc,decorations.markings,shapes.symbols}
\usepackage{ifthen}
\usetikzlibrary{math}
\usetikzlibrary{matrix,graphs,arrows,positioning,calc,decorations.markings,shapes.symbols}

\makeatletter
\renewcommand{\email}[2][]{%
  \ifx\emails\@empty\relax\else{\g@addto@macro\emails{,\space}}\fi%
  \@ifnotempty{#1}{\g@addto@macro\emails{\textrm{(#1)}\space}}%
  \g@addto@macro\emails{#2}%
}
\makeatother

\newtheorem{theorem}{Theorem}[section]
\newtheorem{lemma}[theorem]{Lemma}
\newtheorem{proposition}[theorem]{Proposition}

{ \theoremstyle{definition}
\newtheorem{definition}[theorem]{Definition}}
{ \theoremstyle{remark}
\newtheorem{remark}[theorem]{Remark}}

\newcommand{\N}{\mathbb{N}}

\newcommand{\R}{\mathbb{R}}
\newcommand{\E}{\mathbb{E}}

\renewcommand{\P}{\mathbb{P}}
\newcommand{\pr}{\mathbb{P}}

\newcommand{\weyl}{W^\circ}
\newcommand{\weylc}{\overline{W}}

\newcommand{\im}{\mathsf{i}}
\newcommand{\Real}{\mathrm{Re}\hspace{0.5mm}}
\newcommand{\Imag}{\mathrm{Im}\hspace{0.5mm}}

\newcommand{\cev}[1]{\reflectbox{\ensuremath{\vec{\reflectbox{\ensuremath{#1}}}}}}
\newcommand{\kbmk}{K^{\mathrm{DBM}, \mathsf{m}_k^a}}
\newcommand{\kbmn}{K^{\mathrm{DBM}, n}}

\newcommand{\parF}{\mathcal{P}_{\mathsf{fin}}}
\newcommand{\parP}{\mathcal{P}_{\mathsf{pos}}}

\DeclarePairedDelimiter\abs{\lvert}{\rvert}

\newcommand{\Abs}[1]{\left\lvert #1 \right\rvert}

\newcommand{\DBM}{\mathrm{DBM}}
\newcommand{\Ai}{\mathrm{Airy}}
\newcommand{\diff}{\mathop{}\!\mathrm{d}}

\newcommand{\oparen}[1]{\mathopen{}\left( #1 \right)\mathclose{}}
\newcommand{\paren}[1]{\left( #1 \right)}
\newcommand{\sfm}{\mathsf{m}}
\renewcommand{\Re}{\operatorname{Re}}

\title{Long-time asymptotics for Airy wanderer line ensembles}
\date{\today}
\author{Alexander Clay}
\author{Evgeni Dimitrov}
\author{Rundong Ding}
\author{Alex Fu}

\begin{document}

\begin{abstract}
We investigate the long-time behavior of the Airy wanderer line ensembles, an infinite-parameter family of Brownian Gibbsian line ensembles arising as edge-scaling limits of inhomogeneous models in the Kardar--Parisi--Zhang universality class. These ensembles are governed by sequences of nonnegative parameters that encode the asymptotic slopes of the curves at positive and negative infinity. Our main results characterize the fluctuations around this leading-order behavior and establish functional limit theorems for the ensembles near both ends of the spatial axis.

We show that, at a macroscopic level, an Airy wanderer line ensemble organizes into groups of finitely many curves sharing a common asymptotic slope. After appropriate centering and scaling, each such group converges to a Dyson Brownian motion whose dimension equals the size of the group. In the case where only finitely many slope parameters are positive, we further prove a curve separation phenomenon: the upper curves follow deterministic parabolic trajectories, while the remaining lower curves remain globally flat and converge to the classical Airy line ensemble.
\end{abstract}

\maketitle
\tableofcontents

%
%
\section{Introduction and main results}\label{Section1}

%
%
\subsection{Preface}\label{Section1.1} The {\em Airy line ensemble} $\mathcal{A}$ is a sequence $\{\mathcal{A}_i : \mathbb{R} \to \mathbb{R}\}_{i \geq 1}$ of random continuous functions defined on a common probability space and satisfying an almost sure strict ordering: $\mathcal{A}_i(t) > \mathcal{A}_{i+1}(t)$ for all $i \geq 1$ and all $t \in \mathbb{R}$. This ensemble emerges as the universal edge-scaling limit in a wide range of probabilistic models, including time-dependent Wigner matrices (notably {\em Dyson Brownian motion}) \cite{Sod15}, lozenge tilings \cite{AH21}, and several integrable systems of non-intersecting random walks and last passage percolation (LPP) \cite{DNV19}. A complete construction of the Airy line ensemble was first provided in \cite{CorHamA}, where it was obtained as the weak edge-scaling limit of the {\em Brownian watermelon}, although a number of its finite-dimensional marginals had appeared earlier in the literature. In particular, the top curve $\mathcal{A}_1$ is a stationary process known as the {\em Airy process}, originally identified as the scaling limit of height functions in the {\em polynuclear growth model} \cite{J03, Spohn}. Its one-point distribution is given by the {\em (GUE) Tracy--Widom distribution} from random matrix theory \cite{TWPaper}. Owing to its universality and its central role in the construction of the {\em Airy sheet} \cite{DOV22}, the Airy line ensemble has become a fundamental object in the study of the {\em Kardar--Parisi--Zhang (KPZ)} universality class \cite{CU2}.

In \cite{dimitrov2024airy, dimitrov2024tightness}, the second author introduced infinite-parameter extensions $\mathcal{A}^{a,b,c} = \{\mathcal{A}^{a,b,c}_i : \mathbb{R} \to \mathbb{R}\}_{i \geq 1}$ of the Airy line ensemble, which we refer to as the {\em Airy wanderer line ensembles}. These ensembles arise as weak limits of inhomogeneous {\em Schur processes}, or equivalently of {\em geometric LPP} models; see \cite{DY25b} for a concise discussion of the distributional equivalence between these frameworks. The construction in \cite{dimitrov2024airy, dimitrov2024tightness} extends earlier finite-parameter generalizations of the Airy line ensemble obtained in \cite{CorHamA}, which relied on correlation kernel formulas from \cite{AFM}, and is based on a careful analysis of the kernel representations developed in \cite{BP08}. The Airy wanderer line ensembles are expected to appear as universal scaling limits for KPZ-type models with inhomogeneities or spiked initial data. For example, the one-point marginals of their top curves are described by the {\em Baik--Ben Arous--P{\'e}ch{\'e} (BBP) distributions}, which arise in a variety of contexts, including sample covariance matrices \cite{BBP05}, finite-rank perturbations of random matrices \cite{Peche06}, asymmetric exclusion processes \cite{BFS09, IS07}, directed polymers in random environments \cite{BCD21, TV20}, and exponential LPP \cite{BP08}.

We give a formal definition of $\mathcal{A}^{a,b,c}$ in Section \ref{Section1.2} below, but mention here that up to a horizontal translation the law of $\mathcal{A}^{a,b,c}$ depends on two sequences of non-negative reals $\{a^+_i\}_{i \geq 1}$, $\{b^+_i\}_{i \geq 1}$ and a real parameter $c^+ \in \mathbb{R}$, which satisfy
\begin{equation}\label{Eq.IntroParameters}
a^+_1 \geq a^+_2 \geq \cdots \geq 0, \hspace{2mm} b^+_1 \geq b^+_2 \geq \cdots \geq 0, \mbox{ and } \sum_{i \geq 1}(a^+_i + b^+_i) < \infty.
\end{equation}
Recently, in \cite{D25}, the second author established several structural properties of the Airy wanderer line ensembles; some of these are summarized below.

\vspace{2mm}

{\bf \raggedleft I. Symmetries.} If one vertically translates $\mathcal{A}^{a,b,c}$ by $v$, then the resulting ensemble has the same law as $\mathcal{A}^{a,b,c + v}$. In addition, if one reflects $\mathcal{A}^{a,b,c}$ across the origin, the resulting ensemble has law $\mathcal{A}^{b,a,c}$. See Proposition \ref{Prop.BasicProperties} for the precise statements.
\vspace{2mm}

{\bf \raggedleft II. Continuity.} The laws of $\mathcal{A}^{a,b,c}$ vary continuously with $(a,b,c)$. See \cite[Proposition 1.12]{D25}.
\vspace{2mm}

{\bf \raggedleft III. Monotonicity.} The Airy wanderer line ensembles admit {\em multiple} monotone couplings in their parameters. See Proposition \ref{Prop.MonCoupling}.
\vspace{2mm}

{\bf \raggedleft IV. Asymptotic slopes.} For any sequence $t_N \uparrow \infty$ and $j \in \mathbb{N}$, we have
\begin{equation}\label{Eq.IntroSlopes}
t_N^{-1} \cdot (\mathcal{A}^{a,b,c}_j(t_N) - t_N^2) \Rightarrow -2/a_j^+ \mbox{ if } a_j^+ > 0, \mbox{ and } t_N^{-1} \cdot (\mathcal{A}^{a,b,c}_j(-t_N) - t_N^2) \Rightarrow -2/b_j^+ \mbox{ if } b_j^+ > 0.
\end{equation}
In addition, we have that $\{\mathcal{A}^{a,b,c}_j(t_N) \}_{N \geq 1}$ is tight when $a_j^+ = 0$ and $\{\mathcal{A}^{a,b,c}_j(-t_N) \}_{N \geq 1}$ is tight when $b_j^+ = 0$. See Proposition \ref{Prop.Slopes}.
\vspace{2mm}

{\bf \raggedleft V. Extremality.} The Airy wanderer line ensembles are extreme points in the space of all line ensembles on $\mathbb{R}$ that satisfy the Brownian Gibbs property. See \cite[Theorem 1.19]{D25} for the precise statement.\\

Focusing on property IV above, we see that the physical interpretation of the parameters $\{a^+_i\}_{i \geq 1}$ and $\{b^+_i\}_{i \geq 1}$ is that they play the role of asymptotic slopes for the ensemble $\mathcal{A}^{a,b,c}$ near $+\infty$ and $-\infty$, respectively. In this sense, (\ref{Eq.IntroSlopes}) should be viewed as a weak law of large numbers for the curves $\mathcal{A}^{a,b,c}_j(t)$ as $t \rightarrow \pm \infty$. The main goal of this paper is to describe the fluctuations about this leading-order behavior as $t \rightarrow \pm \infty$, and to establish {\em functional limit theorems} for the curves of the ensemble. While precise statements of our principal results are deferred to Section \ref{Section1.4}, we pause here to provide an informal overview, focusing on the case $t \rightarrow \infty$. 

From (\ref{Eq.IntroParameters}), we see that there are (unique) positive reals $v_1^a > v_2^a > \cdots $, as well as positive integers $\mathsf{m}_i^a$, such that 
$$a_1^+ = a_2^+ = \cdots = a^+_{\mathsf{m}_1^a} = v_1^a, \hspace{2mm} a_{\mathsf{m}_1^a +1}^+ = a_{\mathsf{m}_1^a + 2}^+ = \cdots = a_{\mathsf{m}_1^a + \mathsf{m}_2^a}^+ = v_2^a, \hspace{2mm} \dots,$$
In words, $v_i^a$ are the distinct positive terms that appear in the sequence $\{a^+_i\}_{i \geq 1}$, and $\mathsf{m}_i^a$ are their multiplicities. The sequences $v_i^a, \mathsf{m}_i^a$ could be finite or even empty when $a_j^+ = 0$ for some $j \geq 1$. If we set $\mathsf{M}_k^a = \mathsf{m}_1^a + \cdots + \mathsf{m}_k^a$, then we see from (\ref{Eq.IntroSlopes}) that the curves $\{\mathcal{A}_i^{a,b,c}(t) - t^2\}_{i \geq 1}$ split into groups $\{\mathcal{A}_i^{a,b,c}(t) - t^2\}_{i = \mathsf{M}_{k-1}^a+1}^{\mathsf{M}_k^a}$, where the $k$-th group consists of $\mathsf{m}_k^a$ curves that have the same asymptotic slope $-2/v_k^a$, see the left side of Figure \ref{Fig.Simulations}. In Theorem \ref{Thm.ConvDBM}(a), we show that as $t \rightarrow \infty$ the $k$-th group $\{\mathcal{A}_i^{a,b,c}(t) - t^2\}_{i = \mathsf{M}_{k-1}^a+1}^{\mathsf{M}_k^a}$ (under appropriate scaling) converges to a Dyson Brownian motion with $\mathsf{m}_k^a$ curves started from zero. In view of the reflection symmetry in property I above, an analogous statement holds near negative infinity. In particular, if we define $v_i^b, \mathsf{m}_i^b, \mathsf{M}_k^b$ analogously for the sequence $\{b^+_i\}_{i \geq 1}$, then as $t \rightarrow \infty$, the group $\{\mathcal{A}_i^{a,b,c}(-t) - t^2\}_{i = \mathsf{M}_{k-1}^b+1}^{\mathsf{M}_k^b}$ converges to a Dyson Brownian motion with $\mathsf{m}_k^b$ curves started from zero. The precise result is stated in Theorem \ref{Thm.ConvDBM}(b).

\begin{figure}[ht]
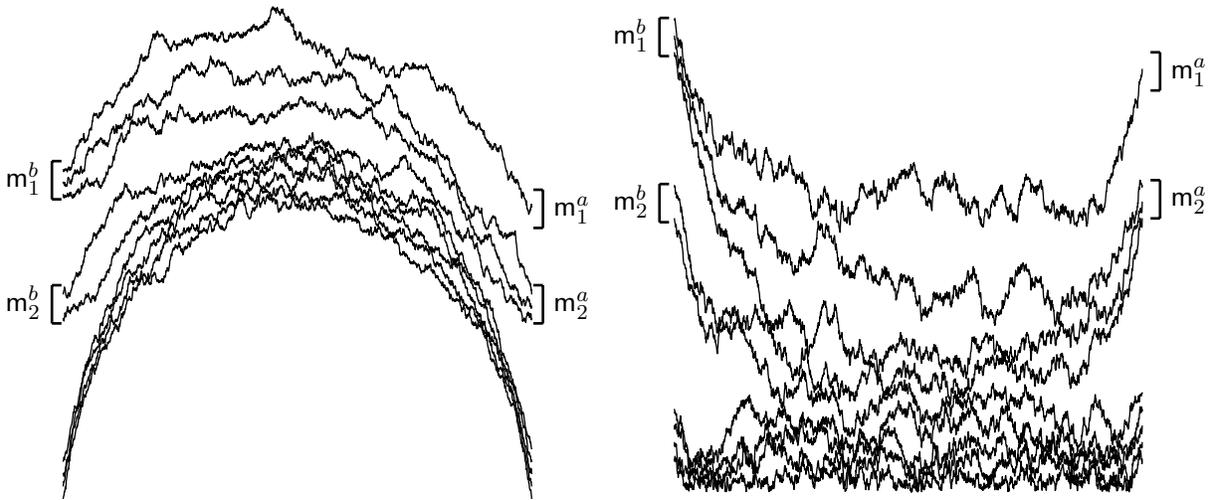

\centering
\begin{tikzpicture}[x=1in,y=1in] 

\path[use as bounding box] (-2in,0.5) rectangle (2in,3.5in);

  \node[anchor=south west, inner sep=0, scale=0.5] at (-3.4,0) {\input{S1.pgf}};
  \node[anchor=south west, inner sep=0, scale=0.5] at (-0.2,0) {\input{S2.pgf}};

   \draw[line width=1pt]
    (-2.9,2.25) -- (-2.85,2.25)
    (-2.9,2.25) -- (-2.9,2.05)
    (-2.9,2.05) -- (-2.85,2.05);
    \draw (-3.05,2.3) node[below = 0pt]{$\mathsf{m}_1^b$};

   \draw[line width=1pt]
    (-2.9,1.6) -- (-2.85,1.6)
    (-2.9,1.6) -- (-2.9,1.4)
    (-2.9,1.4) -- (-2.85,1.4);
    \draw (-3.05,1.65) node[below = 0pt]{$\mathsf{m}_2^b$};

   \draw[line width=1pt]
    (0.28,3) -- (0.33,3)
    (0.28,3) -- (0.28,2.8)
    (0.28,2.8) -- (0.33,2.8);
    \draw (0.13,3.05) node[below = 0pt]{$\mathsf{m}_1^b$};

   \draw[line width=1pt]
    (0.28,2.13) -- (0.33,2.13)
    (0.28,2.13) -- (0.28,1.93)
    (0.28,1.93) -- (0.33,1.93);
    \draw (0.13,2.18) node[below = 0pt]{$\mathsf{m}_2^b$};

   \draw[line width=1pt]
    (-0.33,2.1) -- (-0.38,2.1)
    (-0.33,2.1) -- (-0.33,1.9)
    (-0.33,1.9) -- (-0.38,1.9);
    \draw (-0.18,2.12) node[below = 0pt]{$\mathsf{m}_1^a$};

   \draw[line width=1pt]
    (-0.33,1.6) -- (-0.38,1.6)
    (-0.33,1.6) -- (-0.33,1.4)
    (-0.33,1.4) -- (-0.38,1.4);
    \draw (-0.18,1.62) node[below = 0pt]{$\mathsf{m}_2^a$};

   \draw[line width=1pt]
    (2.9,2.82) -- (2.85,2.82)
    (2.9,2.82) -- (2.9,2.62)
    (2.9,2.62) -- (2.85,2.62);
    \draw (3.05,2.84) node[below = 0pt]{$\mathsf{m}_1^a$};

   \draw[line width=1pt]
    (2.9,2.15) -- (2.85,2.15)
    (2.9,2.15) -- (2.9,1.95)
    (2.9,1.95) -- (2.85,1.95);
    \draw (3.05,2.17) node[below = 0pt]{$\mathsf{m}_2^a$};

\end{tikzpicture}
\caption{The left side depicts the curves $\{\mathcal{A}^{a,b,c}_i(t) - t^2\}_{i \geq 1}$ and the right depicts the curves $\{\mathcal{A}^{a,b,c}_i(t)\}_{i \geq 1}$ when $\mathsf{m}_1^a = 1$, $\mathsf{m}_2^a = 3$, $\mathsf{m}_1^b = 3$, $\mathsf{m}_2^b = 2$. Here, $J_a = \mathsf{m}_1^a + \mathsf{m}_2^a = 4$. }
\label{Fig.Simulations}
\end{figure}

If the number of parameters $a_i^+$ that are positive is finite, i.e. there exists $J_a \in \mathbb{Z}_{\geq 0}$ such that $a_{J_a+1}^+ = a_{J_a+2}^+ = \cdots = 0$, then property IV above suggests that the curves $\mathcal{A}_{J_a+i}^{a,b,c}(t)$ remain flat as $t \rightarrow \infty$, while if $J_a \geq 1$ and $a_{J_a}^+ > 0$, the curve $\mathcal{A}_{J_a}^{a,b,c}(t)$ follows a parabola, see the right side of Figure \ref{Fig.Simulations}. In particular, for large times the curves $\{\mathcal{A}_{i}^{a,b,c}(t)\}_{i = 1}^{J_a}$ {\em separate} from $\{\mathcal{A}_{J_a + i}^{a,b,c}(t)\}_{i \geq 1}$ and the latter converge to the usual Airy line ensemble. The precise statement appears in Theorem \ref{Thm.ConvAiry}(a) when $t \rightarrow \infty$ and Theorem \ref{Thm.ConvAiry}(b) when $t \rightarrow -\infty$. We mention that a similar curve separation phenomenon was recently established in the context of the half-space geometric LPP, see \cite{DZ25}. 

Although this paper is devoted to the Airy wanderer line ensembles, we expect analogous convergence results to hold for a broad class of inhomogeneous models in the KPZ universality class including the aforementioned geometric and exponential LPP, directed polymers, and spiked random matrices. Our proofs of Theorems \ref{Thm.ConvDBM} and \ref{Thm.ConvAiry} rely on the determinantal point process structure of the ensembles $\mathcal{A}^{a,b,c}$ together with their local resampling invariance, commonly referred to as the {\em Brownian Gibbs property}. Geometric and exponential LPP, as well as several families of random matrix models, are known to be determinantal and to satisfy discrete analogues of the Brownian Gibbs property, and we believe our arguments can be readily extended to these settings. For non-determinantal models, like the log-gamma polymer \cite{BCD21, TV20}, different techniques are required; nevertheless, we expect that several components of our approach --- particularly those concerning Gibbsian line ensembles and curve separation --- will remain relevant in that context. \\

The rest of the introduction is structured as follows. In Section \ref{Section1.2} we formally define the Airy wanderer line ensembles, and recall some results from \cite{dimitrov2024airy}. In Section \ref{Section1.3} we summarize the structural properties from \cite{D25} that we require for our arguments later in the paper. The main results of the paper are presented in Section \ref{Section1.4} with the convergence to Dyson Brownian motion established in Theorem \ref{Thm.ConvDBM} and the convergence to the Airy line ensemble in Theorem \ref{Thm.ConvAiry}. Section \ref{Section1.5} contains an outline of the paper, and discusses some of the key ideas behind our arguments.

%
%
\subsection{The Airy wanderer line ensembles}\label{Section1.2} The goal of this section is to give a formal definition of the Airy wanderer line ensembles constructed in \cite{CorHamA,dimitrov2024airy}. Our exposition here closely follows that of \cite[Section 1.2]{dimitrov2024airy}. We begin by fixing the parameters of the model and some notation.
\begin{definition}\label{Def.DLP} 
We assume that we are given four sequences of non-negative real numbers $\{a_i^+\}_{ i \geq 1}$, $\{a_i^-\}_{ i \geq 1}$, $\{b_i^+\}_{ i \geq 1}$, $\{b_i^-\}_{ i \geq 1}$ such that 
\begin{equation}\label{Eq.DLP}
\sum_{i = 1}^{\infty} (a_i^+ + a_i^- + b_i^+ + b_i^-) < \infty \mbox{ and } a_{i}^{\pm} \geq a_{i+1}^{\pm},  b_{i}^{\pm} \geq b_{i+1}^{\pm} \mbox{ for all } i \geq 1,
\end{equation}
as well as two real parameters $c^+, c^-$. We let $J_a^{\pm} = \inf \{ k \geq 1: a_{k}^{\pm} = 0\} - 1$ and $J_b^{\pm} = \inf \{ k \geq 1: b_{k}^{\pm} = 0\} - 1$. In words, $J_a^{\pm}$ is the largest index $k$ such that $a_{k}^{\pm} > 0$, with the convention that $J_a^{\pm} = 0$ if all $a_k^{\pm} = 0$ and $J_a^{\pm} = \infty$ if all $a_k^{\pm} > 0$, and analogously for $J_b^{\pm}$. For future reference, we denote the set of parameters satisfying the above conditions such that $c^- = 0$ and $J_a^- + J_b^- < \infty$ by $\parF$, and the subset of $\parF$ such that $c^+ = J_a^- = J_b^- = 0$ by $\parP$.

Lastly, we define 
$$\underline{a} = \begin{cases} 0  &\hspace{-3.5mm}  \mbox{ if } a_1^- + b_1^- > 0 \mbox{ or } c^- \neq 0, \\   \infty & \hspace{-3.5mm} \mbox{ if }   a_1^- + b_1^- =  c^- = 0\mbox{ and }  a_1^+ = 0, \\ 1/a_1^+ & \hspace{-3.5mm}  \mbox{ if }a_1^- + b_1^- = c^- =  0 \mbox{ and } a_1^+ > 0, \end{cases} \hspace{1mm} \mbox{ and }\hspace{1mm} \underline{b} = \begin{cases} 0 &\hspace{-3.5mm}  \mbox{ if }  a_1^- + b_1^- > 0 \mbox{ or } c^- \neq 0, \\  -\infty & \hspace{-3.5mm} \mbox{ if } a_1^- + b_1^- = c^- =  0 \mbox{ and } b_1^+ = 0, \\ -1/b_1^+ & \hspace{-3.5mm} \mbox{ if } a_1^- + b_1^- = c^- =  0 \mbox{ and } b_1^+ > 0. \end{cases}$$
Observe that $\underline{a} \in [0, \infty]$ and $\underline{b} \in [- \infty, 0]$. 
\end{definition}

For $z \in \mathbb{C} \setminus \{0\} $ we define the function
\begin{equation}\label{Eq.DefPhi}
\Phi_{a,b,c}(z) = e^{c^+z + c^-/ z} \cdot \prod_{i = 1}^{\infty} \frac{(1 + b_i^+ z) (1 + b_i^- /z)}{(1 - a_i^+ z) ( 1 - a_i^{-}/z)}.
\end{equation}
From (\ref{Eq.DLP}) and \cite[Chapter 5, Proposition 3.2]{Stein}, we have that the above defines a meromorphic function on $\mathbb{C} \setminus \{0\}$ whose zeros are at $\{-(b_i^+)^{-1}\}_{i =1}^{J_b^+}$ and $\{- b_i^-\}_{i =1}^{J_b^-}$, while its poles are at $\{(a_i^+)^{-1}\}_{i =1}^{J_a^+}$ and $\{ a_i^-\}_{i =1}^{J_a^-}$. We also observe that $\Phi_{a,b,c}(z) $ is analytic in $\mathbb{C} \setminus [\underline{a}, \infty)$, and its inverse is analytic in $\mathbb{C} \setminus (-\infty, \underline{b}]$, where $\underline{a}, \underline{b}$ are as in Definition \ref{Def.DLP}.

The following definitions present the Airy wanderer kernel, introduced in \cite{BP08}, starting with the contours that appear in it. 
\begin{definition}\label{Def.Contours}  Fix $a \in \mathbb{R}$. We let $\Gamma_a^+$ denote the union of the contours $\{a + y e^{\pi \im /4}\}_{y \in [0,\infty)}$ and $\{a + y e^{-\pi \im/4}\}_{y \in [0,\infty)}$, and $\Gamma_a^-$ the union of the contours $\{a + y e^{ 3\pi \im /4}\}_{y \in [0,\infty)}$ and $\{a + y e^{-3 \pi \im/4}\}_{y \in [0,\infty)}$. Both contours are oriented in the direction of increasing imaginary part.
\end{definition}
\begin{definition}\label{Def.3BPKernelDef} Assume the same notation as in Definition \ref{Def.DLP}. For $t_1, t_2,x_1,x_2 \in \mathbb{R}$ we define 
\begin{equation}\label{Eq.3BPKer}
\begin{split}
&K_{a,b,c} (t_1, x_1; t_2, x_2) = K^1_{a,b,c} (t_1, x_1; t_2, x_2) +  K^2_{a,b,c} (t_1, x_1; t_2, x_2) +  K^3_{a,b,c} (t_1, x_1; t_2, x_2), \mbox{ with } \\
& K^1_{a,b,c} (t_1, x_1; t_2, x_2) = \frac{1}{2\pi \im} \int_{\gamma}dw \cdot  e^{(t_2 - t_1)w^2 + (t_1^2 - t_2^2) w + w (x_2-x_1) + x_1 t_1 - x_2 t_2 - t_1^3/3 + t_2^3/3}    \\
&K^2_{a,b,c} (t_1, x_1; t_2, x_2)  = -  \frac{{\bf 1}\{ t_2 > t_1\} }{\sqrt{4\pi (t_2 - t_1)}} \cdot e^{ - \frac{(x_2 - x_1)^2}{4(t_2 - t_1)} - \frac{(t_2 - t_1)(x_2 + x_1)}{2} + \frac{(t_2 - t_1)^3}{12} }; \\
& K^3_{a,b,c} (t_1, x_1; t_2, x_2) = \frac{1}{(2\pi \im)^2} \int_{\Gamma_{\alpha }^+} d z \int_{\Gamma_{\beta}^-} dw \frac{e^{z^3/3 -x_1z - w^3/3 + x_2w}}{z + t_1 - w - t_2} \cdot \frac{\Phi_{a,b,c}(z + t_1) }{\Phi_{a,b,c}(w + t_2)}.
\end{split}
\end{equation}
In (\ref{Eq.3BPKer}) $\alpha, \beta  \in \mathbb{R}$ are such that $\alpha + t_1 < \underline{a}$ and $\beta + t_2 > \underline{b}$, the function $\Phi_{a,b,c}$ is as in (\ref{Eq.DefPhi}) and the contours of integration in $K^3_{a,b,c}$ are as in Definition \ref{Def.Contours}. If $\Gamma^+_{\alpha + t_1} (=t_1 + \Gamma^+_{\alpha})$ and $\Gamma^-_{\beta + t_2} (= t_2 + \Gamma^-_{\beta} )$ have zero or one intersection points, we take $\gamma = \emptyset$ and then $K^1_{a,b,c} \equiv 0$. Otherwise, $\Gamma^+_{\alpha + t_1} $ and $\Gamma^-_{\beta + t_2}$ have exactly two intersection points, which are complex conjugates, and $\gamma$ is the straight vertical segment that connects them with the orientation of increasing imaginary part. See Figure \ref{Fig.Contours}.
\end{definition}
\begin{figure}[h]
    \centering
     \begin{tikzpicture}[scale=2.7]

        \def\tra{3} 
        \draw[->, thick, gray] (-1.2,0)--(1.2,0) node[right]{$\Real$};
        \draw[->, thick, gray] (0,-1.2)--(0,1.2) node[above]{$\Imag$};

        \draw[-,thick][black] (0.6,0) -- (0.2,-0.4);
        \draw[->,thick][black] (-0.4,-1) -- (0.2,-0.4);
        \draw[black, fill = black] (0.6,0) circle (0.02);
        \draw (0.6,-0.3) node{$\beta + t_2$};
        \draw[->,very thin][black] (0.6,-0.2) -- (0.6, -0.05);
        \draw[->,thick][black] (0.6,0) -- (0.2,0.4);
        \draw[-,thick][black]  (-0.4,1) -- (0.2,0.4);       

        \draw[-,thick][black] (-0.75, 0) -- (-0.25,-0.5);
        \draw[->,thick][black] (0.25, -1) -- (-0.25, -0.5);
        \draw[black, fill = black] (-0.75,0) circle (0.02);
        \draw (-0.75,-0.275) node{$\alpha + t_1$};
        \draw[->,very thin][black] (-0.75,-0.2) -- (-0.75, -0.05);
        \draw[->,thick][black] (-0.75,0) -- (-0.25,0.5);
        \draw[-,thick][black]  (-0.25,0.5) -- (0.25,1);

        \draw[->,thick][black] (-0.075, -0.675) -- (-0.075,0.2);
        \draw[-,thick][black] (-0.075, 0.2) -- (-0.075,0.675);

        \draw[black, fill = black] (-0.075,0.675) circle (0.02);
        \draw[black, fill = black] (-0.075,-0.675) circle (0.02);
        \draw (-0.075,0.8) node{$u_+$};
        \draw (-0.075,-0.825) node{$u_-$};

        \draw (0.35,0.825) node{$\Gamma^+_{\alpha + t_1}$};
        \draw (-0.4,0.825) node{$\Gamma^-_{\beta + t_2}$};
        \draw (-0.15,0.225) node{$\gamma$};

    \end{tikzpicture} 
    \caption{The figure depicts the contours $\Gamma_{\alpha + t_1}^+, \Gamma_{\beta + t_2}^-$ when they have two intersection points, denoted by $u_-$ and $u_+$. The contour $\gamma$ is the segment from $u_-$ to $u_+$.}
    \label{Fig.Contours}
\end{figure}
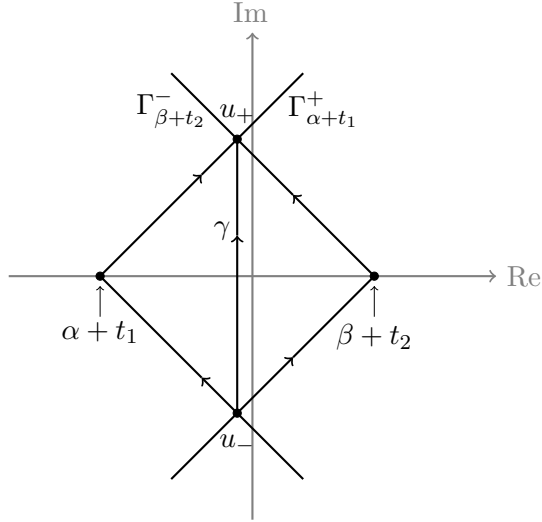

The following definition introduces certain measures that appear throughout the paper.
\begin{definition}\label{Def.Measures} For a finite set $\mathsf{S} = \{s_1, \dots, s_m\} \subset \mathbb{R}$, we let $\mu_{\mathsf{S}}$ denote the counting measure on $\mathbb{R}$, defined by $\mu_{\mathsf{S}}(A) = |A \cap \mathsf{S}|$. We also let $\mathrm{Leb}$ be the usual Lebesgue measure on $\mathbb{R}$ and $\mu_{\mathsf{S}} \times \mathrm{Leb}$ the product measure on $\mathbb{R}^2$.
\end{definition}

With the above notation in place, we can introduce the Airy wanderer line ensembles. We do this in the following statement, which follows from \cite[Theorems 1.8 and 1.10]{dimitrov2024airy}. 
\begin{proposition}\label{Prop.AWLE} Assume the same notation as in Definition \ref{Def.DLP} and fix $(a,b,c) \in \parF$. Then, there exists a unique line ensemble $\mathcal{A}^{a,b,c} = \{\mathcal{A}_i^{a,b,c}\}_{i \geq 1}$ on $\mathbb{R}$ such that the following all hold. Firstly, the ensemble is non-intersecting, meaning that almost surely
\begin{equation}\label{Eq.OrdAWLE}
\mathcal{A}^{a,b,c}_i(t) > \mathcal{A}^{a,b,c}_{i+1}(t) \mbox{ for all } i \in \mathbb{N}, t \in \mathbb{R}.
\end{equation} 
For each $m \in \mathbb{N}$, and $s_1, \dots, s_m \in \mathbb{R}$ with $s_1 < s_2 < \cdots < s_m$ we have that the random measure
\begin{equation}\label{T2E1}
M(\omega, A) = \sum_{i \geq 1} \sum_{j = 1}^m {\bf 1}\left\{ (s_j, \mathcal{A}^{a,b,c}_i(s_j, \omega)) \in A \right\}
\end{equation}
is a determinantal point process on $\mathbb{R}^2$, with correlation kernel $K_{a,b,c}$ as in (\ref{Eq.3BPKer}), and reference measure $\mu_{\mathsf{S}} \times \mathrm{Leb}$ as in Definition \ref{Def.Measures}. In addition, if we define the line ensemble $\mathcal{L}^{a,b,c}$ via
\begin{equation}\label{S1FDE}
\left(\sqrt{2} \cdot \mathcal{L}_i^{a, b, c}(t) + t^2: i \geq 1, t \in \mathbb{R} \right)  = \left(\mathcal{A}^{a,b,c}_i(t): i \geq 1 , t \in \mathbb{R} \right),
\end{equation}
then $\mathcal{L}^{a,b,c}$ satisfies the Brownian Gibbs property from \cite[Definition 2.2]{CorHamA}, see also Definition \ref{Def.BGPVanilla}.
\end{proposition}
\begin{remark}\label{Rem.AWLE} We mention that \cite[Theorems 1.8 and 1.10]{dimitrov2024airy} only define $\mathcal{A}^{a,b,c}$ when $c^+ \geq 0$. For general $c^+ \in \mathbb{R}$, one can construct such an ensemble by starting with $\mathcal{A}^{a,b,0}$ (i.e. an ensemble with the same $a_i^{\pm}, b_i^{\pm}$ parameters but with $c^+ = 0$), and then set $\mathcal{A}_i^{a,b,c}(t):= \mathcal{A}_i^{a,b,0}(t) + c^+$. The fact that the latter satisfies the conditions of the proposition follows by a simple change of variables, see \cite[Proposition 1.10(a)]{D25} and its proof.
\end{remark}
\begin{remark}\label{Rem.AWLE2} When all parameters $a_i^{\pm}, b_i^{\pm}$ and $c^{\pm}$ are equal to zero, the ensemble $\mathcal{A}^{0,0,0}$ is just the usual Airy line ensemble from \cite{CorHamA}. 
\end{remark}

%
%
\subsection{Properties of the Airy wanderer line ensembles}\label{Section1.3} In this section, we summarize various properties of the Airy wanderer line ensembles, which we require later in the text. All of these can be found in \cite[Section 1]{D25}.

The following statement describes the behavior of $\mathcal{A}^{a,b,c}$ under horizontal and vertical shifts, as well as reflections across the origin.
\begin{proposition}\label{Prop.BasicProperties}\cite[Proposition 1.10]{D25} Assume the same notation as in Definition \ref{Def.DLP}, fix $(a,b,c) \in \parF$ and let $\mathcal{A}^{a,b,c}$ be as in Proposition \ref{Prop.AWLE}. Then, the following statements all hold.
\begin{enumerate}
    \item[(a)] (Translation) Suppose that $v \in \mathbb{R}$. Then, the line ensemble 
    $$\bar{\mathcal{A}} :=\left(\mathcal{A}^{a,b,c}_i(t) + v : i \geq 1, t \in \mathbb{R}\right).$$ 
    has the same law as $\mathcal{A}^{\bar{a},\bar{b},\bar{c}}$, where $(\bar{a},\bar{b},\bar{c}) \in \parF$ satisfy $\bar{a}^{\pm}_i = a_i^{\pm}$, $\bar{b}^{\pm}_i = b^{\pm}_i$ for $i \geq 1$ and $\bar{c}^+ = c^+ + v$.
    \item[(b)] (Reflection) The line ensemble $\cev{\mathcal{A}} :=\left(\mathcal{A}^{a,b,c}_i(-t) : i \geq 1, t \in \mathbb{R}\right)$ has the same law as $\mathcal{A}^{b,a,c}$, i.e. the Airy wanderer line ensemble whose ``$a$'' and ``$b$'' parameters have been swapped, but with the same parameter $c^+$. 
    \item[(c)] There exist $(\tilde{a}, \tilde{b}, \tilde{c}) \in \parP$ and $\Delta \in \mathbb{R}$, such that we have the following equality in law
    $$\left( \mathcal{A}^{\tilde{a},\tilde{b},\tilde{c}}_i(t-\Delta) + c^+: i \geq 1, t \in \mathbb{R} \right) =  \left( \mathcal{A}^{a,b,c}_i(t): i \geq 1, t \in \mathbb{R} \right).$$
\end{enumerate}
\end{proposition}
\begin{remark} Part (c) of the proposition states that we can realize any $\mathcal{A}^{a,b,c}$ from Proposition \ref{Prop.AWLE} as an appropriately translated Airy wanderer line ensemble with parameters in $\parP$. In order to simplify our exposition later in the paper, we will state subsequent results for $\mathcal{A}^{a,b,c}$ under the additional assumption that $(a,b,c) \in \parP$. However, in view of Proposition \ref{Prop.BasicProperties}(c), one can transfer properties of such ensembles to all Airy wanderer line ensembles by merely translating them. We mention that the parameters $(\tilde{a}, \tilde{b}, \tilde{c})$, and $\Delta$ depend in a non-trivial way on $(a,b,c)$, and we refer to the proof of \cite[Proposition 1.10]{D25} for the precise relationship.
\end{remark}

The next result shows that there are {\em multiple} ways (indexed by $A, B \in \mathbb{Z}_{\geq 0}$) to couple two Airy wanderer line ensembles with different sets of parameters $(a,b,c)$ and $(\tilde{a}, \tilde{b}, \tilde{c})$, that ensure that their appropriately reindexed curves are stochastically ordered.
\begin{proposition}\label{Prop.MonCoupling}\cite[Theorem 1.13]{D25} Assume the same notation as in Definition \ref{Def.DLP} and Proposition \ref{Prop.AWLE}. Fix $A, B \in \mathbb{Z}_{\geq 0}$ and two sets of parameters $(a,b,c), (\tilde{a}, \tilde{b}, \tilde{c}) \in \parP$, such that $a_{i + A}^+ \leq \tilde{a}_i^+$ and $b_{i+B}^+ \leq \tilde{b}_i^+$ for $i \geq 1$. Then, we can couple $\mathcal{A}^{a,b,c}$ and $\mathcal{A}^{\tilde{a},\tilde{b},\tilde{c}}$ on the same probability space, so that almost surely
\begin{equation}\label{Eq.MonCoupling1}
\mathcal{A}^{a,b,c}_{k + \max(A,B)}(t) \leq \mathcal{A}_k^{\tilde{a},\tilde{b},\tilde{c}}(t) \mbox{ for all } k \geq 1, t \in \mathbb{R}.
\end{equation}
\end{proposition}

The next result describes the global behavior of the points $\{\mathcal{A}_k^{a,b,c}(t)\}_{k \geq 1}$ as $t \rightarrow \pm \infty$.
\begin{proposition}\label{Prop.Slopes}\cite[Theorem 1.15]{D25} Assume the same notation as in Definition \ref{Def.DLP} and Proposition \ref{Prop.AWLE}. Fix $(a,b,c) \in \parP$ and any sequence $t_N > 0$ with $t_N \uparrow \infty$. Then, the following statements all hold.
\begin{enumerate}
\item[(a)] If $k \in \{1, \dots, J_a^+\}$, then $t_N^{-1} \cdot (\mathcal{A}^{a,b,c}_k(t_N) - t_N^2) \Rightarrow -2/a_k^+$.
\item[(b)] If $k \in \{1, \dots, J_b^+\}$, then $t_N^{-1} \cdot (\mathcal{A}^{a,b,c}_k(-t_N) - t_N^2) \Rightarrow -2/b_k^+$.
\item[(c)] If $J_a^+ < k <\infty$, then $\{\mathcal{A}^{a,b,c}_k(t_N) \}_{N \geq 1}$ is tight.
\item[(d)] If $J_b^+ < k <\infty$, then $\{\mathcal{A}^{a,b,c}_k(-t_N) \}_{N \geq 1}$ is tight.
\end{enumerate}
\end{proposition}

%
%
\subsection{Main results}\label{Section1.4} In this section, we present the main results of the paper. We begin by summarizing some useful notation in the following definition.

\begin{definition}\label{Def.ParMultiplicities} Assume the same notation as in Definition \ref{Def.DLP}, and fix $(a,b,c) \in \parP$. Since $a_i^{-} = b_i^{-} = c^+ = c^- = 0$ here, we drop the ``plus'' signs from the notation and simply write $\{a_i\}_{\geq 1}, \{b_i\}_{i \geq 1}, J_a, J_b$ in place of $\{a^+_i\}_{\geq 1}, \{b^+_i\}_{i \geq 1}, J^+_a, J^+_b$. For such a choice of parameters the function $\Phi_{a,b,c}$ from (\ref{Eq.DefPhi}) becomes
\begin{equation}\label{Eq.DefPhi2}
\Phi_{a,b,c}(z) =  \prod_{i = 1}^{\infty} \frac{1 + b_i z}{1 - a_i z}.
\end{equation}
For the above choice of parameters, we let $\mathsf{V}_a = \{x \in \mathbb{R}: x = a_i \mbox{ for some } i \le J_a\}$, and $\mathsf{V}_b = \{x \in \mathbb{R}: x = b_i \mbox{ for some } i \le J_b\}$. If $J_a > 0$, we can order the elements in $\mathsf{V}_a$ by $v_1^a > v_2^a > \cdots$, and there are unique integers $\mathsf{m}^a_1, \mathsf{m}^a_2, \dots \in \mathbb{N}$, such that 
$$a_{1} = \cdots = a_{\mathsf{m}^a_1} = v_1^a, \hspace{2mm} a_{\mathsf{m}_1+ 1} = \cdots = a_{\mathsf{m}^a_1 + \mathsf{m}^a_2} = v_2^a, \mbox{ and so on.}$$
In other words, $\mathsf{m}_i^a$ is the number of times $v_i^a$ appears in the sequence $\{a_i\}_{i \geq 1}$. We similarly define for $J_b > 0$ the numbers $v_i^b$, and $\mathsf{m}_i^b$. 

Finally, we set $\mathsf{M}_0^a = \mathsf{M}_0^b = 0$ and for $i \geq 1$ put $\mathsf{M}_i^a = \mathsf{m}_1^a + \cdots + \mathsf{m}_i^a$, $\mathsf{M}_i^b = \mathsf{m}_1^b + \cdots + \mathsf{m}_i^b$.
\end{definition}

Before we state our first main result, we need to introduce Dyson Brownian motion, which is done in the following definition.

\begin{definition}\label{Def.DBM}
Fix $ n \in \mathbb{N}$, and define the open and closed {\em Weyl chambers}
\begin{equation}\label{Eq.WeylChamber}
\weyl_n = \{(x_1, \dots, x_n) \in \mathbb{R}^n: x_1 > \cdots > x_n\}, \hspace{2mm}  \weylc_n = \{(x_1, \dots, x_n) \in \mathbb{R}^n: x_1 \geq \cdots \geq x_n\}.
\end{equation}
For $\lambda^n(0) = (\lambda_1^n(0), \dots, \lambda_n^n(0)) \in \weylc_n$, we define $\lambda(t) = (\lambda_1^n(t), \dots, \lambda_n^n(t)) \in C([0,\infty), \weylc_n)$ to be the unique strong solution to the stochastic differential system
\begin{equation}\label{Eq.DysonSDE}
d\lambda_i^n(t) = dB_i(t) + \sum_{j = 1, j \neq i}^n \frac{dt}{\lambda_i^n(t) - \lambda_j^n(t)},
\end{equation}  
where $B_1, \dots, B_n$ are i.i.d. standard Brownian motions. The fact that the above system has a unique strong solution is established in \cite[Proposition 4.3.5]{AGZ}, and it is called {\em Dyson Brownian motion} (with $\beta = 2$), {\em started from} $\lambda^n(0)$.
\end{definition}

With the above notation in place, we are ready to state our first main result.
\begin{theorem}\label{Thm.ConvDBM} Assume the same notation as in Definitions \ref{Def.DLP}, \ref{Def.ParMultiplicities}, \ref{Def.DBM}, and Proposition \ref{Prop.AWLE}. The following statements hold for any sequence $T_N > 0$ with $T_N\uparrow \infty$.
\begin{enumerate}
\item[(a)] Fix $k \in \mathbb{N}$, suppose $(a,b,c) \in \parP$ satisfies $|\mathsf{V}_a| \geq k$, and define the scaled processes
\begin{equation}\label{Eq.ScaledSlopeA}
\mathcal{L}^N_i(t) = (2T_N)^{-1/2} \cdot \left( \mathcal{A}^{a,b,c}_{\mathsf{M}_{k-1}^a + i}(tT_N) + 2tT_N/v_k^a - t^2T_N^2 \right),
\end{equation}
for $t > 0$ and $i = 1, \dots, \mathsf{m}_k^a$. Then,
\begin{equation}\label{Eq.ConvDBMA}
\left(\mathcal{L}^N_i(t): t > 0 \mbox{, }i = 1, \dots, \mathsf{m}_k^a \right) \Rightarrow \left(\lambda^{\mathsf{m}_k^a}_i(t): t > 0 \mbox{, }i = 1, \dots, \mathsf{m}_k^a\right).
\end{equation}
\item[(b)] Fix $k \in \mathbb{N}$, suppose $(a,b,c) \in \parP$ satisfies $|\mathsf{V}_b| \geq k$, and define the scaled processes
\begin{equation}\label{Eq.ScaledSlopeB}
\cev{\mathcal{L}}^N_i(t) = (2T_N)^{-1/2} \cdot \left( \mathcal{A}^{a,b,c}_{\mathsf{M}_{k-1}^b + i}(-tT_N) + 2tT_N/v_k^b - t^2T_N^2 \right),
\end{equation}
for $t > 0$ and $i = 1, \dots, \mathsf{m}_k^b$. Then,
\begin{equation}\label{Eq.ConvDBMB}
\left(\cev{\mathcal{L}}^N_i(t): t > 0 \mbox{, }i = 1, \dots, \mathsf{m}_k^b \right) \Rightarrow \left(\lambda^{\mathsf{m}_k^b}_i(t): t > 0 \mbox{, }i = 1, \dots, \mathsf{m}_k^b\right).
\end{equation}
\end{enumerate}
\end{theorem}
\begin{remark}\label{Rem.ConvergenceDBM} The convergence in (\ref{Eq.ConvDBMA}) and (\ref{Eq.ConvDBMB}) is that of random elements in $C(\{1, \dots, m\} \times (0,\infty))$ (with $m = \mathsf{m}_k^a$ or $m = \mathsf{m}_k^b$), where the latter space is endowed with the topology of uniform convergence over compact sets.
\end{remark}

We next turn to the second main result of the paper.
\begin{theorem}\label{Thm.ConvAiry} Assume the same notation as in Definitions \ref{Def.DLP}, \ref{Def.ParMultiplicities}, and Proposition \ref{Prop.AWLE}. The following statements hold for any sequence $T_N > 0$ with $T_N\uparrow \infty$.
\begin{enumerate}
\item[(a)] Suppose $(a,b,c) \in \parP$ satisfies $J_a < \infty$. Define the processes
\begin{equation}\label{Eq.ScaledFlatA}
\mathcal{A}^N_i(t) =  \mathcal{A}^{a,b,c}_{J_a + i}(t + T_N),
\end{equation}
for $t \in \mathbb{R}$ and $i \in \mathbb{N}$. Then,
\begin{equation}\label{Eq.ConvAiryA}
\left(\mathcal{A}^N_i(t): t \in \mathbb{R} \mbox{, }i \in \mathbb{N} \right) \Rightarrow \left(\mathcal{A}^{0,0,0}_i(t): t \in \mathbb{R} \mbox{, }i  \in \mathbb{N}\right).
\end{equation}
\item[(b)] Suppose $(a,b,c) \in \parP$ satisfies $J_b < \infty$. Define the scaled processes
\begin{equation}\label{Eq.ScaledFlatB}
\cev{\mathcal{A}}^N_i(t) =  \mathcal{A}^{a,b,c}_{J_b + i}(-t - T_N),
\end{equation}
for $t \in \mathbb{R}$ and $i \in \mathbb{N}$. Then,
\begin{equation}\label{Eq.ConvAiryB}
\left(\cev{\mathcal{A}}^N_i(t): t \in \mathbb{R} \mbox{, }i \in \mathbb{N} \right) \Rightarrow \left(\mathcal{A}^{0,0,0}_i(t): t \in \mathbb{R} \mbox{, }i  \in \mathbb{N}\right).
\end{equation}
\end{enumerate}
\end{theorem}
\begin{remark}\label{Rem.ConvergenceAiry} The convergence in (\ref{Eq.ConvAiryA}) and (\ref{Eq.ConvAiryB}) is that of random elements in $C(\mathbb{N} \times \mathbb{R})$, where the latter space is endowed with the topology of uniform convergence over compact sets.
\end{remark}

%
%
\subsection{Key ideas and paper outline}\label{Section1.5} In this section, we outline several key ideas underlying the proofs of Theorems \ref{Thm.ConvDBM} and \ref{Thm.ConvAiry}, and indicate where these arguments appear in the paper. Our discussion focuses on the parts (a) of both theorems; the corresponding parts (b) follow directly from (a) by reflecting the ensembles across the origin, using Proposition \ref{Prop.BasicProperties}.

The core of the asymptotic analysis is carried out in Section \ref{Section2}, where we study the limit of the kernel $K_{a,b,c}$ from (\ref{Eq.3BPKer}) under the two scaling regimes in Theorems \ref{Thm.ConvDBM}(a) and \ref{Thm.ConvAiry}(a); see Propositions \ref{Prop.KernelConvTop} and \ref{Prop.KernelConvFlat}. The proofs of both propositions proceed by deforming the contours in the definition of $K_{a,b,c}$ across the poles at $1/a^+_i$, showing that the resulting residue contributions are asymptotically negligible, and establishing convergence of the remaining double-contour integral by an appropriate application of the dominated convergence theorem. We note that Propositions \ref{Prop.KernelConvTop} and \ref{Prop.KernelConvFlat} are proved under special parameter assumptions that ensure all poles encountered during the contour deformation are simple. This simplification makes the asymptotic analysis considerably easier, but comes at the expense of only providing us with convergence statements for ``most'' rather than ``all'' choices of parameters in Theorems \ref{Thm.ConvDBM}(a) and \ref{Thm.ConvAiry}(a). The latter limitation is ultimately harmless, as the monotone couplings in Proposition \ref{Prop.MonCoupling} allow us later to ``fill in the gaps''. 

The limiting kernel in Proposition \ref{Prop.KernelConvFlat} is the extended Airy kernel, which coincides with $K_{0,0,0}$ in (\ref{Eq.3BPKer}), while the limiting kernel in Proposition \ref{Prop.KernelConvTop} is denoted by $\kbmn$ and given in (\ref{Eq.KDBM}). Although Dyson Brownian motion is well-known to possess a determinantal structure and several expressions for its correlation kernel exist in the literature, none of the available formulas directly match $\kbmn$. For this reason, Section \ref{Section2.3} is devoted to explaining how $\kbmn$ arises from Dyson Brownian motion; see in particular Proposition \ref{Prop.DBMDPP}. Specifically, we begin with a correlation kernel obtained by Katori and Tanemura in \cite{KM10}, and perform a sequence of transformations --- most notably a time reversal of the Dyson Brownian motion --- that ultimately lead us to $\kbmn$.\\

In Section 3, we establish the finite-dimensional convergence of the line ensembles appearing in Theorems \ref{Thm.ConvDBM}(a) and \ref{Thm.ConvAiry}(a); see Propositions \ref{Prop.FDSlope} and \ref{Prop.FDFlat}. To prove convergence to Dyson Brownian motion, we apply a finite-dimensional convergence criterion from \cite[Lemma 7.1]{DZ25}, which reduces the problem to verifying the following two statements:
\begin{enumerate}
\item[I.] The sequence of random measures $M^N$, defined by
$$M^N(A) = \sum_{i = 1}^{\mathsf{m}_k^a} \sum_{j = 1}^m {\bf 1}\left\{\left(s_j, \mathcal{L}^N_i(s_j)\right) \in A \right\},$$
for fixed times $0 < s_1 < \cdots < s_m$, converges weakly (as a sequence of random measures).
\item[II.] For each $i$ and $j$, the sequence of random variables $\{\mathcal{L}^N_i(s_j)\}_{N \geq 1}$ is tight.
\end{enumerate}

The measures $M^N$ themselves possess little intrinsic structure, since they consist only of atoms at the points
$$\left(s_j, (2T_N)^{-1/2}\left(\mathcal{A}^{a,b,c}_i(s_jT_N)+\frac{2s_jT_N}{v_k^a}-s_j^2T_N^2\right) \right), \quad \mbox{for } i \in \{\mathsf{M}_{k-1}^a+1, \dots, \mathsf{M}_k^a\}, j \in \{1, \dots, m\},$$ 
corresponding to the $k$-th group of curves into which the ensemble organizes in the long-time limit. As a result, establishing statement I directly is somewhat challenging. 

The key idea is instead to consider the auxiliary measures $M^{v_k^a,T_N}$ formed by 
$$\left(s_j, (2T_N)^{-1/2}\left(\mathcal{A}^{a,b,c}_i(s_jT_N)+\frac{2s_jT_N}{v_k^a}-s_j^2T_N^2\right) \right), \quad \mbox{ for all $i \in \mathbb{N}$, $j \in \{1, \dots, m\}$.}$$ 
The measures $M^{v_k^a,T_N}$ are determinantal and their weak convergence can be established using the kernel asymptotics from Proposition \ref{Prop.KernelConvTop}, together with the convergence criterion for determinantal point processes from \cite[Proposition 2.18]{dimitrov2024airy}. Although $M^{v_k^a,T_N}$ contains more atoms than $M^N$, all additional atoms escape to either $+ \infty$ or $-\infty$ as $N \rightarrow \infty$, in view of the asymptotic slopes for $\mathcal{A}^{a,b,c}$ from Proposition \ref{Prop.Slopes}. Consequently, the difference between $M^{v_k^a,T_N}$ and $M^N$ converges to the zero measure. Since the former converges weakly, the same is true for the latter, which completes the proof of statement I.

To establish statement II, we apply the tightness criterion from \cite[Lemma 7.2]{DZ25}. The details of the argument are given in Step 4 of the proof of Proposition \ref{Prop.FDSlope}; here we only note that the essential additional ingredient is the following tightness-from-above estimate: 
$$\lim_{a \rightarrow \infty}\limsup_{N \rightarrow \infty} \mathbb{P}\left( \mathcal{L}_1^N(s_j) \geq a \right) = 0,$$
which is proved as Proposition \ref{Prop.TightnessSlope} in Section \ref{Section3.1} using ideas similar to those in Proposition \ref{Prop.KernelConvTop}. Once statements I and II are established, the finite-dimensional convergence to Dyson Brownian motion follows for the restricted parameter choices appearing in Propositions \ref{Prop.KernelConvTop} and \ref{Prop.TightnessSlope}. As mentioned earlier, the extension to all parameter values is obtained via the monotone couplings from Proposition \ref{Prop.MonCoupling}; see Step 1 of the proof of Proposition \ref{Prop.FDSlope} for the details.

The proof of finite-dimensional convergence to the Airy line ensemble in Proposition \ref{Prop.FDFlat} proceeds in a fully analogous manner. One key difference is that, in this setting, we are dealing with infinitely many curves rather than a finite group. As a result, we apply the finite-dimensional convergence criterion from \cite[Proposition 2.19]{dimitrov2024airy} in place of \cite[Lemma 7.1]{DZ25}, which again reduces the problem to analogues of statements I and II above. The weak convergence required in statement I is established exactly as before with the kernel convergence from Proposition \ref{Prop.KernelConvTop} replaced by that of Proposition \ref{Prop.KernelConvFlat}. Verifying the tightness condition in statement II is even simpler in this case and relies solely on Proposition \ref{Prop.Slopes}(c).\\

Propositions \ref{Prop.FDSlope} and \ref{Prop.FDFlat} establish finite-dimensional convergence of the line ensembles appearing in Theorems \ref{Thm.ConvDBM}(a) and \ref{Thm.ConvAiry}(a). To upgrade this convergence to uniform convergence on compact sets, it remains to prove tightness of the corresponding sequences of ensembles. Given the finite-dimensional convergence already in hand, the main remaining challenge is to control the moduli of continuity of the ensemble curves along subsequences. 

The key input we use for this control is the Brownian Gibbs property satisfied by the Airy wanderer line ensembles; see Definition \ref{Def.BGPVanilla} for a precise formulation. This property enables a strong comparison between the ensembles in Theorems \ref{Thm.ConvDBM}(a) and \ref{Thm.ConvAiry}(a) and ensembles of finitely many avoiding Brownian bridges. Since such Brownian bridge ensembles have well-controlled moduli of continuity, this comparison allows us to transfer analogous regularity estimates to the ensembles under consideration.

In Section \ref{Section4}, we introduce the terminology and objects needed to implement this strategy and establish several technical estimates needed for the proofs of Theorems \ref{Thm.ConvDBM}(a) and \ref{Thm.ConvAiry}(a). These estimates fall into two classes: the first concerns bounds on the maxima and minima of line ensembles consisting of finitely many avoiding Brownian bridges with varying boundary data (see Section \ref{Section4.3}), and the second concerns modulus-of-continuity estimates for such ensembles (see Section \ref{Section4.4}).

The proof of Theorem \ref{Thm.ConvDBM} is given in Section \ref{Section5.2} and relies on the following key observations:
\begin{enumerate}
\item[A.] The curves $\{\mathcal{L}_i^N\}_{i = 1}^{\mathsf{m}_k^a}$ have one-point marginals that are tight. This follows directly from the finite-dimensional convergence established in Proposition \ref{Prop.FDSlope}.
\item[B.] The curves $\{(2T_N)^{-1/2}\left(\mathcal{A}^{a,b,c}_i(tT_N)+\frac{2tT_N}{v_k^a}-t^2T_N^2\right)\}_{i = 1}^{\mathsf{M}_{k-1}^a}$ lie well above $\{\mathcal{L}_i^N\}_{i = 1}^{\mathsf{m}_k^a}$. This follows from Proposition \ref{Prop.CloseToZero}, which is proved using the finite-dimensional convergence in Proposition \ref{Prop.FDSlope} and the min/max bounds developed in Section \ref{Section4.3}.
\item[C.] The curves $\{(2T_N)^{-1/2}\left(\mathcal{A}^{a,b,c}_i(tT_N)+\frac{2tT_N}{v_k^a}-t^2T_N^2\right)\}_{i \geq \mathsf{M}_{k}^a+1}$ are much lower than $\{\mathcal{L}_i^N\}_{i = 1}^{\mathsf{m}_k^a}$. The statement follows from Proposition \ref{Prop.CloseToZero} (if $\mathsf{M}_{k}^a > J_a$) and Proposition \ref{Prop.UnderParabola} (if $\mathsf{M}_{k}^a = J_a$). The latter is established using the asymptotic slopes in Proposition \ref{Prop.Slopes}(c) and the min/max bounds in Section \ref{Section4.3}.
\end{enumerate}
The three statements above imply that $\{\mathcal{L}_i^N\}_{i = 1}^{\mathsf{m}_k^a}$ remains bounded on any compact interval $[a,b]$ and is well separated from all other curves of the appropriately scaled ensemble $\mathcal{A}^{a,b,c}$. Exploiting the Brownian Gibbs property, we show that $\{\mathcal{L}_i^N\}_{i = 1}^{\mathsf{m}_k^a}$ behaves like a finite collection of avoiding Brownian bridges. We may then apply the modulus-of-continuity estimates from Lemma \ref{Lem.MOCUniform} for finite Brownian bridge ensembles, to obtain corresponding regularity estimates for $\{\mathcal{L}_i^N\}_{i = 1}^{\mathsf{m}_k^a}$. The details can be found in Step 3 of the proof of Theorem \ref{Thm.ConvDBM} in Section \ref{Section5.2}. 

The proof of Theorem \ref{Thm.ConvAiry} is given in Section \ref{Section6.2} and relies on similar observations to A and B above. The tightness in observation A follows from the finite-dimensional convergence in Proposition \ref{Prop.FDFlat} instead of Proposition \ref{Prop.FDSlope}. In addition, the separation of the curves above $\{\mathcal{A}_i^N\}_{i \geq 1}$ again follows from Proposition \ref{Prop.CloseToZero}. Once curve separation is established, we can argue that for each $k \in \mathbb{N}$ the ensemble $\{\mathcal{A}_i^N\}_{i = 1}^{k}$ remains bounded on any compact interval $[a,b]$ and is well separated from all higher curves of the appropriately scaled ensemble $\mathcal{A}^{a,b,c}$. We can also show that the curve $\mathcal{A}_{k+1}^N$ is not too high, see Proposition \ref{Prop.NoBigMax}. Exploiting the Brownian Gibbs property, we conclude that $\{\mathcal{A}_i^N\}_{i = 1}^{k}$ behaves like a finite collection of avoiding Brownian bridges whose bottom boundary is not very high. We may then apply the modulus-of-continuity estimates from Lemma \ref{Lem.MOCInside} for finite Brownian bridge ensembles, to obtain corresponding regularity estimates for $\{\mathcal{A}_i^N\}_{i = 1}^{k}$. We emphasize that the presence of a nearby bottom boundary requires that we use Lemma \ref{Lem.MOCInside}, which is stronger than Lemma \ref{Lem.MOCUniform} that is used in the proof of Theorem \ref{Thm.ConvDBM}. The details of the outline in this paragraph can be found in Step 3 of the proof of Theorem \ref{Thm.ConvAiry} in Section \ref{Section6.2}.

%
%
\subsection*{Acknowledgments} E.D. was partially supported by Simons Foundation International through Simons Award TSM-00014004.

%
%
\section{Kernel convergence}\label{Section2} In this section, we investigate the limit of the (appropriately rescaled) kernel $K_{a,b,c}$ from (\ref{Eq.3BPKer}). In the scaling regime in Theorem~\ref{Thm.ConvDBM}(a), the precise limiting statement appears in Proposition \ref{Prop.KernelConvTop}, and for the one in Theorem \ref{Thm.ConvAiry}(a), it appears in Proposition \ref{Prop.KernelConvFlat}. In Section \ref{Section2.3} (see Proposition \ref{Prop.DBMDPP}), we relate the limiting kernel from Proposition \ref{Prop.KernelConvTop} to Dyson Brownian motion. Throughout this section we assume the notation from Section \ref{Section1}. We also use freely the definitions and notations pertaining to determinantal point processes from \cite[Section 2]{dimitrov2024airy}.

%
%
\subsection{Dyson Brownian motion limit}\label{Section2.1} The goal of this section is to analyze the limit of the correlation kernel associated with the scaling in Theorem~\ref{Thm.ConvDBM}(a). We start with some useful notation.

Fix $(a,b,c) \in \parP$. For any $\rho, T > 0$ and finite $\mathsf{S} \subset (0,\infty)$ with $\mathsf{S} = \{s_1, \dots, s_m\}$ and $s_1 < s_2 < \cdots < s_m$, we define the random measures
\begin{equation}\label{Eq.GenSlopedMeasures}
M^{\rho,T}(A) = \sum_{i \geq 1 } \sum_{j = 1}^m {\bf 1}\left\{\left(s_j, (2T)^{-1/2} \cdot \left( \mathcal{A}^{a,b,c}_{i}(s_jT) + 2s_jT/\rho - s_j^2T^2 \right) \right) \in A \right\}.
\end{equation}
If $\phi: \mathbb{R}^2 \rightarrow \mathbb{R}^2$ is the bijection with $\phi^{-1}(s,x) = (sT, \sqrt{2T}x -2sT/\rho + s^2T^2)$, we observe that $M^{\rho,T} = \hat{M}^T \phi^{-1}$, where 
$$\hat{M}^T(A) = \sum_{i \geq 1}\sum_{j = 1}^m {\bf 1}\{(s_jT, \mathcal{A}_{i}^{a,b,c}(s_jT)) \in A\}.$$
From Proposition \ref{Prop.AWLE}, we know that $\hat{M}^T$ is a determinantal point process with correlation kernel $K_{a,b,c}$ and reference measure $\mu_{T \cdot \mathsf{S}} \times \mathrm{Leb}$. By the change of variables formula \cite[Proposition 2.13(5)]{dimitrov2024airy} with $\phi$ as above, \cite[Proposition 2.13(6)]{dimitrov2024airy} with $c = (2T)^{-1/2}$, and the gauge transformation \cite[Proposition 2.13(4)]{dimitrov2024airy} with 
$$f(s,x) = \exp\left(-\sqrt{2}sxT^{3/2} -\frac{2s^3T^3}{3} + \frac{x\sqrt{2T}}{\rho} + \frac{2s^2T^2}{\rho} - \frac{sT}{\rho^2}  \right),$$
we conclude that $M^{\rho,T}$ is a determinantal point process with reference measure $\mu_{\mathsf{S}} \times \mathrm{Leb}$ and correlation kernel
\begin{equation}\label{Eq.NewKernelSlopes}
\begin{split}
&K^{\rho,T}(s,x; t,y) = -  \frac{{\bf 1}\{ t > s\} }{\sqrt{2\pi (t - s)}} \cdot e^{-\frac{(x-y)^2}{2(t-s)}} + K^{3, \rho, T}(s,x; t,y), \mbox{ where } \\
&K^{3,\rho, T}(s,x;t,y) = \frac{\sqrt{2T}}{(2\pi \im)^2} \int_{\Gamma_{\alpha }^+} d z \int_{\Gamma_{\beta}^-} dw \frac{\exp\left(H^{\rho,T}(z,s,x) - H^{\rho,T}(w,t,y)\right)}{z - w } \cdot \frac{\Phi_{a,b,c}(z) }{\Phi_{a,b,c}(w)}.
\end{split}
\end{equation}
In the last equation we have that the contours are as in Definition \ref{Def.Contours} with $a_1^{-1}> \alpha > \beta > 0$, $\Phi_{a,b,c}$ is as in (\ref{Eq.DefPhi2}) and 
\begin{equation}\label{Eq.DefH}
H^{\rho,T}(z,s,x) = \frac{z^3}{3} -z^2sT + \left(-x\sqrt{2T} + \frac{2sT}{\rho} \right)z + \frac{x\sqrt{2T}}{\rho} - \frac{sT}{\rho^2}.
\end{equation}
We mention that in deriving (\ref{Eq.NewKernelSlopes}), we applied the change of variables $z \rightarrow z - sT$, $w \rightarrow w - tT$ in the formula for $K_{a,b,c}^3$ in Definition \ref{Def.3BPKernelDef}.

With the above notation in place, we turn to the main result of the section.

\begin{proposition} \label{Prop.KernelConvTop} Assume the same notation as in Definitions~\ref{Def.DLP}, \ref{Def.ParMultiplicities}, and Proposition~\ref{Prop.AWLE}. Fix $k \in \N$, suppose that $(a, b, c) \in \parP$ satisfies $\abs{\mathsf{V}_a} \geq k$, and that $\mathsf{m}_1 = \cdots = \mathsf{m}_{k-1} = 1$. If $s, t > 0$ and $x, y \in \R$ are fixed, $x_T, y_T$ are bounded with $\lim_{T \to \infty} x_T = x$, and $\lim_{T \to \infty} y_T = y$, then 
\begin{equation}\label{Eq.KernelLimitDBM}
\lim_{T \to \infty} K^{v^a_k,T}(s, x_T; t, y_T) = \kbmk(s,x;t,y).
\end{equation}
In the last equation, $K^{v_k^a,T}$ is as in \eqref{Eq.NewKernelSlopes}, and
\begin{equation}\label{Eq.KDBM}
\begin{aligned}
\kbmn(s, x; t, y)
& = -\frac{\mathbf{1}\{t > s\}}{\sqrt{2\pi(t - s)}} \cdot e^{-\frac{(x-y)^2}{2(t-s)}} \\
&  + \frac{1}{(2\pi\im)^2} \oint_{\mathcal{C}_1} dz \int_{-2+\im\R} d w \frac{\exp(- sz^2/2 - xz + tw^2/2  + yw)}{w - z} \cdot \frac{w^n}{z^n},
\end{aligned}
\end{equation}
where $\mathcal{C}_1$ is the positively oriented zero-centered circle of unit radius.
\end{proposition}
\begin{remark}\label{Rem.KernelConvTop}
We impose the simplifying assumption $\mathsf{m}_1 = \cdots = \mathsf{m}_{k-1} = 1$ in Proposition~\ref{Prop.KernelConvTop} to streamline the proof. We expect that the conclusion holds without this restriction, but do not pursue that extension here. The present formulation of Proposition~\ref{Prop.KernelConvTop} is sufficient for our purposes.
\end{remark}
\begin{proof}
Note that since $\mathsf{m}_1 = \cdots = \mathsf{m}_{k-1} = 1$, we have $K^{v^a_k,T} = K^{a_k,T}$. Comparing the expressions for $K^{a_k,T}$ in \eqref{Eq.NewKernelSlopes} and $\kbmk$ in \eqref{Eq.KDBM}, we see that it suffices to show
\begin{equation}\label{Eq.KernConvTopRed1}
\lim_{T \to \infty} K^{3,a_k,T}(s, x_T; t, y_T) = \frac{1}{(2\pi\im)^2} \oint_{\mathcal{C}_1}d z \int_{-2+\im\R} d w \frac{\exp(-sz^2/2 - xz + tw^2/2 + yw)}{w - z} \cdot \frac{w^{\sfm^a_k}}{z^{\sfm^a_k}}.
\end{equation}

For convenience, we set $\delta = \pi/48$ and fix $A > 0$ sufficiently large so that $|x_T|, |y_T| \leq A$. In what follows $C_i > 0$ denote sufficiently large constants and $c_i > 0$ sufficiently small constants, that depend on the parameters $s,t$, $(a,b,c)$ and $A$. In addition, all inequalities below hold provided that $T$ is sufficiently large, depending on the same parameters. We do not mention this further. For clarity, we split the proof into three steps.\\

{\bf \raggedleft Step 1.} Notice that our assumption $\sfm_1 = \cdots = \sfm_{k-1} = 1$ implies that $1/a_1, \ldots, 1/a_{k-1}$ are \emph{simple} poles of $\Phi_{a,b,c}(z)$ in $K^{3,a_k,T}$ \eqref{Eq.NewKernelSlopes}.
    
We start by deforming $\Gamma^+_\alpha$ to $\Gamma^+_{v_T}$, where $v_T = 1/a_k - 1/\sqrt{2T}$, during which the contour crosses the simple poles at $1/a_1, \ldots, 1/a_{k-1}$. We subsequently deform $\Gamma^-_\beta$ to $\Gamma^-_{u_T}$, where $u_T = 1/a_k - 2/\sqrt{2T}$, during which the contour does not cross any poles. By the residue theorem,
\begin{equation}\label{Eq.K3DBMDecomp}
K^{3,a_k,T}(s, x_T; t, y_T) = V^T + \sum_{j = 1}^{k-1} U^T_j,
\end{equation}
where
\begin{equation}\label{Eq.DBMDefUV}
\begin{aligned}
V^T & = \frac{\sqrt{2T}}{(2\pi\im)^2} \int_{\Gamma^+_{v_T}} d z \int_{\Gamma^-_{u_T}} d w \frac{\exp(H^{a_k,T}(z, s, x_T) - H^{a_k,T}(w, t, y_T))}{z - w} \cdot \frac{\Phi_{a,b,c}(z)}{\Phi_{a,b,c}(w)}, \\
U^T_j& = \frac{\sqrt{2T}}{2\pi\im} \int_{\Gamma^-_{u_T}} dw\frac{\exp(H^{a_k,T}(1/a_j, s, x_T) - H^{a_k,T}(w, t, y_T))}{( 1/a_j - w) \Phi_{a,b,c}(w)} \cdot \frac{\prod_{i=1}^\infty (1 + b_i/a_j)}{a_j \prod_{i=1, i \neq j}^\infty (1 - a_i/a_j)}.
\end{aligned}
\end{equation}
We mention that the contour deformation near infinity is standard, and justified due to the decay ensured by the cubic terms in $H^{a_k,T}$. For example, see the discussion under \cite[(4.9)]{D25}. 

In view of (\ref{Eq.K3DBMDecomp}), we see that to show (\ref{Eq.KernConvTopRed1}) it suffices to prove 
\begin{equation}\label{Eq.DBMUVanish}
\lim_{T \rightarrow \infty} U_j^T = 0 \mbox{ for } j = 1, \dots, k-1,
\end{equation}
\begin{equation}\label{Eq.DBMVConverge}
\lim_{T \rightarrow \infty} V^T = \frac{1}{(2\pi\im)^2} \oint_{\mathcal{C}_1}d z \int_{-2+\im\R} d w \frac{\exp(-sz^2/2 - xz + tw^2/2 + yw)}{w - z} \cdot \frac{w^{\sfm^a_k}}{z^{\sfm^a_k}}.
\end{equation}
We establish (\ref{Eq.DBMUVanish}) and (\ref{Eq.DBMVConverge}) in the next two steps.\\

{\bf \raggedleft Step 2.} In this step, we prove (\ref{Eq.DBMUVanish}). 

By the definition of $H^{a_k,T}$ in \eqref{Eq.DefH}, there exist $c_1, C_1 > 0$ such that
\begin{equation}\label{Eq.DBM-kernel-convergence-inequality-1}
\begin{aligned}
\left|\exp(H^{a_k,T}(1/a_j, s, x_T)) \right|
& = \exp\biggl(-sT \paren{-\frac{1}{a_j} + \frac{1}{a_k}}^2 + x_T \sqrt{2T} \paren{-\frac{1}{a_j} + \frac{1}{a_k}} + \frac{1}{3a_j^3}\biggr) \\
& \leq \exp(-c_1 T + C_1 \sqrt{T}).
\end{aligned}
\end{equation}

For $a \in \mathbb{R}$, $\phi \in (0, \pi)$, we denote by $\mathcal{C}^\phi_a$ the contour 
\begin{equation}\label{Eq.ContAngle}
\mathcal{C}^\phi_a = \{a + re^{-\im\phi}: r \in [0, \infty)\} \cup \{a + re^{\im\phi}: r \in [0, \infty)\},
\end{equation}
oriented in the direction of increasing imaginary part. Note that there exist $c_2, C_2 > 0$, such that for all $\phi \in [\pi/2 + \delta, 3\pi/4 - \delta]$ and $w \in \mathcal{C}^{\phi}_{u_T}$, we have
\begin{equation}\label{Eq.DBM-kernel-convergence-inequality-2}
\begin{aligned}
&\left|\exp(-H^{a_k,T}(w, t, y_T))\right|  = \exp\biggl(tT \Re\biggl[\paren{w - \frac{1}{a_k}}^2\biggr] + y_T \sqrt{2T} \Re\oparen{w - \frac{1}{a_k}} - \frac{\Re(w^3)}{3}\biggr) \\
& = \exp\biggl(tT\Abs{w - \frac{1}{a_k}}^2 \cos(2\phi) + y_T \sqrt{2T} \Abs{w - \frac{1}{a_k}} \cos(\phi) \\
& \hphantom{{} = \exp\biggl(} - \frac{\abs{w - 1/a_k}^3 \cos(3\phi)}{3} - \frac{\abs{w - 1/a_k}^2 \cos(2\phi)}{a_k} - \frac{\abs{w - 1/a_k} \cos(\phi)}{a_k^2} - \frac{1}{3a_k^3}\biggr) \\
& \leq \exp\biggl(-c_2 \Abs{w - \frac{1}{a_k}}^3 - c_2 T \Abs{w - \frac{1}{a_k}}^2 + C_2 \sqrt{2T} \Abs{w - \frac{1}{a_k}} + C_2\biggr).
\end{aligned}
\end{equation}
We point out that in deriving the last inequality we used that $\cos(2\phi) < 0$ and $\cos(3\phi) > 0$ for $\phi \in [\pi/2 + \delta, 3\pi/4 - \delta]$ as $\delta = \pi/48$.

We next note that there exists $C_3 > 0$, such that for all $j \in \{1, \dots, k-1\}$, $\phi \in [\pi/2 + \delta, 3\pi/4 - \delta]$, and $w \in \mathcal{C}^{\phi}_{u_T}$, we have
\begin{equation}\label{Eq.DBM-kernel-convergence-inequality-3}
\Abs{\frac{1}{1/a_j - w}} \leq C_3.
\end{equation}
    
Using the simple inequalities
\begin{equation}\label{Eq.Phi-terms-inequality}
\abs{1 + z} \leq \exp(\abs{z}) \text{ for all } z \in \mathbb{C}, \text{ and }
\Abs{\frac{1}{1 + z}} \leq 1 + \frac{\abs{z}}{\abs{1 + z}} \leq \exp(\abs{z}/d) \text{ for } \abs{1 + z} \geq d > 0,
\end{equation}
taking $d = 1/2$ in particular, we can find $C_4, C_5 > 0$ such that for all $\phi \in [\pi/2 + \delta, 3\pi/4 - \delta]$, and $w \in \mathcal{C}^{\phi}_{u_T}$, we have
\begin{equation} \label{Eq.DBM-kernel-convergence-inequality-4}
\begin{aligned}
\Abs{\frac{1}{\Phi_{a,b,c}(w)}} = \Abs{\prod_{i=1}^\infty \frac{1 - a_i w}{1 + b_i w}} & \leq \exp(C_4 \abs{w}) \leq \exp\biggl(C_4 \Abs{w - \frac{1}{a_k}} + C_5\biggr).
\end{aligned}
\end{equation}

Fix $\phi_w \in [\pi/2 + \delta, 3\pi/4 - \delta]$, say $\phi_w = 2\pi/3$. We deform the contour $\Gamma^-_{u_T}$ in the definition of $U_j^T$ in (\ref{Eq.DBMDefUV}) to $\mathcal{C}^{\phi_w}_{u_T}$ without crossing any poles and hence without changing the value of the integral by Cauchy's theorem. We mention that the deformation of the contour near infinity is justified by the inequalities (\ref{Eq.DBM-kernel-convergence-inequality-1}), (\ref{Eq.DBM-kernel-convergence-inequality-2}), (\ref{Eq.DBM-kernel-convergence-inequality-3}), and (\ref{Eq.DBM-kernel-convergence-inequality-4}). Performing the deformation, and changing variables $\tilde{w} = \sqrt{2T} (w - 1/a_k)$, we conclude from the same four inequalities for some $C_6 > 0$ and all $j \in \{1, \dots, k -1\}$ that
\begin{equation}\label{Eq.BoundUT}
\left|U_j^T\right| \leq e^{-c_1 T + C_1 \sqrt{T}} \int_{\mathcal{C}_{-2}^{\phi_w}} |d\tilde{w}| \exp\left( -c_2 |\tilde{w}|^3 - c_2|\tilde{w}|^2 + C_6|\tilde{w}| + C_6 \right),
\end{equation}
where $|d\tilde{w}|$ denotes integration with respect to arc-length. As the last integral is finite, we see that the last inequality implies (\ref{Eq.DBMUVanish}).\\

{\bf \raggedleft Step 3.} In this step, we prove (\ref{Eq.DBMVConverge}). 

Fix $\phi_z \in [\pi/6 + \delta, \pi/4 - \delta]$, say $\phi_z = 3\pi/16$ and $\phi_w = 2\pi/3$ as in Step 2. We start by deforming the contours $\Gamma_{u_T}^-$ and $\Gamma_{v_T}^+$ to $\mathcal{C}^{\phi_w}_{u_T}$ and $\mathcal{C}^{\phi_z}_{v_T}$, respectively without crossing any poles, and hence without changing the value of the integral by Cauchy's theorem. The justification of the deformations is analogous to the previous step, so we omit it. Performing the deformations, and changing variables $\tilde{z} = \sqrt{2T} (z - 1/a_k)$, $\tilde{w} = \sqrt{2T} (w - 1/a_k)$, we obtain
\begin{equation}\label{Eq.DBMNewVT}
\begin{split}
&V^T  = \frac{1}{(2\pi\im)^2} \int_{\mathcal{C}^{\phi_z}_{-1}} d \tilde{z} \int_{\mathcal{C}^{\phi_w}_{-2}} d \tilde{w}  G_T(\tilde{z}, \tilde{w}), \mbox{ where } G_T(\tilde{z}, \tilde{w}) = \frac{\Phi_{a,b,c}(1/a_k + (2T)^{-1/2}\tilde{z})}{\Phi_{a,b,c}(1/a_k + (2T)^{-1/2}\tilde{w})} \\
&  \times \frac{\exp(H^{a_k,T}(1/a_k + (2T)^{-1/2}\tilde{z}, s, x_T) - H^{a_k,T}(1/a_k + (2T)^{-1/2}\tilde{w}, t, y_T))}{\tilde{z} - \tilde{w}}.
\end{split}
\end{equation}

Using the formula for $H^{a_k,T}$ from (\ref{Eq.DefH}) and $\Phi_{a,b,c}$ from (\ref{Eq.DefPhi2}), we have the pointwise limit
\begin{equation}\label{Eq.DBMGLimit}
\lim_{T\rightarrow \infty} G_T(\tilde{z}, \tilde{w}) =  \frac{\exp(-s\tilde{z}^2/2 - x\tilde{z} + t\tilde{w}^2/2 + y\tilde{w})}{\tilde{z} - \tilde{w}} \cdot \frac{\tilde{w}^{\sfm^a_k}}{\tilde{z}^{\sfm^a_k}}.
\end{equation}
In addition, we have the following bounds for some $C_7, c_7 > 0$ and all $\tilde{z} \in \mathcal{C}^{\phi_z}_{-1}$, and $\tilde{w} \in \mathcal{C}^{\phi_w}_{-2}$:
\begin{equation}\label{Eq.BunchOfBounds}
\begin{split}
&\left|\exp\left(- H^{a_k,T}(1/a_k + (2T)^{-1/2}\tilde{w}, t, y_T)) \right)\right| \leq \exp\left( - c_7|\tilde{w}|^3 - c_7|\tilde{w}|^2 + C_7|\tilde{w}| + C_7  \right) \\
&\left|\exp\left(H^{a_k,T}(1/a_k + (2T)^{-1/2}\tilde{z}, s, x_T)) \right)\right| \leq \exp\left( - c_7|\tilde{z}|^3 - c_7|\tilde{z}|^2 + C_7|\tilde{z}| + C_7  \right) \\
&\left| \frac{1}{\tilde{z} - \tilde{w}} \right| \leq 1 \mbox{ and } \left| \frac{\Phi_{a,b,c}(1/a_k + (2T)^{-1/2}\tilde{z})}{\Phi_{a,b,c}(1/a_k + (2T)^{-1/2}\tilde{w})} \right| \leq  |\tilde{w}|^{\sfm_k^a}\exp\left(C_7|\tilde{z}| + C_7|\tilde{w}| + C_7 \right).
\end{split}
\end{equation}
Indeed, the first line in (\ref{Eq.BunchOfBounds}) is the same (after changing variables and replacing $C_7, c_7, \phi_w$ with $C_2,c_2, \phi$) as (\ref{Eq.DBM-kernel-convergence-inequality-2}), and the second is established analogously using that $\cos(2\phi_z) > 0$ and $\cos(3\phi_z) < 0$, which holds as $\phi_z = 3\pi/16$. The first inequality on the third line follows from the fact that the distance between the contours $\mathcal{C}^{\phi_z}_{-1}$ and $\mathcal{C}^{\phi_w}_{-2}$ is equal to one. To see the last inequality, note that 
\begin{equation*}
\begin{split}
\frac{\Phi_{a,b,c}(z)}{\Phi_{a,b,c}(w)} = I_1 \cdot I_2 \mbox{ with } I_1 = \prod_{i = 1}^{\infty} \frac{1 + b_i z}{1 + b_i w} \cdot \prod_{i = 1}^{k-1} \frac{1 - a_i w}{1 - a_i z} \cdot \prod_{j = k+\sfm_k^a +1}^{\infty} \frac{1 - a_j w}{1 - a_j z} \mbox{ and } I_2 = \frac{\tilde{w}^{\sfm_k^a}}{\tilde{z}^{\sfm_k^a}}.
\end{split}
\end{equation*}
We can bound $|I_1|$ by $\exp(C_7|z| + C_7 |w|)$ using (\ref{Eq.Phi-terms-inequality}), with $d \in (0,1/2]$ satisfying 
$$d <  a_{k-1}/a_k - 1 \mbox{ if $k \geq 1$ and } d \leq (1/2) (v^a_{k+1}/a_k - 1) \mbox{ if } |\mathsf{V}_a| \geq k+1.$$
In addition, $|I_2|$ is bounded by $ (2|\tilde{w}|)^{\sfm_k^a}$ as $|\tilde{z}| \geq 1/2$ for all $\tilde{z} \in \mathcal{C}_{-1}^{\phi_z}$. The last few observations yield the second inequality on the third line of (\ref{Eq.BunchOfBounds}).\\

From (\ref{Eq.BunchOfBounds}), we conclude for some $C_8, c_8 > 0$ and all $\tilde{z} \in \mathcal{C}^{\phi_z}_{-1}$, $\tilde{w} \in \mathcal{C}^{\phi_w}_{-2}$ that
$$\left|G_T(\tilde{z}, \tilde{w})\right| \leq \exp\left(C_8 + C_8|\tilde{z}| + C_8|\tilde{w}| - c_8|\tilde{z}|^2 - c_8|\tilde{z}|^3 - c_8 |\tilde{w}|^2 - c_8 |\tilde{w}|^3\right).$$
As the latter function is integrable, we conclude from (\ref{Eq.DBMNewVT}), (\ref{Eq.DBMGLimit}) and the dominated convergence theorem that 
\begin{equation}\label{Eq.DBMVTLimit}
\begin{split}
&\lim_{T\rightarrow \infty} V^T  = \frac{1}{(2\pi\im)^2} \int_{\mathcal{C}^{\phi_z}_{-1}} d \tilde{z} \int_{\mathcal{C}^{\phi_w}_{-2}} d \tilde{w} \frac{\exp(-s\tilde{z}^2/2 - x\tilde{z} + t\tilde{w}^2/2 + y\tilde{w})}{\tilde{z} - \tilde{w}} \cdot \frac{\tilde{w}^{\sfm^a_k}}{\tilde{z}^{\sfm^a_k}}.
\end{split}
\end{equation}
Equation (\ref{Eq.DBMVTLimit}) implies (\ref{Eq.DBMVConverge}) after deforming $\mathcal{C}^{\phi_w}_{-2}$ to $-2+\im\R$ and $\mathcal{C}^{\phi_z}_{-1}$ to $\mathcal{C}_1$, without crossing any poles. We note that the resulting contour $\mathcal{C}_1$ is {\em negatively} oriented; reversing its orientation makes it positively oriented and changes the factor $\tilde{z} - \tilde{w}$ to $\tilde{w} - \tilde{z}$, which matches (\ref{Eq.DBMVConverge}).
\end{proof}

%
%
\subsection{Airy line ensemble limit}\label{Section2.2}

Fix $(a,b,c) \in \parP$. For any $T > 0$ and finite $\mathsf{S} \subset (0,\infty)$ with $\mathsf{S} = \{s_1, \dots, s_m\}$ and $s_1 < s_2 < \cdots < s_m$, we define the random measures
\begin{equation}\label{Eq.GenFlatMeasures}
M^{T}(A) = \sum_{i \geq 1 } \sum_{j =1}^m {\bf 1}\left\{\left(s_j, \mathcal{A}^{a,b,c}_{i}(s_j + T)  \right) \in A \right\}.
\end{equation}
If $\phi: \mathbb{R}^2 \rightarrow \mathbb{R}^2$ is the bijection with $\phi^{-1}(s,x) = (s+T, x)$, we observe that $M^{T} = \tilde{M}^T \phi^{-1}$, where 
$$\tilde{M}^T(A) = \sum_{i \geq 1}\sum_{j = 1}^m {\bf 1}\{(s_j + T, \mathcal{A}_{i}^{a,b,c}(s_j+T)) \in A\}.$$
From Proposition \ref{Prop.AWLE}, we know that $\tilde{M}^T$ is a determinantal point process with correlation kernel $K_{a,b,c}$ and reference measure $\mu_{T + \mathsf{S}} \times \mathrm{Leb}$. By the change of variables formula \cite[Proposition 2.13(5)]{dimitrov2024airy} with $\phi$ as above, we conclude that $M^{T}$ is a determinantal point process with reference measure $\mu_{\mathsf{S}} \times \mathrm{Leb}$ and correlation kernel
\begin{equation}\label{Eq.NewKernelFlat}
\begin{split}
&K^{T}(s,x; t,y) = K^2_{a,b,c} (s, x; t, y)  + K^{3, T}(s,x; t,y), \mbox{ where } \\
&K^{3,T}(s,x;t,y) = \frac{1}{(2\pi \im)^2} \int_{\Gamma_{\alpha }^+} d z \int_{\Gamma_{\beta}^-} dw \frac{e^{(z-T)^3/3 -(z-T)x - (w-T)^3/3 + (w-T)y}}{z + s - w - t} \cdot \frac{\Phi_{a,b,c}(z + s) }{\Phi_{a,b,c}(w + t)}.
\end{split}
\end{equation}
In the last equation we have that the contours are as in Definition \ref{Def.Contours} with $a_1^{-1}> \alpha +s > \beta + t > 0$, and $\Phi_{a,b,c}$ is as in (\ref{Eq.DefPhi2}). We mention that in deriving (\ref{Eq.NewKernelFlat}), we applied the change of variables $z \rightarrow z - T$, $w \rightarrow w - T$ in the formula for $K_{a,b,c}^3$ in Definition \ref{Def.3BPKernelDef}.

With the above notation in place, we turn to the main result of the section.

\begin{proposition}\label{Prop.KernelConvFlat}
Assume the same notation as in Definitions~\ref{Def.DLP}, \ref{Def.ParMultiplicities}, and in Proposition~\ref{Prop.AWLE}. Fix $(a, b, c) \in \parP$, and suppose that $\abs{\mathsf{V}_a} = J_a < \infty$. If $s, t, x, y \in \R$ are fixed, $x_T, y_T$ are bounded with $\lim_{T \to \infty} x_T = x$, and $\lim_{T \to \infty} y_T = y$, then 
\begin{equation}\label{Eq.KernelLimitAiry}
\lim_{T \to \infty} K^T(s, x_T; t, y_T) = K_{0,0,0}(s, x; t, y).
\end{equation}
In the last equation, $K^T$ is as in \eqref{Eq.NewKernelFlat} and $K_{0,0,0}$ is as in Definition~\ref{Def.3BPKernelDef} with all parameters $a_i^{\pm}, b_i^{\pm}, c^{\pm}$ set to zero.
\end{proposition}
\begin{proof} Fix $A > 0$ sufficiently large so that $|x_T|, |y_T| \leq A$. In what follows $C_i > 0$ denote sufficiently large constants and $c_i > 0$ sufficiently small constants, that depend on the parameters $s,t,\alpha, \beta$, $(a,b,c)$ and $A$. In addition, all inequalities below hold provided that $T$ is sufficiently large, depending on the same parameters. We do not mention this further. The proof we present is similar to that of Proposition \ref{Prop.KernelConvTop}, so we will be brief.

Comparing the expressions for $K^T$ in \eqref{Eq.NewKernelFlat} and $K_{0,0,0}$ in Definition~\ref{Def.3BPKernelDef}, we see that to prove (\ref{Eq.KernelLimitAiry}) it suffices to show
\begin{equation}\label{Eq.KerConvFlatRed1}
\begin{aligned}
\lim_{T \to \infty} K^{3,T}(s, x_T; t, y_T)
= \frac{1}{(2\pi\im)^2} \int_{\Gamma^+_{\alpha}} dz \int_{\Gamma^-_{\beta}} dw \frac{\exp(z^3/3 - xz - w^3/3 + yw)}{z + s - w - t}.
\end{aligned}
\end{equation}

Notice that the condition $\abs{\mathsf{V}_a} = J_a < \infty$ implies that $a_1 > \cdots > a_{J_a}$, that $a_{J_a+1} = a_{J_a+2} = \cdots = 0$, and that $1/a_1 - s, \ldots, 1/a_{J_a} - s$ are \emph{simple} poles of $\Phi_{a,b,c}(z + s)$ in \eqref{Eq.NewKernelFlat}. We deform $\Gamma^+_\alpha$ to $\Gamma^+_{\alpha + T}$ in \eqref{Eq.NewKernelFlat}, during which the contour crosses the simple poles at $1/a_1 - s, \ldots, 1/a_{J_a} - s$. We subsequently deform $\Gamma^-_\beta$ to $\Gamma^-_{\beta + T}$ and do not cross any poles. After performing the deformations, applying the residue theorem, and changing variables $z \rightarrow z+T$, $w \rightarrow w + T$, we obtain
\begin{equation*}
K^{3,T}(s, x_T; t, y_T) = V^T + \sum_{j=1}^{J_a} U^T_j,
\end{equation*}
where
\begin{equation}\label{Eq.DefUVFlat}
\begin{aligned}
V^T & = \frac{1}{(2\pi\im)^2} \int_{\Gamma^+_{\alpha}} dz \int_{\Gamma^-_{\beta}} dw \frac{\exp(z^3/3 - x_T z - w^3/3 + y_T w)}{z + s - w - t} \cdot \frac{\Phi_{a,b,c}(z + T + s)}{\Phi_{a,b,c}(w + T + t)}, \\
U^T_j & = \frac{1}{2\pi\im} \int_{\Gamma^-_{\beta}} dw \frac{\exp((1/a_j - T - s)^3/3 - x_T (1/a_j - T - s) - w^3/3 + y_T w)}{1/a_j - w - T - t} \\
& \hphantom{{} = \frac{1}{2\pi\im} \int_{\Gamma^-_{-t}} \diff w \biggl[} \times \frac{\prod_{i=1}^\infty (1 + b_i/a_j)}{\Phi_{a,b,c}(w + T + t) a_j \prod_{i=1, i \neq j}^\infty (1 - a_i/a_j)}.
\end{aligned}
\end{equation}
The last two displayed equations show that to prove (\ref{Eq.KerConvFlatRed1}) it suffices to establish 
\begin{equation}\label{Eq.AiryUVanish}
\lim_{T \rightarrow \infty} U_j^T = 0 \mbox{ for } j = 1, \dots, J_a,
\end{equation}
\begin{equation}\label{Eq.AiryVConverge}
\lim_{T \rightarrow \infty} V^T = \frac{1}{(2\pi\im)^2} \int_{\Gamma^+_{\alpha}} dz \int_{\Gamma^-_{\beta}} dw \frac{\exp(z^3/3 - xz - w^3/3 + yw)}{z + s - w - t}.
\end{equation}

We observe that for some $C_1 > 0$, and all $w \in \Gamma_{\beta}^-$ and $j \in \{1, \dots, J_a\}$, we have
$$\left|\frac{1}{1/a_j - w - T - t}  \right| \leq 1, \hspace{2mm} \left|\frac{1}{\Phi_{a,b,c}(w + T + t)}\right| =  \frac{\prod_{i = 1}^{J_a}|1 - a_i (w+T+t)|}{\prod_{i=1}^\infty|1 + b_i (w+T+t)|} \leq \exp(C_1|w| + C_1 T),$$
where in the last inequality we used (\ref{Eq.Phi-terms-inequality}) with $d = \sqrt{2}/2$. In addition, by investigating $\Re[w^3]$, we see that for some $C_2, c_2 > 0$ and all $w \in \Gamma_{\beta}^-$, we have
$$\abs{\exp(-w^3/3 + y_T w)} \leq \exp(- c_2 |w|^3 + C_2|w|^2 + C_2).$$
The last two displayed inequalities show that for some $C_3 > 0$ and all $j \in \{1, \dots, J_a\}$, we have
\begin{equation}\label{Eq.FlatUTUB}
\left|U^T_j\right| \leq e^{-T^3/3 + C_3 T^2} \cdot \int_{\Gamma^-_{\beta}} |dw| \exp\left(- c_2 |w|^3 + C_1|w| + C_2|w|^2\right),
\end{equation}
which implies (\ref{Eq.AiryUVanish}). \\

In the remainder of the proof, we establish (\ref{Eq.AiryVConverge}). We first observe that for $w \in \Gamma_{\beta}^-$ and $z \in \Gamma_{\alpha}^+$
$$\left|\prod_{i = 1}^{J_a} \frac{1 -a_i(w+T+t)}{1 -a_i(z+T+s)}\right| = \prod_{i = 1}^{J_a} \frac{|(a_iT)^{-1} - wT^{-1} - tT^{-1} - 1|}{|(a_iT)^{-1} - zT^{-1} - sT^{-1} - 1|} \leq 2^{J_a} \cdot (2 + |w|)^{J_a},$$
where in the last inequality we used that $T^{-1}\Gamma_{\alpha}^+$ is distance $1 - \alpha T^{-1}$ away from $-1$ for large $T$. In addition, for some $C_4 > 0$ and all $w \in \Gamma_{\beta}^-$ and $z \in \Gamma_{\alpha}^+$, we have
$$\left|\prod_{i \geq 1} \frac{1 +b_i(z+T+s)}{1 + b_i(w+T+t)}\right| = \prod_{i \geq 1} \frac{\left|1 + \frac{b_i(z+s)}{1 + b_iT}\right|}{\left|1 + \frac{b_i(w+t)}{1 + b_iT}\right|} \leq \exp\left(C_4|z+s| + C_4|w+t| \right),$$
where in the last inequality we used (\ref{Eq.Phi-terms-inequality}) with $d = \sqrt{2}/2$ and that $\beta + t > 0$, see the discussion under (\ref{Eq.NewKernelFlat}). By investigating $\Re[w^3]$ and $\Re[z^3]$, we also see that for some $C_5, c_5 > 0$, and all $w \in \Gamma_{\beta}^-$ and $z \in \Gamma_{\alpha}^+$, we have
$$ |\exp(z^3/3 - x_T z - w^3/3 + y_T w)| \leq \exp\left(-c_5|z|^3 - c_5 |w|^3 + C_5 + C_5|z|^2 + C_5|w|^2\right).$$
Lastly, since $\alpha + s > \beta + t$, we know for some $C_6 > 0$ and all $w \in \Gamma_{\beta}^-$ and $z \in \Gamma_{\alpha}^+$, we have
$$\left|\frac{1}{z + s - w -t}\right| \leq C_6.$$
Combining the above four estimates, we conclude for some $C_7 > 0$, and all $w \in \Gamma_{\beta}^-$ and $z \in \Gamma_{\alpha}^+$, that we have
\begin{equation}\label{Eq.BoundAllAiry}
\left|\frac{\exp(z^3/3 - x_T z - w^3/3 + y_T w)}{z + s - w - t} \cdot \frac{\Phi_{a,b,c}(z + T + s)}{\Phi_{a,b,c}(w + T + t)}\right| \leq e^{-c_2|w|^3 - c_2 |z|^3 + C_7|z|^2 + C_7|w|^2 + C_7}.
\end{equation}
As the left side is integrable and 
$$\lim_{T \rightarrow \infty} \frac{\exp(z^3/3 - x_T z - w^3/3 + y_T w)}{z + s - w - t} \cdot \frac{\Phi_{a,b,c}(z + T + s)}{\Phi_{a,b,c}(w + T + t)} = \frac{\exp(z^3/3 - x z - w^3/3 + y w)}{z + s - w - t},$$
we conclude (\ref{Eq.AiryVConverge}) by the dominated convergence theorem.
\end{proof}

%
\subsection{Determinantal structure of Dyson Brownian motion}\label{Section2.3} In this section we explain the relationship between Dyson Brownian motion and the kernel $\kbmn$ from (\ref{Eq.KDBM}). The precise formulation appears in Proposition \ref{Prop.DBMDPP} below. Before stating it, we establish two auxiliary results required for its proof. 

\begin{lemma}\label{Lem.Gauge} Fix $k \in \mathbb{N}$, and suppose that $M$ is a determinantal point process on $\mathbb{R}^k$ with correlation kernel $K$ and reference measure $\lambda$ as in \cite[Definition 2.11]{dimitrov2024airy}. Let $\mu$ be a measure on $\mathbb{R}^k$, and let $f:\mathbb{R}^k \rightarrow [0, \infty)$ be a measurable function such that $\mu(V) < \infty$, $\sup_{x \in V} f(x) < \infty$ for all compact $V \subset \mathbb{R}^k$. Suppose further that $\lambda$ is absolutely continuous with respect to $\mu$, with Radon-Nikodym derivative $f = \frac{d\lambda}{d\mu}$. Then $M$ is also a determinantal point process on $\mathbb{R}^k$ with reference measure $\mu$ and correlation kernel $f(x)^{\alpha}K(x,y) f(y)^{1- \alpha}$ for any $\alpha \in [0,1]$ (using the convention $0^0 =1$). 
\end{lemma}
\begin{proof} Fix $m \in \mathbb{N}$, $n_1, \dots, n_m \in \mathbb{N}$, $n = n_1 + \cdots + n_m$ and $m$ pairwise disjoint bounded Borel sets $A_1, \dots, A_m \subset \mathbb{R}^k$. From \cite[(2.13)]{dimitrov2024airy} and \cite[Definition 2.11]{dimitrov2024airy}, we know that 
\begin{equation*} 
\begin{split}
&\mathbb{E}\left[ \prod_{j = 1}^m \frac{M(A_j)!}{(M(A_j)-n_j)!} \right] = \int_{A_1^{n_1} \times \cdots \times A_m^{n_m}} \det\left[K(x_i,x_j)\right]_{i,j = 1}^n \lambda^n(dx) \\
&= \int_{A_1^{n_1} \times \cdots \times A_m^{n_m}} \det\left[K(x_i,x_j)\right]_{i,j = 1}^n \prod_{i = 1}^n f(x_i)\mu^n(dx) \\
&= \int_{A_1^{n_1} \times \cdots \times A_m^{n_m}} \det\left[f(x_i)^{\alpha}K(x_i,x_j) f(x_j)^{1-\alpha}\right]_{i,j = 1}^n \mu^n(dx),
\end{split}
\end{equation*}
where in going from the first to the second line we used $f = \frac{d\lambda}{d\mu}$ and in going from the second to the third we used the linearity of the determinant. The last displayed equation and \cite[Lemma~2.8]{dimitrov2024airy}, imply the statement of the lemma.
\end{proof}

\begin{lemma}\label{Lem.TimeReversal}  Fix $n \in \mathbb{N}$, and let $(\lambda^n(t): t \geq 0)$ be Dyson Brownian motion started from $\lambda^n(0) = 0$ as in Definition \ref{Def.DBM}. For $t > 0$, define $\tilde{\lambda}^n_i(t) = t \lambda^n_i(1/t)$, and set $\tilde{\lambda}^n_i(0) = 0$. Then, almost surely, $\tilde{\lambda}^n \in C([0,\infty), \weylc_n)$, and $\tilde{\lambda}^n$ has the same distribution as $\lambda^n$.
\end{lemma}
\begin{proof}
Let
$$X(t)=[X_{ij}(t)]_{i,j = 1}^n=
\begin{cases}
    B_{ii}(t), & i=j, \\
    2^{-1/2}B_{ij}(t)+\im 2^{-1/2}\cdot \widetilde{B}_{ij}(t), & i<j, \\
    2^{-1/2}B_{ij}(t)-\im 2^{-1/2}\cdot \widetilde{B}_{ij}(t), & i>j,
\end{cases}$$
where the $B_{ij}$ and $\widetilde{B}_{ij}$ are i.i.d. standard Brownian motions. Put $Y(t) = tX(1/t)$ for $t > 0$ and $Y(0) = 0$. Since for a standard Brownian motion $B$, we have $tB(1/t) \overset{d}{=}B(t)$, see \cite[Theorem 7.2.6]{Durrett}, we conclude $X(t) \overset{d}{=}Y(t)$.

Denote the eigenvalue processes of $X(t)$ and $Y(t)$ by
$$\lambda^{n,X}(t) = \left(\lambda_1^{n,X}(t),\lambda_2^{n,X}(t),\ldots,\lambda_n^{n,X}(t)\right) \mbox{ and }\lambda^{n,Y}(t) = \left(\lambda_1^{n,Y}(t),\lambda_2^{n,Y}(t),\ldots,\lambda_n^{n,Y}(t)\right),$$
respectively. Since $X(t) \overset{d}{=}Y(t)$, we conclude $\lambda^{n,X}(t) \overset{d}{=}\lambda^{n,Y}(t)$. From \cite[Theorems 3 and 4]{Grabiner99}, we have $\lambda^{n,X}(t) \overset{d}{=} \lambda^n(t)$. Since $Y(t) \overset{a.s.}{=} tX(1/t)$, we conclude $\lambda^{n,Y}(t) \overset{a.s.}{=} t\lambda^{n,X}(1/t) \overset{d}{=} \tilde{\lambda}^n(t)$. The last three observations give the statement of the lemma. 
\end{proof}

\begin{proposition}\label{Prop.DBMDPP} Fix $n \in \mathbb{N}$, and let $(\lambda^n(t): t \geq 0)$ be Dyson Brownian motion started from $\lambda^n(0) = 0$ as in Definition \ref{Def.DBM}. Fix $m \in \mathbb{N}$, $0 < t_1 < t_2 < \cdots < t_m$, and set $\mathcal{T} = \{t_1, \dots, t_m\}$. Then, the random measure 
$$M(A) = \sum_{i = 1}^n \sum_{j = 1}^m {\bf 1}\{(t_j, \lambda_i^n(t_j)) \in A\}$$
is a determinantal point process on $\mathbb{R}^2$ with correlation kernel $\kbmn$ as in (\ref{Eq.KDBM}) and reference measure $\mu_{\mathcal{T}} \times \mathrm{Leb}$ as in Definition \ref{Def.Measures}.
\end{proposition}
\begin{proof} For clarity, we split the proof into two steps. In Step 1, we show that $M$ is a determinantal point process with reference measure $\mu_{\mathcal{T}} \times \mathrm{Leb}$ and correlation kernel $K$ as in (\ref{Eq.KernelDBMV1}). In Step 2, we define an auxiliary random measure $\tilde{M}$, and show it is determinantal using our work in Step 1 and Lemma \ref{Lem.TimeReversal}. We further show that $\tilde{M}\tilde{\phi}^{-1} = M$ for a suitable bijection $\tilde{\phi}$. By applying an appropriate combination of a change of variables, gauge transformation and Lemma \ref{Lem.Gauge}, we conclude that $M$ is also determinantal with correlation kernel $\kbmn$.\\

{\bf \raggedleft Step 1.} Define the random measure on $\mathbb{R}^2$
$$\hat{M}(A) = \sum_{i = 1}^n\sum_{j = 1}^m  {\bf 1}\{ (j, \lambda_i^n(t_j)) \in A \}.$$
From the proof of \cite[Proposition 2.1]{KM10} we have that $\hat{M}$ is a determinantal point process on $\mathbb{R}^2$ with reference measure $\mu_{\{1, \dots, m\}} \times \mathrm{Leb}$ and correlation kernel
$$\hat{K}(u,x; v,y) = S^{u,v}(x,y) - {\bf 1}\{ u > v\} \cdot \frac{e^{-(x-y)^2/2(t_u - t_v)}}{\sqrt{2\pi (t_u - t_v)}},$$
where 
$$S^{u,v}(x,y) = \frac{1}{(2\pi \im)^2} \oint_{\mathcal{C}_1} dz \frac{e^{-(z-x)^2/2t_u}}{\sqrt{t_u}} \int_{\im \mathbb{R}} dw \frac{e^{(w - y)^2/2t_v}}{\sqrt{t_v}} \cdot \frac{w^n - z^n}{z^n (w - z)},$$
and $\mathcal{C}_1$ is the positively-oriented zero-centered circle of unit radius. The above formula can be deduced from the last displayed equation on \cite[page 482]{KM10} upon setting $N = n$, $x_1 = \cdots = x_N = 0$, $m = u$, $n = v$ and changing variables $z \mapsto z/\sqrt{t_u}$, $y' \mapsto -\im w/\sqrt{t_v}$.

Let $\hat{f}: \mathbb{R} \rightarrow \mathbb{R}$ be a piecewise linear increasing bijection such that $\hat{f}(i) = t_i$ for $i \in \{1, \dots, m\}$. Define $\hat{\phi}: \mathbb{R}^2 \rightarrow \mathbb{R}^2$ through $\hat{\phi}(s, x) = (\hat{f}(s),  x),$ and note that $M = \hat{M} \hat{\phi}^{-1}$. Consequently, from \cite[Proposition 2.13(5)]{dimitrov2024airy} we conclude that $M$ is determinantal with reference measure $\mu_{\mathcal{T}} \times \mathrm{Leb}$ and correlation kernel $K(s,x;t,y) = \hat{K}(\hat{f}^{-1}(s),x; \hat{f}^{-1}(t),y)$, which equals
\begin{equation}\label{Eq.KernelDBMV1}
K(s,x;t,y) =   \frac{- {\bf 1}\{ s > t\}e^{-\frac{(x-y)^2}{2(s - t)}}}{\sqrt{2\pi (s - t)}} + \frac{1}{(2\pi \im)^2} \oint_{\mathcal{C}_1} dz \int_{2 + \im \mathbb{R}} dw  \frac{e^{-(z-x)^2/2s} e^{(w - y)^2/2t}}{\sqrt{st}  (w - z)} \cdot \frac{w^n}{z^n}.
\end{equation}
We mention that in deriving the last equation, we deformed the contour $\im \mathbb{R}$ in $S^{u,v}$ to $2 + \im \mathbb{R}$, and then integrated the term $-z^n$ in the last fraction to zero using Cauchy's theorem.\\

{\bf \raggedleft Step 2.} Define $\tilde{\lambda}^n_i(t) = t \lambda^n_i(1/t)$, for $t > 0$. Fix $\mathsf{S} = \{s_1, \dots, s_m\} \subset (0, \infty)$ with $s_1 < \cdots < s_m$, and define the random measure
$$\tilde{M}(A) = \sum_{i = 1}^n\sum_{j = 1}^m  {\bf 1}\{ (s_j, \tilde{\lambda}_i^n(s_j)) \in A \}.$$ From Lemma \ref{Lem.TimeReversal}, we have that $M \overset{d}{=} \tilde{M}$ provided that $\mathsf{S} = \mathcal{T}$. Consequently, $\tilde{M}$ is a determinantal point process on $\mathbb{R}^2$ with reference measure $\mu_{\mathsf{S}} \times \mathrm{Leb}$ and correlation kernel $K$ as in (\ref{Eq.KernelDBMV1}).\\

In the remainder of the proof, we put $s_j = 1/t_{m - j + 1}$ for $j = 1, \dots, m$. Define the bijection $\tilde{\phi}: \mathbb{R}^2 \rightarrow \mathbb{R}^2$ by $\tilde{\phi}(s,x) = (1/s, x/s)$ for $s > 0$ and $\tilde{\phi}(s,x) = (s,x)$ for $s \leq 0$, and $x \in \mathbb{R}$. We observe
\begin{equation*}
\begin{aligned}
&\tilde{M}\tilde{\phi}^{-1}(A) = \sum_{i = 1}^n\sum_{j = 1}^m  {\bf 1}\{ (1/t_{m-j+1}, \tilde{\lambda}_i^n(1/t_{m-j+1})) \in \tilde{\phi}^{-1}(A) \} \\
& = \sum_{i = 1}^n\sum_{j = 1}^m  {\bf 1}\{ \tilde{\phi}(1/t_{m-j+1}, \tilde{\lambda}_i^n(1/t_{m-j+1})) \in A \}  = \sum_{i = 1}^n\sum_{j = 1}^m  {\bf 1}\{ (t_{m-j+1}, t_{m-j+1}\cdot\tilde{\lambda}_i^n(1/t_{m-j+1})) \in A\} \\
& = \sum_{i = 1}^n\sum_{j = 1}^m  {\bf 1}\{ (t_{m-j+1}, \lambda_i^n(t_{m-j+1})) \in A\} = \sum_{i = 1}^n\sum_{j = 1}^m  {\bf 1}\{ (t_{j}, \lambda_i^n(t_{j})) \in A\}= M(A),
\end{aligned}
\end{equation*}
so that $\tilde{M}\tilde{\phi}^{-1} = M$ almost surely.

By \cite[Proposition~2.13(5)]{dimitrov2024airy}, we conclude that $\tilde{M}\tilde{\phi}^{-1}$, and hence $M$, is determinantal with correlation kernel $K(\tilde{\phi}^{-1}(s,x);\tilde{\phi}^{-1}(t,y))$ and reference measure $(\mu_{\mathsf{S}} \times \mathrm{Leb})\tilde{\phi}^{-1}$. Notice that $\mu_{\mathcal{T}} \times \mathrm{Leb}$ is absolutely continuous with respect to $(\mu_{\mathsf{S}} \times \mathrm{Leb})\tilde{\phi}^{-1}$, with density $t^{-1} \cdot {\bf 1}\{t \in \mathcal{T}\}$. By Lemma \ref{Lem.Gauge} with $\alpha=1/2$, we conclude that $M$ is determinantal with reference measure $\mu_{\mathcal{T}} \times \mathrm{Leb}$ and correlation kernel $\bar{K}(s,x; t,y) = (st)^{-1/2}K(\tilde{\phi}^{-1}(s,x);\tilde{\phi}^{-1}(t,y))$, which from (\ref{Eq.KernelDBMV1}) equals
\begin{equation*}
\bar{K}(x,s;t,y) =   \frac{- {\bf 1}\{ t > s\}e^{-\frac{(x/s-y/t)^2}{2(1/s - 1/t)}} }{\sqrt{2\pi (t-s)}} + \frac{1}{(2\pi \im)^2} \oint_{\mathcal{C}_1} dz \int_{2 + \im \mathbb{R}} dw  \frac{e^{-s(z-x/s)^2/2} e^{t(w - y/t)^2/2}}{w - z} \cdot \frac{w^n}{z^n}.
\end{equation*}

Applying a gauge transformation, see \cite[Proposition 2.13(4)]{dimitrov2024airy} with $f(s,x) = e^{x^2/2s}$, we conclude that $M$ is determinantal with reference measure $\mu_{\mathcal{T}} \times \mathrm{Leb}$ and correlation kernel $\frac{f(s,x)}{f(t,y)}\bar{K}(x,s;t,y)$, which with a bit of algebra is seen to equal
$$ \frac{- {\bf 1}\{ t > s\}e^{-\frac{(x-y)^2}{2(t - s)}} }{\sqrt{2\pi (t-s)}} + \frac{1}{(2\pi \im)^2} \oint_{\mathcal{C}_1} dz \int_{2 + \im \mathbb{R}} dw  \frac{e^{-sz^2/2 +xz + tw^2/2 - yw}}{w - z} \cdot \frac{w^n}{z^n}.$$
The last formula now matches with (\ref{Eq.KDBM}) once we change variables $z \rightarrow -z$ and $w \rightarrow -w$.
\end{proof}

%
%
\section{Finite-dimensional convergence}\label{Section3} The goal of this section is to establish the finite-dimensional convergence of the line ensembles appearing in Theorems \ref{Thm.ConvDBM} and \ref{Thm.ConvAiry}. We accomplish this in Sections \ref{Section3.3} and \ref{Section3.4}, see Propositions \ref{Prop.FDSlope} and \ref{Prop.FDFlat}. In order to establish these results, we require two upper-tail estimates. The first is for the scaling in Theorem \ref{Thm.ConvDBM}, see Proposition \ref{Prop.TightnessSlope} in Section \ref{Section3.1}, and the second is for the scaling in Theorem \ref{Thm.ConvAiry}, see Proposition \ref{Prop.TightnessFlat} in Section \ref{Section3.2}. Throughout this section we assume the notation from Section \ref{Section1}. We also use freely the definitions and notations pertaining to determinantal point processes from \cite[Section 2]{dimitrov2024airy}.

%
%
\subsection{Upper tail estimates for the Dyson Brownian motion limit}\label{Section3.1} The goal of this section is to establish the following result, which shows that the one-point marginals of the curve $\mathcal{L}_1^N$ in Theorem \ref{Thm.ConvDBM}(a) are tight from above.

\begin{proposition} \label{Prop.TightnessSlope} Assume the same notation as in Definitions~\ref{Def.DLP}, \ref{Def.ParMultiplicities}, and Proposition~\ref{Prop.AWLE}. Fix $k \in \N$, suppose that $(a, b, c) \in \parP$ satisfies $\abs{\mathsf{V}_a} \geq k$, and suppose that $\mathsf{m}_1 = \cdots = \mathsf{m}_{k-1} = 1$. For a fixed $t \in (0, \infty)$, define
\begin{equation*}
X_i^T = (2T)^{-1/2} \left(\mathcal{A}_i^{a,b,c}(tT) + 2tT/v^a_k - (tT)^2\right) \mbox{ for each } i \geq 1.
\end{equation*}
Then, the following statement holds:
\begin{equation}\label{Eq.TightnessSlope}
    \lim_{A \to \infty} \limsup_{T \to \infty} \P\left(X^T_k \geq A\right) = 0.
\end{equation}
\end{proposition}
\begin{proof} For clarity, we split the proof into two steps. In the first step, we reduce the proof of the proposition to establishing that certain single and double integrals converge to zero as $T \rightarrow \infty$, see (\ref{Eq.DBMUBUVanish}) and (\ref{Eq.DBMUBVVanish}), and that certain probabilities converge to $1$, see (\ref{Eq.DBMUBProbConv}). In Step 2 we prove (\ref{Eq.DBMUBUVanish}) and (\ref{Eq.DBMUBVVanish}) by utilizing the bounds from Steps 2 and 3 of the proof of Proposition \ref{Prop.KernelConvTop}. We also prove (\ref{Eq.DBMUBProbConv}) using the asymptotic slopes of the Airy wanderer line ensembles from Proposition~\ref{Prop.Slopes}(a).\\

{\bf \raggedleft Step 1.} The argument in this step essentially repeats Steps 1 and 2 in the proof of \cite[Lemma 4.2]{D25}, so we will be brief. Let $M^{v^a_k,T}$ be as in (\ref{Eq.GenSlopedMeasures}) for $\mathsf{S} = \{t\}$. Notice that by the definition of $X_i^T$, we have $M^{v^a_k,T}(\cdot) = \sum_{i=1}^\infty {\bf 1}\{(t, X^T_i) \in {\cdot}\}$. In addition, from our work in Section \ref{Section2.1}, we know that $M^{v^a_k,T}$ is determinantal with correlation kernel $K^{v^a_k,T}$ as in \eqref{Eq.NewKernelSlopes}. From \cite[(2.13)]{dimitrov2024airy}, we conclude
\begin{equation*}
\begin{aligned}
\sum_{i=1}^\infty \P(X^T_i \geq A)
& = \E \left[M^{v^a_k,T}(\{t\} \times [A, \infty))\right] = \int_A^\infty K^{v^a_k,T}(t, x; t, x) dx \\
& = \frac{\sqrt{2T}}{(2\pi\im)^2} \int_{\Gamma^+_\alpha} dz \int_{\Gamma^-_\beta} dw \int_A^\infty dx \frac{\exp(H^{a_k,T}(z, t, x) - H^{a_k,T}(w, t, x))}{z - w} \cdot \frac{\Phi_{a,b,c}(z)}{\Phi_{a,b,c}(w)} \\
& = \frac{1}{(2\pi\im)^2} \int_{\Gamma^+_\alpha} dz \int_{\Gamma^-_\beta} dw \frac{\exp(H^{a_k,T}(z, t, A) - H^{a_k,T}(w, t, A))}{(z - w)^2} \cdot\frac{\Phi_{a,b,c}(z)}{\Phi_{a,b,c}(w)}.
\end{aligned}
\end{equation*}
We recall that $0 < \beta < \alpha < 1/a_1$ and that $H^{a_k,T}$ is as in \eqref{Eq.DefH}. We mention that the exchange of the order of the integrals is justified by Fubini's theorem, whose applicability follows from the same bounds established after \cite[(4.4)]{D25}.

As in the proof of Proposition~\ref{Prop.KernelConvTop}, we deform $\Gamma^+_\alpha$ to $\Gamma^+_{v_T}$, where $v_T = 1/v^a_k - 1/\sqrt{2T}$, during which the contour crosses the simple poles at $1/a_1, \ldots, 1/a_{k-1}$. We subsequently deform $\Gamma^-_\beta$ to $\Gamma^-_{u_T}$, where $u_T = 1/v^a_k - 2/\sqrt{2T}$, during which the contour \emph{does} cross the poles at $1/a_1, \ldots, 1/a_{k-1}$. These poles were not present in Step 1 of the proof of Proposition \ref{Prop.KernelConvTop} and can be traced to the fact that our present denominator is $(z - w)^2$ instead of $z - w$. In addition, we observe that the additional $k-1$ residues that we collect are all equal to $1$. By the residue theorem, we conclude
\begin{equation}\label{Eq.DBMUBDecomp}
\sum_{i=1}^\infty \P(X^T_i \geq A) = k-1 + V^T(A) + \sum_{j = 1}^{k-1} U^T_j(A),
\end{equation}
where 
\begin{equation}\label{Eq.DBMUBVUDef}
\begin{aligned}
V^T(A) & = \frac{1}{(2\pi\im)^2} \int_{\Gamma^+_{v_T}} d z \int_{\Gamma^-_{u_T}} d w \frac{\exp(H^{a_k,T}(z, t, A) - H^{a_k,T}(w, t, A))}{(z - w)^2} \cdot \frac{\Phi_{a,b,c}(z)}{\Phi_{a,b,c}(w)}, \\
U^T_j(A) & = \frac{\sqrt{2T}}{2\pi\im} \int_{\Gamma^-_{u_T}} dw\frac{\exp(H^{a_k,T}(1/a_j, t, A) - H^{a_k,T}(w, t, A))}{( 1/a_j - w)^2 \Phi_{a,b,c}(w)} \cdot \frac{\prod_{i=1}^\infty (1 + b_i/a_j)}{a_j \prod_{i=1, i \neq j}^\infty (1 - a_i/a_j)}.
\end{aligned}
\end{equation}

From (\ref{Eq.DBMUBDecomp}), we would conclude the statement of the proposition, if we can show that 
\begin{equation}\label{Eq.DBMUBUVanish}
\lim_{T \rightarrow \infty} \left|U^T_j(A)\right| = 0 \mbox{ for } j = 1, \dots, k-1 \mbox{ and }A \in \mathbb{R},
\end{equation}
\begin{equation}\label{Eq.DBMUBVVanish}
\lim_{A \rightarrow \infty}\limsup_{T \rightarrow \infty} \left|V^T(A)\right| = 0,
\end{equation}
\begin{equation}\label{Eq.DBMUBProbConv}
\lim_{T \rightarrow \infty} \mathbb{P}(X_j^T \geq A) = 1, \mbox{ for }j = 1, \dots, k-1 \mbox{ and }A \in \mathbb{R}.
\end{equation}
We establish the last three equations in the next step.\\

{\bf \raggedleft Step 2.} From (\ref{Eq.DBMUBVUDef}), we see that $U^T_j(A)$ agrees with $U^T_j$ from (\ref{Eq.DBMDefUV}) for $x_T = y_T = A$, except that the denominator of the integrand has an extra $(1/a_j-w)$. Consequently, we can repeat the same arguments as in Step 2 of the proof of Proposition \ref{Prop.KernelConvTop} to obtain the following analogue of (\ref{Eq.BoundUT}):
\begin{equation*}
\left|U_j^T(A)\right| \leq C_3e^{-c_1 T + C_1 \sqrt{T}} \int_{\mathcal{C}_{-2}^{\phi_w}} |d\tilde{w}| \exp\left( -c_2 |\tilde{w}|^3 - c_2|\tilde{w}|^2 + C_6|\tilde{w}| + C_6 \right),
\end{equation*} 
where $C_i,c_i > 0$ are positive constants that depend on $A$, in addition to $t$, $(a,b,c)$. We mention that the additional factor $C_3$ arises from the extra term $(1/a_j-w)$ and (\ref{Eq.DBM-kernel-convergence-inequality-3}). In addition, we recall that the contour $\mathcal{C}_{-2}^{\phi_w}$ is as in (\ref{Eq.ContAngle}), that $\phi_w = 2\pi/3$ and $|d\tilde{w}|$ denotes integration with respect to arc-length. The last displayed inequality proves (\ref{Eq.DBMUBUVanish}).\\

We can next repeat the argument in Step 3 of the proof of Proposition \ref{Prop.KernelConvTop} to obtain the following analogue of (\ref{Eq.DBMNewVT}):
\begin{equation}\label{Eq.DBMUBNewVT}
\begin{split}
&V^T(A)  = \frac{1}{(2\pi\im)^2} \int_{\mathcal{C}^{\phi_z}_{-1}} d \tilde{z} \int_{\mathcal{C}^{\phi_w}_{-2}} d \tilde{w}  \frac{G_T(\tilde{z}, \tilde{w})}{\tilde{z}-\tilde{w}}, 
\end{split}
\end{equation}
where $G_T(\tilde{z}, \tilde{w})$ is as in (\ref{Eq.DBMNewVT}) and the contours $\mathcal{C}_{-2}^{\phi_w}, \mathcal{C}_{-1}^{\phi_z}$ are as in (\ref{Eq.ContAngle}) with $\phi_w = 2\pi/3, \phi_z = 3\pi/16$.

From the definition of $H^{\rho,T}$ in (\ref{Eq.DefH}), and the first two inequalities in (\ref{Eq.BunchOfBounds}) applied to $x_T = y_T = 0$, we have for some $C_7', c_7' > 0$, depending on $t$, $(a,b,c)$, and all $A \geq 0$, $\tilde{z} \in \mathcal{C}^{\phi_z}_{-1}$, and $\tilde{w} \in \mathcal{C}^{\phi_w}_{-2}$
\begin{equation}\label{Eq.DBMUBHBound}
\begin{split}
&\left|\exp\left(H^{a_k,T}(1/a_k + (2T)^{-1/2}\tilde{z}, t, A)- H^{a_k,T}(1/a_k + (2T)^{-1/2}\tilde{w}, t, A)) \right)\right| \\
&= \left|\exp\left(H^{a_k,T}(1/a_k + (2T)^{-1/2}\tilde{z}, t, 0)- H^{a_k,T}(1/a_k + (2T)^{-1/2}\tilde{w}, t, 0)) \right)\right| \cdot \left|e^{A(\tilde{w}-\tilde{z})}\right| \\
& \leq \exp\left( - c_7'|\tilde{w}|^3 - c_7'|\tilde{w}|^2 + C_7'|\tilde{w}| - c_7'|\tilde{z}|^3 - c_7'|\tilde{z}|^2 + C_7'|\tilde{z}| + C_7' - A  \right) ,
\end{split}
\end{equation}
where we used that $\left|e^{A(\tilde{w}-\tilde{z})}\right| = e^{A\Re(\tilde{w} - \tilde{z})} \leq e^{-A}$ as the horizontal distance from $\mathcal{C}^{\phi_z}_{-1}$ to $\mathcal{C}^{\phi_w}_{-2}$ is at least $1$. Combining (\ref{Eq.DBMUBNewVT}) with the bounds in (\ref{Eq.DBMUBHBound}), and the third line in (\ref{Eq.BunchOfBounds}), we conclude for some $C_8', c_8' > 0$, depending on $t, (a,b,c)$ and all $A \geq 0$
$$\left|V^T(A) \right| \leq e^{-A} \cdot \int_{\mathcal{C}^{\phi_z}_{-1}} |d \tilde{z}| \int_{\mathcal{C}^{\phi_w}_{-2}} |d \tilde{w}| e^{C_8' + C_8'|\tilde{z}| + C_8'|\tilde{w}| - c_8'|\tilde{z}|^2 - c_8'|\tilde{z}|^3 - c_8' |\tilde{w}|^2 - c_8' |\tilde{w}|^3}.$$
The last displayed inequality proves (\ref{Eq.DBMUBVVanish}).\\

We finally focus on (\ref{Eq.DBMUBProbConv}). From Proposition~\ref{Prop.Slopes}(a) and Slutsky's theorem, the following weak convergence holds for each $i \in \{1, \dots, k-1\}$ as $T \to \infty$:
\begin{equation*}
\frac{\sqrt{2T} X^T_i + (-2/a_k + 2/a_i) tT}{tT} = \frac{(\mathcal{A}^{a,b,c}_i(tT) + (2/a_i) tT - t^2 T^2)}{tT} \Rightarrow 0.
\end{equation*}
Since $a_i > a_k$ by assumption, the last statement implies (\ref{Eq.DBMUBProbConv}).
\end{proof}

%
%
\subsection{Upper tails for the Airy line ensemble limit}\label{Section3.2} The goal of this section is to establish the following result, which shows that the one-point marginals of the curve $\mathcal{A}_1^N$ in Theorem \ref{Thm.ConvAiry}(a) are tight from above.

\begin{proposition}\label{Prop.TightnessFlat}
Assume the same notation as in Definitions~\ref{Def.DLP}, \ref{Def.ParMultiplicities}, and Proposition~\ref{Prop.AWLE}. Fix $(a, b, c) \in \parP$, and suppose that $\abs{\mathsf{V}_a} = J_a < \infty$. For a fixed $t \in \R$, it holds that
\begin{equation}\label{Eq.TightnessFlat}
\lim_{A \to \infty} \limsup_{T \to \infty} \P\left(\mathcal{A}^{a,b,c}_{1+J_a}(T + t) \geq A\right) = 0.
\end{equation}
\end{proposition}
\begin{proof} For clarity, we split the proof into two steps. In the first step, we reduce the proof of the proposition to establishing that certain single and double integrals converge to zero as $T \rightarrow \infty$, see (\ref{Eq.AiryUBUVanish}) and (\ref{Eq.AiryUBVVanish}), and that certain probabilities converge to $1$, see (\ref{Eq.AiryUBProbConv}). In Step 2 we prove (\ref{Eq.AiryUBUVanish}) and (\ref{Eq.AiryUBVVanish}) by utilizing the bounds from the proof of Proposition \ref{Prop.KernelConvFlat}. We also prove (\ref{Eq.AiryUBProbConv}) using the asymptotic slopes of the Airy wanderer line ensembles from Proposition~\ref{Prop.Slopes}(a).\\

{\bf \raggedleft Step 1.} Let $M^{T}$ be as in (\ref{Eq.GenFlatMeasures}) for $\mathsf{S} = \{t\}$. From our work in Section \ref{Section2.2}, we know that $M^{T}$ is determinantal with correlation kernel $K^{T}$ as in \eqref{Eq.NewKernelFlat}. From \cite[(2.13)]{dimitrov2024airy} and Fubini's theorem, we conclude that
\begin{equation*}
\begin{aligned}
&\sum_{i=1}^\infty \P(\mathcal{A}^{a,b,c}_i(T + t) \geq A) = \E\left[M^T(\{t\} \times [A, \infty))\right] = \int_A^\infty K^T(t, x; t, x) dx \\
& = \frac{1}{(2\pi\im)^2} \int_{\Gamma^+_\alpha} dz \int_{\Gamma^-_\beta} dw \int_A^\infty dx \frac{\exp((z - T)^3/3 - zx - (w - T)^3/3 + wx)}{z - w} \cdot \frac{\Phi_{a,b,c}(z + t)}{\Phi_{a,b,c}(w + t)} \\
& = \frac{1}{(2\pi\im)^2} \int_{\Gamma^+_\alpha} dz \int_{\Gamma^-_\beta} dw \frac{\exp((z - T)^3/3 - zA - (w - T)^3/3 + wA)}{(z - w)^2} \cdot \frac{\Phi_{a,b,c}(z + t)}{\Phi_{a,b,c}(w + t)}.
\end{aligned}
\end{equation*}
We recall from \eqref{Eq.NewKernelFlat} that $0 < \beta + t < \alpha + t < 1/a_1$.

As in the proof of Proposition~\ref{Prop.KernelConvFlat}, we deform $\Gamma^+_\alpha$ to $\Gamma^+_{\alpha + T}$, during which the contour crosses the simple poles at $1/a_1-t, \ldots, 1/a_{J_a} - t$. We subsequently deform $\Gamma^-_\beta$ to $\Gamma^-_{\beta + T}$, during which the contour \emph{does} cross the poles at $1/a_1-t, \ldots, 1/a_{J_a} -t$. These poles were not present in the proof of Proposition \ref{Prop.KernelConvFlat} and can be traced to the fact that our present denominator is $(z - w)^2$ instead of $z - w$. In addition, we observe that the additional $J_a$ residues that we collect are all equal to $1$. By the residue theorem, and changing variables $z \rightarrow z+ T$, $w \rightarrow w + T$, we conclude
\begin{equation}\label{Eq.AiryUBDecomp}
\sum_{i=1}^\infty \P(\mathcal{A}^{a,b,c}_i(T + t) \geq A) = J_a + V^T(A) + \sum_{j = 1}^{J_a} U^T_j(A),
\end{equation}
where
\begin{equation}\label{Eq.AiryUBVUDef}
\begin{aligned}
V^T(A) & =  \frac{1}{(2\pi\im)^2} \int_{\Gamma^+_\alpha} dz \int_{\Gamma^-_{\beta}} dw \frac{\exp(z^3/3 - zA - w^3/3 + wA)}{(z - w)^2} \cdot \frac{\Phi_{a,b,c}(z + T + t)}{\Phi_{a,b,c}(w + T + t)}, \\
U^T_j(A) &= \frac{1}{2\pi\im} \int_{\Gamma^-_{\beta}} dw \frac{\exp((1/a_j - T - t)^3/3 - A (1/a_j - T - t) - w^3/3 + A w)}{(1/a_j - w - T - t)^2} \\
& \hphantom{{} = \frac{1}{2\pi\im} \int_{\Gamma^-_{-t}} \diff w \biggl[} \times \frac{\prod_{i=1}^\infty (1 + b_i/a_j)}{\Phi_{a,b,c}(w + T + t) a_j \prod_{i=1, i \neq j}^\infty (1 - a_i/a_j)}.
\end{aligned}
\end{equation}

From (\ref{Eq.AiryUBDecomp}), we would conclude the statement of the proposition, if we can show that 
\begin{equation}\label{Eq.AiryUBUVanish}
\lim_{T \rightarrow \infty} \left|U^T_j(A)\right| = 0 \mbox{ for } j = 1, \dots, J_a \mbox{ and }A \in \mathbb{R},
\end{equation}
\begin{equation}\label{Eq.AiryUBVVanish}
\lim_{A \rightarrow \infty}\limsup_{T \rightarrow \infty} \left|V^T(A)\right| = 0,
\end{equation}
\begin{equation}\label{Eq.AiryUBProbConv}
\lim_{T \rightarrow \infty} \mathbb{P}(\mathcal{A}^{a,b,c}_j(T + t) \geq A) = 1, \mbox{ for }j = 1, \dots, J_a \mbox{ and }A \in \mathbb{R}.
\end{equation}
We establish the last three equations in the next step.\\

{\bf \raggedleft Step 2.} From (\ref{Eq.AiryUBVUDef}), we see that $U^T_j(A)$ agrees with $U^T_j$ from (\ref{Eq.DefUVFlat}) for $x_T = y_T = A$, except that the denominator of the integrand has an extra $(1/a_j-w-T-t)$. Consequently, we can repeat the same arguments as in the proof of Proposition \ref{Prop.KernelConvFlat} to obtain the following analogue of (\ref{Eq.FlatUTUB}):
\begin{equation*}
\left|U_j^T(A)\right| \leq e^{-T^3/3 + C_3 T^2} \cdot \int_{\Gamma^-_{\beta}} |dw| \exp\left(- c_2 |w|^3 + C_1|w| + C_2|w|^2\right),
\end{equation*} 
where $C_i,c_i > 0$ are positive constants that depend on $A$, in addition to $t$, $(a,b,c)$. The last displayed inequality proves (\ref{Eq.AiryUBUVanish}).\\

From (\ref{Eq.BoundAllAiry}) applied to $x_T = y_T = 0$ and the trivial bounds 
$$|e^{-zA + wA}| \leq e^{-A(\alpha - \beta)} \mbox{ and } \left|\frac{1}{z-w}\right|\leq \frac{1}{\alpha - \beta},$$
which hold for all $A \geq 0$, $z \in \Gamma_{\alpha}^+$, $w \in \Gamma_{\beta}^-$ as the horizontal distance between $\Gamma_{\alpha}^+$ and $\Gamma_{\beta}^-$ is equal to $\alpha - \beta$, we conclude
$$\left|V^T(A)\right| \leq e^{-A(\alpha -\beta)} \cdot \int_{\Gamma^+_\alpha} |dz| \int_{\Gamma^-_{\beta}} |dw| \frac{e^{-c_2|w|^3 - c_2 |z|^3 + C_7|z|^2 + C_7|w|^2 + C_7}}{\alpha - \beta}.$$
The last displayed inequality proves (\ref{Eq.AiryUBVVanish}).\\

We finally focus on (\ref{Eq.AiryUBProbConv}). From Proposition~\ref{Prop.Slopes}(a) and Slutsky's theorem, the following weak convergence holds for each $i \in \{1, \dots, J_a\}$ as $T \to \infty$:
\begin{equation*}
\frac{\mathcal{A}^{a,b,c}_i(T + t)}{T + t} - (T + t) + \frac{2}{a_i} \Rightarrow 0.
\end{equation*}
The last statement implies (\ref{Eq.AiryUBProbConv}).
\end{proof}

%
%
\subsection{Finite-dimensional convergence to Dyson Brownian motion}\label{Section3.3} The goal of this section is to establish the finite-dimensional convergence of the line ensembles in Theorem \ref{Thm.ConvDBM}(a).
\begin{proposition}\label{Prop.FDSlope} Assume the same notation as in Theorem \ref{Thm.ConvDBM}(a). Then,
$$\left(\mathcal{L}^N_i(t): t > 0 \mbox{, }i = 1, \dots, \mathsf{m}_k^a \right) \overset{f.d.}{\rightarrow}\left(\lambda^{\mathsf{m}_k^a}_i(t): t > 0 \mbox{, }i = 1, \dots, \mathsf{m}_k^a\right).$$
\end{proposition}
\begin{proof} Fix $\mathsf{S} = \{s_1, \dots, s_m\} \subset (0,\infty)$ with $s_1 < \cdots < s_m$ and define for $s \in \mathsf{S}$
\begin{equation}\label{eq:Xin-def}
X^N_i(s)
:=(2T_N)^{-1/2}\left(\mathcal{A}^{a,b,c}_i(sT_N)+\frac{2sT_N}{v_k^a}-s^2T_N^2\right).
\end{equation}
From the definition of $\mathcal{L}^N$ in (\ref{Eq.ScaledSlopeA}), we see that it suffices to prove
\begin{equation}\label{Eq.FDSlopeR1}
\begin{split}
&\left(X^N_{i+K}(s_j): i \in \{ 1, \dots, \mathsf{m}_k^a \} \mbox{, } j \in \{1, \dots, m\} \right) \\
& \overset{f.d.}{\rightarrow}\left(\lambda^{\mathsf{m}_k^a}_i(s_j): i \in \{ 1, \dots, \mathsf{m}_k^a \} \mbox{, } j \in \{1, \dots, m\} \right),
\end{split}
\end{equation}
where we have set $K = \mathsf{M}^a_{k-1}$.

For clarity, we split the proof of (\ref{Eq.FDSlopeR1}) into four steps. In Step 1, we reduce the proof of the proposition to the case when $a_1 > a_2 > \cdots > a_{K+1}$. Once we know the result for such parameters, we can use the monotone coupling from Proposition \ref{Prop.MonCoupling} to ``squeeze'' $\mathcal{A}^{a,b,c}$ between two ensembles $\mathcal{A}^{a^-,b^-,c^-}$ and $\mathcal{A}^{a^+,b^+,c^+}$ whose parameters satisfy $a^{\pm}_1 > a^{\pm}_2 > \cdots > a^{\pm}_{K+1}$ and for which (\ref{Eq.FDSlopeR1}) holds. This allows us to conclude the result for all parameters as in Theorem \ref{Thm.ConvDBM}(a). In the rest of the proof we assume $a_1 > a_2 > \cdots > a_{K+1}$, which allows us to use Propositions \ref{Prop.KernelConvTop} and \ref{Prop.TightnessSlope}.  

In Step 2, we use \cite[Lemma 7.1]{DZ25} to reduce the proof of the proposition to two statements, see (\ref{Eq.DBMVagueConv}) and (\ref{Eq.DBMTightParticles}). The first of these statements, (\ref{Eq.DBMVagueConv}), implies that the point processes on $\mathbb{R}^2$ formed by $(s, \mathcal{L}_i^N(s))$ converge weakly to those formed by $(s,\lambda^{\mathsf{m}_k^a}_i(s))$. We establish this in Step 3, by combining our kernel convergence from Proposition \ref{Prop.KernelConvTop}, the asymptotic slopes of the Airy wanderer line ensemble in Proposition \ref{Prop.Slopes}, and the weak convergence of determinantal point processes criterion in \cite[Proposition 2.18]{dimitrov2024airy}. The second of these statements, (\ref{Eq.DBMTightParticles}), implies that the sequence $\{\mathcal{L}_i^N(s_j)\}_{N \geq 1}$ is tight. We establish this in Step 4, by combining Propositions \ref{Prop.KernelConvTop} and \ref{Prop.Slopes} (as in Step 3) with the tightness criterion in \cite[Lemma 7.2]{DZ25}.\\

{\bf \raggedleft Step 1.} In this step, we assume that (\ref{Eq.FDSlopeR1}) holds when $a_1 > a_2 > \cdots > a_{K+1}$, where $K = \mathsf{M}^a_{k-1}$ and proceed to prove it for general parameters $k, (a,b,c)$ as in the statement of Theorem \ref{Thm.ConvDBM}(a).

Fix $\varepsilon > 0$ sufficiently small so that $v_{k-1}^a > v_{k}^a + K\varepsilon$ with the convention $v_0^a = \infty$. We define $(a^{+},b^{+},c^{+}), (a^{-},b^{-},c^-) \in \parP$ to be the same $(a,b,c)$, except that 
$$a_i^{+} = a_i + (K-i+1), \mbox{ and } a_i^- = a_i - i \varepsilon \mbox{ for } i = 1, \dots, K,$$
and observe that by our choice of $\varepsilon$ we have
\begin{equation}\label{Eq.DBMSqueezePar}
\begin{split}
&a_1^+ > a_2^+ > \cdots > a_{K+1}^+ = a_{K+1}, \hspace{2mm} a_1^- > a_2^- > \cdots > a_{K+1}^- = a_{K+1}\\
&a_i^- \leq a_i \leq a_i^+ \mbox{ for } i \in \mathbb{N}, \hspace{2mm} c^+ = c = c^- = 0, \hspace{2mm} b_i^- = b_i = b_i^+ \mbox{ for }i\in\mathbb{N}.
\end{split}
\end{equation}

From the second line in (\ref{Eq.DBMSqueezePar}) and Proposition \ref{Prop.MonCoupling} applied to $A = B = 0$, we can couple $\mathcal{A}^{a^-,b^-,c^-}$ with $\mathcal{A}^{a,b,c}$, so that $\mathcal{A}_i^{a^-,b^-,c^-}(t) \leq \mathcal{A}^{a,b,c}_i(t)$ for all $(i,t) \in \mathbb{N} \times \mathbb{R}$. Similarly, we can couple $\mathcal{A}^{a^+,b^+,c^+}$ with $\mathcal{A}^{a,b,c}$, so that $\mathcal{A}_i^{a,b,c}(t) \leq \mathcal{A}^{a^+,b^+,c^+}_i(t)$ for all $(i,t) \in \mathbb{N} \times \mathbb{R}$. Let $X_i^{\pm,N}$ be as in (\ref{eq:Xin-def}) with $\mathcal{A}^{a,b,c}$ replaced with $\mathcal{A}^{a^{\pm},b^{\pm},c^{\pm}}$ and observe that our two monotone couplings imply for all $r_{i,j} \in \mathbb{R}$ that 
\begin{equation}\label{Eq.DBMSqueezeX}
\begin{split}
&P^{+,N} \leq \mathbb{P}\left(X^N_{i+K}(s_j) \leq r_{i,j}:  i \in \{ 1, \dots, \mathsf{m}_k^a \} \mbox{, } j \in \{1, \dots, m\} \right)  \leq P^{-,N} \mbox{, where }\\
&P^{\pm,N} = \mathbb{P}\left(X^{\pm, N}_{i+K}(s_j) \leq r_{i,j}: i \in \{1, \dots, \mathsf{m}_k^a \} \mbox{, } j \in \{1, \dots, m\} \right).
\end{split}
\end{equation}

The first line in (\ref{Eq.DBMSqueezePar}) and our assumption in the beginning of the step show that 
$$\lim_{N \rightarrow \infty}P^{\pm,N} = \mathbb{P}\left(\lambda^{\mathsf{m}_k^a}_i(s_j) \leq r_{i,j}:  i \in \{ 1, \dots, \mathsf{m}_k^a \} \mbox{, } j \in \{1, \dots, m\} \right).$$
Combining the latter with (\ref{Eq.DBMSqueezeX}), we conclude (\ref{Eq.FDSlopeR1}).\\

{\bf \raggedleft Step 2.} In the remainder of the proof, we assume $a_1 > a_2 > \cdots > a_{K+1}$, where $K = \mathsf{M}^a_{k-1}$, so that in particular $K = k-1$, $\mathsf{m}_1^a = \cdots = \mathsf{m}_{k-1}^a = 1$ and $v_{i}^a = a_i$ for $i = 1, \dots, K+1$. We proceed to prove (\ref{Eq.FDSlopeR1}) and for notational convenience, we switch back from $X_{i+K}^N$ to $\mathcal{L}_i^N$. Specifically, from (\ref{Eq.ScaledSlopeA}) and (\ref{eq:Xin-def}), we have that (\ref{Eq.FDSlopeR1}) is equivalent to
\begin{equation}\label{Eq.FDSlopeR2}
\left(\mathcal{L}^N_i(s_j): i \in \{ 1, \dots, \mathsf{m}_k^a \} \mbox{, } j \in \{1, \dots, m\} \right) \overset{f.d.}{\rightarrow}\left(\lambda^{\mathsf{m}_k^a}_i(s_j): i \in \{ 1, \dots, \mathsf{m}_k^a \} \mbox{, } j \in \{1, \dots, m\} \right).
\end{equation}
 
Our strategy is to use \cite[Lemma 7.1]{DZ25}, and we next list the conditions required to apply this result. Firstly, we define the measure
\begin{equation}\label{Eq.DBMMeasure}
M^{\DBM}(A) = \sum_{i = 1}^{\mathsf{m}_k^a} \sum_{j = 1}^m {\bf 1}\left\{\left(s_j, \lambda_i^{\mathsf{m}_k^a} (s_j)\right) \in A \right\},
\end{equation}
which from Proposition \ref{Prop.DBMDPP} is a determinantal point process on $\mathbb{R}^2$ with correlation kernel $\kbmk$ and reference measure $\mu_{\mathsf{S}} \times \mathrm{Leb}$. We also define the measures
\begin{equation}\label{Eq.DBMMeasureLN}
M^N(A) = \sum_{i = 1}^{\mathsf{m}_k^a} \sum_{j = 1}^m {\bf 1}\left\{\left(s_j, \mathcal{L}^N_i(s_j)\right) \in A \right\},
\end{equation}
We claim that 
\begin{equation}\label{Eq.DBMVagueConv}
M^N \Rightarrow M^{\DBM}, 
\end{equation}
\begin{equation}\label{Eq.DBMTightParticles}
\mbox{ the sequence $\{\mathcal{L}_i^N(s_j) \}_{N \geq 1}$ is tight for each $i \in \{1, \dots, \mathsf{m}_k^a\}$ and $j \in \{1, \dots, m\}$}.
\end{equation}

Combining (\ref{Eq.DBMVagueConv}) and (\ref{Eq.DBMTightParticles}), with the fact that 
$\mathcal{L}_1^N(s) > \cdots > \mathcal{L}_{\mathsf{m}_k^a}^N(s)$, which follows from the non-intersecting property of $\mathcal{A}^{a,b,c}$, see (\ref{Eq.OrdAWLE}), we see that the conditions of \cite[Lemma 7.1]{DZ25} are satisfied, and the lemma then implies (\ref{Eq.FDSlopeR2}). We have thus reduced the proof of the proposition to establishing (\ref{Eq.DBMVagueConv}) and (\ref{Eq.DBMTightParticles}). We prove these statements in the next two steps.\\

{\bf \raggedleft Step 3.} In this step, we prove (\ref{Eq.DBMVagueConv}). We first observe that if $M^{v_k^a,T_N}$ is as in (\ref{Eq.GenSlopedMeasures}), then for any $A \subseteq \mathbb{R}^2$, we have
\begin{equation}\label{Eq.DBMMeasInX}
 M^{v_k^a,T_N}(A) = \sum_{i \geq 1 } \sum_{j = 1}^m {\bf 1}\left\{\left(s_j, X_i^{N}(s_j)\right) \in A \right\}.
 \end{equation}
From our work in Section \ref{Section2.1}, we know that $M^{v_k^a,T_N}$ is a determinantal point process with correlation kernel $K^{v_k^a,T}$ as in (\ref{Eq.NewKernelSlopes}) and reference measure $\mu_{\mathsf{S}} \times \mathrm{Leb}$. From Proposition \ref{Prop.KernelConvTop}, we know that for each $s,t \in \mathsf{S}$ 
$$\lim_{N \rightarrow \infty} K^{v^a_k,T_N}(s, x; t, y) = \kbmk(s,x;t,y)$$
uniformly as $(x,y)$ vary in a compact set in $\mathbb{R}^2$. The latter verifies the conditions of \cite[Proposition 2.18]{dimitrov2024airy} from which we conclude that $M^{v_k^a,T_N}$ converges weakly to a determinantal point process on $\mathbb{R}^2$ with correlation kernel $\kbmk$ and reference measure $\mu_{\mathsf{S}} \times \mathrm{Leb}$. As the law of a determinantal point process is completely characterized by its correlation kernel and reference measure, see \cite[Proposition 2.13(3)]{dimitrov2024airy}, we conclude that
\begin{equation}\label{Eq.WeakConvPointProcessDBM}
M^{v_k^a,T_N} \Rightarrow M^{\DBM}.
\end{equation}

We next claim that for each fixed $s\in \mathsf{S}$ and each index $i \not\in \{\mathsf{M}^a_{k-1} + 1, \ldots, \mathsf{M}^a_k\}$,
\begin{equation}\label{Eq.escape}
X^N_i(s)\xrightarrow{\pr}
\begin{cases}
+\infty, & i\le k-1,\\
-\infty, & i\ge \mathsf{M}_k^a+1.
\end{cases}
\end{equation}
Indeed, by (\ref{eq:Xin-def}), we have for each $i \in \{1, \dots, J_a\}$, and $s \in \mathsf{S}$ that 
\begin{equation}\label{Eq.DBMSplitX}
X^N_i(s) = \sqrt{2T_N} \left( \xi^N_i(s) + \frac{s}{a_k} - \frac{s}{a_i} \right), \mbox{ where } \xi^N_i(s) = \frac{\mathcal{A}^{a,b,c}_i(sT_N)}{2T_N}+\frac{s}{a_i}-s^2T_N/2.
\end{equation}
From Proposition \ref{Prop.Slopes}(a), we know that $\xi^N_i(s)$ converges weakly to zero, and so (\ref{Eq.DBMSplitX}) implies (\ref{Eq.escape}) when $i \leq J_a$ as $a_i > a_k$ for $i \leq k-1$ and $a_i < a_k$ for $i \in \{\mathsf{M}_k^a+1, \dots, J_a\}$. If $J_a < \infty$, then we also have from (\ref{eq:Xin-def}) for all $i \geq 1$ that
\begin{equation}\label{Eq.DBMSplitX2}
X^N_{i+J_a}(s) = \xi^N_{i+J_a}(s) + \frac{s\sqrt{2T_N}}{a_k}-2^{-1/2}s^2T_N^{3/2} , \mbox{ where } \xi^N_{i+J_a}(s) = (2T_N)^{-1/2}\mathcal{A}^{a,b,c}_{i+J_a}(sT_N).
\end{equation}
From Proposition \ref{Prop.Slopes}(c), we know that $\xi^N_{i+J_a}(s)$ weakly converges to zero, and so (\ref{Eq.DBMSplitX2}) implies (\ref{Eq.escape}) for $i \geq J_a + 1$. \\

From (\ref{Eq.DBMMeasInX}), we know that 
\begin{equation}\label{Eq.DBMDecompMeas}
M^{v_k^a,T_N} = M^N + \hat{M}^N \mbox{, where }
\end{equation}
$$\hat{M}^N(A) = \sum_{i = 1}^{k-1} \sum_{j = 1}^m {\bf 1}\left\{\left(s_j, X_i^{N}(s_j)\right) \in A \right\} + \sum_{i \geq \mathsf{M}_k^a+1} \sum_{j = 1}^m {\bf 1}\left\{\left(s_j, X_i^{N}(s_j)\right) \in A \right\}.$$
The last equation, (\ref{Eq.escape}) and the fact that $X_1^N(s) > X_2^N(s) > \cdots$ together imply $\hat{M}^N \Rightarrow 0$ (the zero measure). Combining this with (\ref{Eq.WeakConvPointProcessDBM}) and (\ref{Eq.DBMDecompMeas}), we conclude (\ref{Eq.DBMVagueConv}).\\

{\bf \raggedleft Step 4.} In this step, we prove (\ref{Eq.DBMTightParticles}). In what follows, we fix $j \in \{1, \dots, m\}$. Our strategy is to use \cite[Lemma 7.2]{DZ25}, and we next list the conditions required to apply this result. Firstly, we define the following random measures on $\mathbb{R}$
\begin{equation}\label{Eq.DBMMeasureSlice}
M^{j, \DBM}(A) = \sum_{i = 1}^{\mathsf{m}_k^a}  {\bf 1}\left\{\lambda_i^{\mathsf{m}_k^a} (s_j) \in A \right\}, \hspace{2mm} M^{j,N}(A) = \sum_{i = 1}^{\mathsf{m}_k^a}  {\bf 1}\left\{ \mathcal{L}^N_i(s_j) \in A \right\},
\end{equation}
We claim that 
\begin{equation}\label{Eq.DBMVagueConvSlice}
M^{j,N} \Rightarrow M^{j,\DBM},
\end{equation}
\begin{equation}\label{Eq.DBMProbAtoms}
\mathbb{P}\left(M^{j,\DBM}(\mathbb{R}) \geq \mathsf{m}_k^a\right) = 1,
\end{equation}
\begin{equation}\label{Eq.DBMUpperTailMeas}
\lim_{a \rightarrow \infty}\limsup_{N \rightarrow \infty} \mathbb{P}\left( \mathcal{L}_1^N(s_j) \geq a \right) = 0.
\end{equation}
Combining (\ref{Eq.DBMVagueConvSlice}), (\ref{Eq.DBMProbAtoms}) and (\ref{Eq.DBMUpperTailMeas}) with the fact that $\mathcal{L}_1^N(s_j) > \cdots > \mathcal{L}_{\mathsf{m}_k^a}^N(s_j)$, which follows from the non-intersecting property of $\mathcal{A}^{a,b,c}$, see (\ref{Eq.OrdAWLE}), we see that the conditions of \cite[Lemma 7.2]{DZ25} are satisfied, and the lemma then implies (\ref{Eq.DBMTightParticles}). We have thus reduced the proof of the proposition to establishing (\ref{Eq.DBMVagueConvSlice}), (\ref{Eq.DBMProbAtoms}) and (\ref{Eq.DBMUpperTailMeas}).\\

Since $\lambda_i^{\mathsf{m}_k^a} (s_j) \in \mathbb{R}$ almost surely, we see that $M^{j,\DBM}(\mathbb{R}) = \mathsf{m}_k^a$ almost surely, which implies (\ref{Eq.DBMProbAtoms}). In addition, we have that (\ref{Eq.DBMUpperTailMeas}) follows from the definition of $\mathcal{L}_1^N$ in (\ref{Eq.ScaledSlopeA}) and Proposition \ref{Prop.TightnessSlope}. Consequently, we only need to show (\ref{Eq.DBMVagueConvSlice}).

Define the measures
\begin{equation}\label{Eq.DBMMeasureXSlice}
M^{j, v_k^a,T_N}(A) = \sum_{i \geq 1}{\bf 1}\left\{X^N_i(s_j) \in A \right\}.
\end{equation}
From our work in Section \ref{Section2.1} and \cite[Lemma 2.17]{dimitrov2024airy}, we know that $M^{j, v_k^a,T_N}$ is a determinantal point process on $\mathbb{R}$ with reference measure $\mathrm{Leb}$ and correlation kernel
$$K^{j, v_k^a,T_N}(x,y) = K^{v_k^a,T_N}(s_j,x; s_j,y),$$
where $K^{\rho, T}$ is as in (\ref{Eq.NewKernelSlopes}). From Proposition \ref{Prop.KernelConvTop}, we know that 
$$\lim_{N \rightarrow \infty} K^{j, v_k^a,T_N}(x,y) = K^{\DBM, \mathsf{m}_k^a}(s_j,x; s_j,y),$$
and the convergence is uniform over compact subsets of $\mathbb{R}^2$. From \cite[Proposition 2.15]{dimitrov2024airy}, we conclude that $M^{j, v_k^a,T_N}$ converges weakly to a determinantal point process on $\mathbb{R}$ with reference measure $\mathrm{Leb}$ and correlation kernel $K^{\DBM, \mathsf{m}_k^a}(s_j,x; s_j,y)$. However, by \cite[Lemma 2.17]{dimitrov2024airy} and Proposition \ref{Prop.DBMDPP}, we know that $M^{j,\DBM}$ is also a determinantal point process on $\mathbb{R}$ with reference measure $\mathrm{Leb}$ and correlation kernel $K^{\DBM, \mathsf{m}_k^a}(s_j,x; s_j,y)$. As the law of a determinantal point process is completely characterized by its correlation kernel and reference measure, see \cite[Proposition 2.13(3)]{dimitrov2024airy}, we conclude that
\begin{equation}\label{Eq.DBMMeasureXSliceConv}
M^{j, v_k^a,T_N} \Rightarrow M^{j,\DBM}.
\end{equation}

From (\ref{Eq.DBMMeasureXSlice}), we know that 
\begin{equation}\label{Eq.DBMDecompMeasSlice}
M^{j,v_k^a,T_N} = M^{j,N} + \hat{M}^{j,N} \mbox{, where }
\end{equation}
$$\hat{M}^{j,N}(A) = \sum_{i = 1}^{k-1}  {\bf 1}\left\{X_i^{N}(s_j) \in A \right\} + \sum_{i \geq \mathsf{M}_k^a+1} {\bf 1}\left\{X_i^{N}(s_j) \in A \right\}.$$
The last equation, (\ref{Eq.escape}) and the fact that $X_1^N(s_j) > X_2^N(s_j) > \cdots$ together imply $\hat{M}^{j,N} \Rightarrow 0$ (the zero measure). Combining this with (\ref{Eq.DBMMeasureXSliceConv}) and (\ref{Eq.DBMDecompMeasSlice}), we conclude (\ref{Eq.DBMVagueConvSlice}). This concludes the proof of the proposition.
\end{proof}

%
%
\subsection{Finite-dimensional convergence to the Airy line ensemble}\label{Section3.4} The goal of this section is to establish the finite-dimensional convergence of the line ensembles in Theorem \ref{Thm.ConvAiry}(a).
\begin{proposition}\label{Prop.FDFlat} Assume the same notation as in Theorem \ref{Thm.ConvAiry}(a). Then,
$$\left(\mathcal{A}^N_i(t): t \in \mathbb{R} \mbox{, }i \in \mathbb{N} \right) \overset{f.d.}{\rightarrow}\left( \mathcal{A}^{0,0,0}_i(t): t \in \mathbb{R} \mbox{, }i \in \mathbb{N} \right).$$
\end{proposition}
\begin{proof} The proof we present is similar to that of Proposition \ref{Prop.FDSlope}, so we omit some of the details. Fix $\mathsf{S} = \{s_1, \dots, s_m\} \subset \mathbb{R}$ with $s_1 < \cdots < s_m$.
From the definition of $\mathcal{A}^N$ in (\ref{Eq.ScaledFlatA}), we see that it suffices to prove
\begin{equation}\label{Eq.FDFlatR1}
\left(\mathcal{A}^{a,b,c}_{i + J_a} (s_j + T_N): i \in \mathbb{N} \mbox{, } j \in \{1, \dots, m\} \right) \overset{f.d.}{\rightarrow}\left(\mathcal{A}^{0,0,0}_i(s_j): i \in \mathbb{N} \mbox{, } j \in \{1, \dots, m\} \right).
\end{equation}

For clarity, we split the proof of (\ref{Eq.FDFlatR1}) into three steps. In Step 1, we reduce the proof of the proposition to the case when $a_1 > a_2 > \cdots > a_{J_a}$ using the monotone coupling from Proposition \ref{Prop.MonCoupling}. In Step 2, we use \cite[Proposition 2.19]{dimitrov2024airy} to reduce the proof of the proposition to two statements, see (\ref{Eq.AiryVagueConv}) and (\ref{Eq.AiryTightParticles}), which are established in Step 3. The first of these statements, (\ref{Eq.AiryVagueConv}), implies that the point processes on $\mathbb{R}^2$ formed by $(s_j, \mathcal{A}^{a,b,c}_{i + J_a} (s_j + T_N))$ converge weakly to those formed by $(s,\mathcal{A}^{0,0,0}_i(s))$. We establish this by combining our kernel convergence from Proposition \ref{Prop.KernelConvFlat}, the asymptotic slopes of the Airy wanderer line ensemble in Proposition \ref{Prop.Slopes}(a), and the weak convergence of determinantal point processes criterion in \cite[Proposition 2.18]{dimitrov2024airy}. The second of these statements, (\ref{Eq.AiryTightParticles}), implies that the sequence $\{\mathcal{A}^{a,b,c}_{i + J_a} (s_j + T_N)\}_{N \geq 1}$ is tight, and follows from Proposition \ref{Prop.Slopes}(c).\\

{\bf \raggedleft Step 1.} In this step, we assume that (\ref{Eq.FDFlatR1}) holds when $a_1 > a_2 > \cdots > a_{J_a}$ and proceed to prove it for general parameters $(a,b,c)$ as in the statement of Theorem \ref{Thm.ConvAiry}(a).

Fix $\varepsilon > 0$ sufficiently small so that $v_{J_a}^a >  J_a\varepsilon$ with the convention $v_0^a = \infty$. We define $(a^{+},b^{+},c^{+}), (a^{-},b^{-},c^-) \in \parP$ to be the same $(a,b,c)$, except that 
$$a_i^{+} = a_i + (J_a-i+1), \mbox{ and } a_i^- = a_i - i \varepsilon \mbox{ for } i = 1, \dots, J_a,$$
and observe that by our choice of $\varepsilon$ we have
\begin{equation}\label{Eq.AirySqueezePar}
\begin{split}
&a_1^+ > a_2^+ > \cdots > a_{J_a}^+ > 0 = a_{J_a+1}^+ , \hspace{2mm} a_1^- > a_2^- > \cdots > a_{J_a}^- > 0 = a_{J_a+1}^- \\
&a_i^- \leq a_i \leq a_i^+ \mbox{ for } i \in \mathbb{N}, \hspace{2mm} c^+ = c = c^- = 0, \hspace{2mm} b_i^- = b_i = b_i^+ \mbox{ for }i\in\mathbb{N}.
\end{split}
\end{equation}

From the second line in (\ref{Eq.AirySqueezePar}) and Proposition \ref{Prop.MonCoupling} applied to $A = B = 0$, we can couple $\mathcal{A}^{a^-,b^-,c^-}$ with $\mathcal{A}^{a,b,c}$, so that $\mathcal{A}_i^{a^-,b^-,c^-}(t) \leq \mathcal{A}^{a,b,c}_i(t)$ for all $(i,t) \in \mathbb{N} \times \mathbb{R}$. Similarly, we can couple $\mathcal{A}^{a^+,b^+,c^+}$ with $\mathcal{A}^{a,b,c}$, so that $\mathcal{A}_i^{a,b,c}(t) \leq \mathcal{A}^{a^+,b^+,c^+}_i(t)$ for all $(i,t) \in \mathbb{N} \times \mathbb{R}$. Our two monotone couplings imply for all $k \in \mathbb{N}$, and $r_{i,j} \in \mathbb{R}$ that 
\begin{equation}\label{Eq.AirySqueezeX}
\begin{split}
&P^{+,N} \leq \mathbb{P}\left(\mathcal{A}^{a,b,c}_{i + J_a} (s_j + T_N) \leq r_{i,j}:  i \in \{1, \dots, k\} \mbox{, } j \in \{1, \dots, m\} \right)  \leq P^{-,N} \mbox{, where }\\
&P^{\pm,N} = \mathbb{P}\left(\mathcal{A}^{a^{\pm},b^{\pm},c^{\pm}}_{i + J_a} (s_j + T_N) \leq r_{i,j}: i \in \{1, \dots, k\} \mbox{, } j \in \{1, \dots, m\} \right).
\end{split}
\end{equation}

The first line in (\ref{Eq.AirySqueezePar}) and our assumption in the beginning of the step show that
$$\lim_{N \rightarrow \infty}P^{\pm,N} = \mathbb{P}\left(\mathcal{A}^{0,0,0}_i(s_j) \leq r_{i,j}:  i \in \{ 1, \dots, k \} \mbox{, } j \in \{1, \dots, m\} \right).$$
Combining the latter with (\ref{Eq.AirySqueezeX}), we conclude (\ref{Eq.FDFlatR1}).\\

{\bf \raggedleft Step 2.} In the remainder of the proof, we assume $a_1 > a_2 > \cdots > a_{J_a}$, so that in particular $\mathsf{m}_1^a = \cdots = \mathsf{m}_{J_a}^a = 1$ and $v_{i}^a = a_i$ for $i = 1, \dots, J_a$. 
 
We proceed to prove (\ref{Eq.FDFlatR1}) and our strategy is to use \cite[Proposition 2.19]{dimitrov2024airy}. We next list the conditions required to apply this result. Firstly, we define the measure
\begin{equation}\label{Eq.AiryMeasure}
M^{\Ai}(A) = \sum_{i \geq 1} \sum_{j = 1}^m {\bf 1}\left\{\left(s_j, \mathcal{A}_i^{0,0,0} (s_j)\right) \in A \right\},
\end{equation}
which from Proposition \ref{Prop.AWLE} is a determinantal point process on $\mathbb{R}^2$ with correlation kernel $K_{0,0,0}$ as in (\ref{Eq.3BPKer}) and reference measure $\mu_{\mathsf{S}} \times \mathrm{Leb}$. We also define the measures
\begin{equation}\label{Eq.AiryMeasureN}
M^N(A) = \sum_{i \geq 1} \sum_{j = 1}^m {\bf 1}\left\{\left(s_j, \mathcal{A}^{a,b,c}_{i + J_a}(s_j + T_N)\right) \in A \right\},
\end{equation}
We claim that 
\begin{equation}\label{Eq.AiryVagueConv}
M^N \Rightarrow M^{\Ai}, 
\end{equation}
\begin{equation}\label{Eq.AiryTightParticles}
\mbox{ the sequence $\{\mathcal{A}_{i + J_a}^{a,b,c}(s_j + T_N) \}_{N \geq 1}$ is tight for each $i \in \mathbb{N}$ and $j \in \{1, \dots, m\}$}.
\end{equation}

Combining (\ref{Eq.AiryVagueConv}) and (\ref{Eq.AiryTightParticles}), with the non-intersecting property of $\mathcal{A}^{a,b,c}$, see (\ref{Eq.OrdAWLE}), we see that the conditions of \cite[Proposition 2.19]{dimitrov2024airy} are satisfied, and the proposition then implies (\ref{Eq.FDFlatR1}). We have thus reduced the proof of the proposition to establishing (\ref{Eq.AiryVagueConv}) and (\ref{Eq.AiryTightParticles}).\\

{\bf \raggedleft Step 3.} We observe that (\ref{Eq.AiryTightParticles}) follows from Proposition \ref{Prop.Slopes}(c). In the remainder, we prove (\ref{Eq.AiryVagueConv}). Define the measure
\begin{equation}\label{Eq.AiryMeasAllN}
 \tilde{M}^{N}(A) = \sum_{i \geq 1 } \sum_{j = 1}^m {\bf 1}\left\{\left(s_j, \mathcal{A}_i^{a,b,c}(s_j + T_N)\right) \in A \right\},
\end{equation}
and note that 
\begin{equation}\label{Eq.AiryMeasAllNDecomp}
 \tilde{M}^{N} = M^N + \hat{M}^N, \mbox{ where } \hat{M}^N(A) = \sum_{i = 1 }^{J_a} \sum_{j = 1}^m {\bf 1}\left\{\left(s_j, \mathcal{A}_i^{a,b,c}(s_j + T_N)\right) \in A \right\},
\end{equation}
From our work in Section \ref{Section2.2}, we know that $\tilde{M}^N$ is a determinantal point process with correlation kernel $K^{T_N}$ as in (\ref{Eq.NewKernelFlat}) and reference measure $\mu_{\mathsf{S}} \times \mathrm{Leb}$. From Proposition \ref{Prop.KernelConvFlat}, we know that for each $s,t \in \mathsf{S}$ 
$$\lim_{N \rightarrow \infty} K^{T_N}(s, x; t, y) = K_{0,0,0}(s,x;t,y)$$
uniformly as $(x,y)$ vary in a compact set in $\mathbb{R}^2$. The latter verifies the conditions of \cite[Proposition 2.18]{dimitrov2024airy} from which we conclude that 
\begin{equation}\label{Eq.WeakConvPointProcessAiry}
\tilde{M}^{N} \Rightarrow M^{\Ai}.
\end{equation}

We next claim that for each fixed $i \in \{1, \dots, J_a\}$ and $j \in \{1, \dots, m\}$, we have
\begin{equation}\label{Eq.escapeAiry}
\mathcal{A}_i^{a,b,c}(s_j + T_N) \xrightarrow{\pr} \infty.
\end{equation}
Indeed, if we set $\xi^N_i(s_j) = \frac{\mathcal{A}^{a,b,c}_i(s_j + T_N)}{s_j+T_N}+\frac{2}{a_i}-(s_j + T_N)$, we have
\begin{equation*}
\mathcal{A}_i^{a,b,c}(s_j + T_N) = (s_j +T_N)\left( \xi^N_i(s_j) - \frac{2}{a_i} + (s_j + T_N) \right),
\end{equation*}
and $\xi^N_i(s_j)\Rightarrow 0$ by Proposition \ref{Prop.Slopes}(a), which implies (\ref{Eq.escapeAiry}).

Combining (\ref{Eq.AiryMeasAllNDecomp}), (\ref{Eq.WeakConvPointProcessAiry}), and (\ref{Eq.escapeAiry}), we conclude (\ref{Eq.AiryVagueConv}).
\end{proof}

%
%
\section{Gibbsian line ensembles}\label{Section4} Propositions \ref{Prop.FDSlope} and \ref{Prop.FDFlat} establish finite-dimensional convergence of the line ensembles appearing in Theorems \ref{Thm.ConvDBM} and \ref{Thm.ConvAiry}. To strengthen this convergence to uniform convergence over compact sets, it remains to prove tightness of the corresponding sequences of ensembles. 

The main difficulty in establishing tightness lies in controlling the moduli of continuity of the ensemble curves along subsequences. As noted in Proposition \ref{Prop.AWLE}, the parabolic Airy wanderer line ensembles satisfy the Brownian Gibbs property. This property enables a strong comparison between the ensembles in Theorems \ref{Thm.ConvDBM} and \ref{Thm.ConvAiry} and ensembles of finitely many avoiding Brownian bridges. Since such Brownian bridge ensembles have well-controlled moduli of continuity, this comparison allows us to transfer analogous regularity estimates to the ensembles under consideration.

In Section \ref{Section4.1}, we introduce the terminology and objects required to carry out this strategy. Section \ref{Section4.2} summarizes several relevant results from the literature on Brownian Gibbsian ensembles. In Sections \ref{Section4.3} and \ref{Section4.4}, we establish the technical estimates needed for the proofs of Theorems \ref{Thm.ConvDBM} and \ref{Thm.ConvAiry}, which are carried out in Sections \ref{Section5} and \ref{Section6}, respectively.

%
%
\subsection{Definitions and notation for line ensembles}\label{Section4.1} In this section we recall some basic definitions and notation regarding line ensembles, mostly following \cite[Section 2]{DEA21}.

Given two integers $a \leq b$, we let $\llbracket a, b \rrbracket$ denote the set $\{a, a+1, \dots, b\}$. We also set $\llbracket a,b \rrbracket = \emptyset$ when $a > b$, $\llbracket a, \infty \rrbracket = \{a, a+1, a+2 , \dots \}$, $\llbracket - \infty, b\rrbracket = \{b, b-1, b-2, \dots\}$ and $\llbracket - \infty, \infty \rrbracket = \mathbb{Z}$. Given an interval $\Lambda \subseteq \mathbb{R}$, we endow it with the subspace topology of the usual topology on $\mathbb{R}$. We let $(C(\Lambda), \mathcal{C})$ denote the space of continuous functions $f: \Lambda \rightarrow \mathbb{R}$ with the topology of uniform convergence over compact sets, see \cite[Chapter 7, Section 46]{Munkres}, and Borel $\sigma$-algebra $\mathcal{C}$. Given a set $\Sigma \subseteq \mathbb{Z}$, we endow it with the discrete topology and denote by $\Sigma \times \Lambda$ the set of all pairs $(i,x)$ with $i \in \Sigma$ and $x \in \Lambda$ with the product topology. We also denote by $\left(C (\Sigma \times \Lambda), \mathcal{C}_{\Sigma}\right)$ the space of real-valued continuous functions on $\Sigma \times \Lambda$ with the topology of uniform convergence over compact sets and Borel $\sigma$-algebra $\mathcal{C}_{\Sigma}$. We typically take $\Sigma = \llbracket 1, N \rrbracket$ with $N \in \mathbb{N} \cup \{\infty\}$. We now define the notion of a line ensemble.
\begin{definition}\label{Def.LineEnsembles}
Let $\Sigma \subseteq \mathbb{Z}$ and $\Lambda \subseteq \mathbb{R}$ be an interval. A {\em $\Sigma$-indexed line ensemble $\mathcal{L}$} is a random variable defined on a probability space $(\Omega, \mathcal{F}, \mathbb{P})$ that takes values in $\left(C (\Sigma \times \Lambda), \mathcal{C}_{\Sigma}\right)$. Intuitively, $\mathcal{L}$ is a collection of random continuous curves (sometimes referred to as {\em lines}), indexed by $\Sigma$, each of which maps $\Lambda$ into $\mathbb{R}$. We will often slightly abuse notation and write $\mathcal{L}: \Sigma \times \Lambda \rightarrow \mathbb{R}$, even though it is not $\mathcal{L}$ which is such a function, but $\mathcal{L}(\omega)$ for every $\omega \in \Omega$. For $i \in \Sigma$ we write $\mathcal{L}_i(\omega) = (\mathcal{L}(\omega))(i, \cdot)$ for the curve of index $i$ and note that the latter is a map $\mathcal{L}_i: \Omega \rightarrow C(\Lambda)$, which is $\mathcal{F}/\mathcal{C}$ measurable. If $a,b \in \Lambda$ satisfy $a \leq b$, we let $\mathcal{L}_i[a,b]$ denote the restriction of $\mathcal{L}_i$ to $[a,b]$. We also say that a line ensemble $\mathcal{L}$ is {\em non-intersecting} if $\mathbb{P}$-almost surely $\mathcal{L}_i(t) > \mathcal{L}_j(t)$ for all $i < j$ and $t \in \Lambda$.
\end{definition}
\begin{remark}\label{Rem.Polish} As shown in \cite[Lemma 2.2]{DEA21}, we have that $C(\Sigma \times \Lambda)$ is a Polish space, and so a line ensemble $\mathcal{L}$ is just a random element in $C(\Sigma \times \Lambda)$ in the sense of \cite[Section 3]{Billing}.
\end{remark}

If $W_t$ denotes a standard one-dimensional Brownian motion, then the process
$$\tilde{B}(t) =  W_t - t W_1, \hspace{5mm} 0 \leq t \leq 1,$$
is called a {\em Brownian bridge (from $\tilde{B}(0) = 0$ to $\tilde{B}(1) = 0$).} Given $a,b,x,y \in \mathbb{R}$ with $a < b$, we define a random variable on $(C([a,b]), \mathcal{C})$ through
\begin{equation}\label{Eq.BBDef}
B(t) = (b-a)^{1/2} \cdot \tilde{B} \left( \frac{t - a}{b-a} \right) + \left(\frac{b-t}{b-a} \right) \cdot x + \left( \frac{t- a}{b-a}\right) \cdot y, 
\end{equation}
and refer to the law of this random variable as a {\em Brownian bridge (from $B(a) = x$ to $B(b) = y$)}. If $\vec{x}, \vec{y} \in \mathbb{R}^k$, we let $\mathbb{P}_{\mathrm{free}}^{a,b,\vec{x}, \vec{y}}$ denote the law of $k$ independent Brownian bridges $\{B_i:[a,b] \rightarrow \mathbb{R}\}_{i = 1}^k$ from $B_i(a) = x_i$ to $B_i(b) = y_i$. We write $\mathbb{E}_{\mathrm{free}}^{a,b,\vec{x}, \vec{y}}$ for the expectation with respect to this measure.

We end this section by recalling the Brownian Gibbs property from \cite[Definition 2.2]{CorHamA}, see also \cite[Definition 2.8]{DEA21}. To state it, we require the notion of an $(f,g)$-avoiding Brownian bridge ensemble from \cite[Definition 2.7]{DEA21}. We recall that $\weyl_k$ and $\weylc_k$ are the open and closed Weyl chambers in $\mathbb{R}^{k}$, see (\ref{Eq.WeylChamber}).
\begin{definition}\label{Def.fgAvoidingBE}
Fix $k \in \mathbb{N}$, $\vec{x}, \vec{y} \in \weyl_k$, $a,b \in \mathbb{R}$ with $a < b$, and two continuous functions $f: [a,b] \rightarrow (-\infty, \infty]$ and $g: [a,b] \rightarrow [-\infty,\infty)$. This means that either $f \in C([a,b])$ or $f \equiv \infty$, and similarly for $g$. In addition, we assume that $f(t) > g(t)$ for $t \in[a,b]$, $f(a) > x_1$, $f(b) > y_1$, $g(a) < x_k$, $g(b) < y_k$. With this data we let $\mathbb{P}_{\mathrm{avoid}}^{a,b,\vec{x},\vec{y}, f,g}$ denote the law of $k$ independent Brownian bridges $\{B_i\}_{i =1}^k$ with law $\mathbb{P}_{\mathrm{free}}^{a,b,\vec{x}, \vec{y}}$, conditioned on the event 
\begin{equation}\label{Eq.DefEavoid}
E^{f,g}_{\mathrm{avoid}} = \{f(t) > B_1(t) > B_2(t) > \cdots > B_k(t) > g(t) \mbox{ for all } t\in [a,b]\}.
\end{equation}
As explained in \cite[Definition 2.7]{DEA21}, we have that $E^{f,g}_{\mathrm{avoid}}$ is measurable and $\mathbb{P}_{\mathrm{free}}^{a,b,\vec{x}, \vec{y}}(E^{f,g}_{\mathrm{avoid}}) > 0$, so that the law $\mathbb{P}_{\mathrm{avoid}}^{a,b,\vec{x},\vec{y}, f,g}$ on $C(\llbracket 1, k \rrbracket \times [a,b])$ is well-defined. The expectation with respect to $\mathbb{P}_{\mathrm{avoid}}^{a,b,\vec{x},\vec{y}, f,g}$ is denoted by $\mathbb{E}_{\mathrm{avoid}}^{a,b,\vec{x},\vec{y}, f,g}$. When $f = \infty$ and $g = -\infty$, we drop them from the notation and simply write $\mathbb{P}_{\mathrm{avoid}}^{a,b,\vec{x},\vec{y}}$ and $\mathbb{E}_{\mathrm{avoid}}^{a,b,\vec{x},\vec{y}}$.
\end{definition}

We will require one further class of measures for our arguments later, which we still denote by $\mathbb{P}_{\mathrm{avoid}}^{a,b,\vec{x},\vec{y}, \infty,g}$. This class corresponds to the above definition with $f \equiv \infty$ and 
\begin{equation}\label{Eq.Spike}
g(t) = \begin{cases} g_0 &\mbox{ if } t = t_0, \\ -\infty &\mbox{ for } t \in [a,b] \setminus \{t_0\}, \end{cases}
\end{equation}
where $g_0 \in \mathbb{R}$ and $t_0 \in (a,b)$. Note that in this case the measure $\mathbb{P}_{\mathrm{avoid}}^{a,b,\vec{x},\vec{y}, \infty ,g}$ is well-defined. Indeed, the set $E^{\infty,g}_{\mathrm{avoid}}$ is measurable, since it is the inverse image of an open set in $C(\llbracket 1, k \rrbracket \times [a,b])$, and has positive probability. The latter follows from the inclusion $E^{\infty,\tilde{g}}_{\mathrm{avoid}} \subseteq E^{\infty,g}_{\mathrm{avoid}}$, where $\tilde{g}$ is defined by
$$\tilde{g}(a) = x_k - 1, \hspace{2mm} \tilde{g}(b) = y_k - 1, \hspace{2mm} \tilde{g}(t_0) = g_0,$$
and $\tilde{g}$ is linear on $[a, t_0]$ and $[t_0, b]$.

\begin{definition}\label{Def.BGPVanilla}
An $\mathbb{N}$-indexed line ensemble $\mathcal{L} = \{\mathcal{L}_i\}_{i \geq 1}$ on an interval $\Lambda \subseteq \mathbb{R}$ is said to have the {\em Brownian Gibbs property}, if it is non-intersecting, and the following holds for all $[a,b] \subseteq \Lambda$ and $1 \leq k_1 \leq k_2$. If we set $K = \llbracket k_1, k_2 \rrbracket$, then for any bounded Borel-measurable function $F: C(K \times [a,b]) \rightarrow \mathbb{R}$, we have $\mathbb{P}$-almost surely
\begin{equation}\label{BGPTower}
\mathbb{E} \left[ F\left(\mathcal{L}|_{K \times [a,b]} \right)  {\big \vert} \mathcal{F}_{\mathrm{ext}} (K \times (a,b))  \right] =\mathbb{E}_{\mathrm{avoid}}^{a,b, \vec{x}, \vec{y}, f, g} \bigl[ F(\tilde{\mathcal{Q}}) \bigr].
\end{equation}
On the left side of (\ref{BGPTower}), we have that
$$\mathcal{F}_{\mathrm{ext}} (K \times (a,b)) = \sigma \left \{ \mathcal{L}_i(s): (i,s) \in (\mathbb{N} \times \Lambda) \setminus (K \times (a,b)) \right\},$$
and $ \mathcal{L}|_{K \times [a,b]}$ is the restriction of $\mathcal{L}$ to the set $K \times [a,b]$. On the right side of (\ref{BGPTower}), we have $\vec{x} = (\mathcal{L}_{k_1}(a), \dots, \mathcal{L}_{k_2}(a))$, $\vec{y} = (\mathcal{L}_{k_1}(b), \dots, \mathcal{L}_{k_2}(b))$, $f = \mathcal{L}_{k_1 - 1}[a,b]$ with the convention that $f = \infty$ if $k_1 = 1$, and $g = \mathcal{L}_{k_2 +1}[a,b]$. In addition, $\mathcal{Q} = \{\mathcal{Q}_i\}_{i = 1}^{k_2 - k_1 + 1}$ has law $\mathbb{P}_{\mathrm{avoid}}^{a,b, \vec{x}, \vec{y}, f, g}$ as in Definition \ref{Def.fgAvoidingBE}, and $\tilde{\mathcal{Q}} = \{\tilde{\mathcal{Q}}_{i}\}_{i = k_1}^{k_2}$ satisfies $\tilde{\mathcal{Q}}_i = \mathcal{Q}_{i - k_1 + 1}$.
\end{definition}

%
%
\subsection{Results from previous papers}\label{Section4.2} The goal of this section is to recall a few known results about line ensembles with laws $\mathbb{P}_{\mathrm{avoid}}^{a,b, \vec{x}, \vec{y}, f, g}$ as in Definition \ref{Def.fgAvoidingBE}.

The following result gives a monotone coupling of the measures $\mathbb{P}_{\mathrm{avoid}}^{a,b, \vec{x}, \vec{y}, f, g}$ in their boundary data.  
\begin{lemma}\label{Lem.MonCoup} Fix $k \in \mathbb{N}$, $a < b$, and two continuous functions $g^{\mathsf{t}}, g^{\mathsf{b}}: [a,b] \rightarrow [-\infty, \infty)$ such that $g^{\mathsf{t}}(t) \geq g^{\mathsf{b}}(t)$ for all $t \in [a,b]$. We also fix $\vec{x}, \vec{y}, \vec{x}', \vec{y}' \in \weyl_k$ such that $g^{\mathsf{b}}(a) < x_k$, $g^{\mathsf{b}}(b) < y_k$, $g^{\mathsf{t}}(a) < x_k'$, $g^{\mathsf{t}}(b) < y_k'$, and $x_i \leq x_i'$, $y_i \leq y_i'$ for $i \in \llbracket 1, k \rrbracket$. Then, there exists a probability space $(\Omega, \mathcal{F}, \mathbb{P})$, which supports two $\llbracket 1, k \rrbracket$-indexed line ensembles $\mathcal{L}^{\mathsf{t}}$ and $\mathcal{L}^{\mathsf{b}}$ on $[a,b]$, such that the law of $\mathcal{L}^{\mathsf{t}}$ {\big (}resp. $\mathcal{L}^{\mathsf{b}}${\big )} under $\mathbb{P}$ is $\mathbb{P}_{\mathrm{avoid}}^{a,b, \vec{x}', \vec{y}', \infty, g^{\mathsf{t}}}$ {\big (}resp. $\mathbb{P}_{\mathrm{avoid}}^{a,b, \vec{x}, \vec{y}, \infty, g^{\mathsf{b}}}${\big )}, and such that $\mathbb{P}$-almost surely we have $\mathcal{L}_i^{\mathsf{t}}(t) \geq \mathcal{L}^{\mathsf{b}}_i(t)$ for all $(i,t) \in \llbracket 1, k \rrbracket \times [a,b]$.

In addition, the conclusion of the lemma continues to hold if $g^{\mathsf{b}}$ is not continuous, but instead is of the form (\ref{Eq.Spike}), provided that $g^{\mathsf{t}}(t_0) \geq g^{\mathsf{b}}(t_0)$.
\end{lemma}
\begin{proof} When $g^{\mathsf{t}}, g^{\mathsf{b}}$ are both continuous, the result can be found in \cite[Lemma A.6]{DimMat}, and its proof is based on a Markov chain Monte Carlo argument originating in \cite[Lemmas 2.6 and 2.7]{CorHamA}. 

When $g^{\mathsf{b}}$ is not continuous, but instead is of the form (\ref{Eq.Spike}), one may establish \cite[Lemma A.5]{DimMat} for $f^n = \infty$ and $g^n = g^{\mathsf{b}}$ using exactly the same arguments. The only difference is that the identity $\mathbb{P}(E_1 \cup E_2) = 1$ near the end of the proof of \cite[Lemma A.5]{DimMat} no longer follows solely from \cite[Lemma 2.2]{DimMat}, but instead from a combination of that lemma with the fact that $\mathbb{P}(\mathcal{L}_{k}(t_0) = g^{\mathsf{b}}(t_0)) = 0$. In the present notation, this statement is equivalent to
$$\mathbb{P}_{\mathrm{free}}^{a',b',\vec{x}, \vec{y}}\left(\mathcal{Q}_k(t_0) = g^{\mathsf{b}}(t_0) \right) = 0,$$
which holds as $\mathcal{Q}_k(t_0)$ is a nondegenerate normal variable and hence has a diffuse distribution. 

Once \cite[Lemma A.5]{DimMat} is established for $f^n = \infty$ and $g^n = g^{\mathsf{b}}$, the proof of \cite[Lemma A.6]{DimMat} can be repeated verbatim to obtain the present lemma when $g^{\mathsf{b}}$ is of the form (\ref{Eq.Spike}).
\end{proof}

The second result we require shows that a sequence of $(\infty,-\infty)$-avoiding Brownian bridge ensembles converges weakly if their boundary data converge. 
\begin{lemma}\label{Lem.BridgeEnsemblesCty} Fix $k \in \mathbb{N}$, $\vec{x}, \vec{y} \in \weylc_k$, and $a,b \in \mathbb{R}$ with $a < b$. Suppose that $\vec{x}\,^n, \vec{y}\,^n \in \weyl_k$ satisfy $\lim_{n \rightarrow \infty} \vec{x}\,^n = \vec{x}$, $\lim_{n \rightarrow \infty} \vec{y}\,^n = \vec{y}$. Then, as $n \rightarrow \infty$, the measures $\mathbb{P}_{\mathrm{avoid}}^{a,b,\vec{x}\,^n,\vec{y}\,^n}$ converge weakly to a probability measure on $C(\llbracket 1, k \rrbracket \times [a,b])$, which we denote by $\mathbb{P}_{\mathrm{avoid}}^{a,b,\vec{x},\vec{y}}$.
\end{lemma}
\begin{proof} The statement is a special case of \cite[Lemma 4.17]{DSY26}, corresponding to $f_n = f = \infty$ and $g_n = g = -\infty$.
\end{proof}

We end this section with a lemma, which shows that the Brownian Gibbs property is preserved under affine shifts and Brownian scaling.
\begin{lemma}\label{Lem.Affine} Let $\mathcal{L}$ be an $\mathbb{N}$-indexed line ensemble on $\mathbb{R}$ satisfying the Brownian Gibbs property. Then, for any $\lambda > 0$ and any $a, h, v \in \mathbb{R}$, the line ensemble $\tilde{\mathcal{L}} = \{\tilde{\mathcal{L}}_i\}_{i \geq 1}$ defined by
$$\tilde{\mathcal{L}}_i(t) = \lambda^{-1}\mathcal{L}_i(\lambda^2 t + h) + v + at, \mbox{ for } i \geq 1, t \in \mathbb{R},$$
also satisfies the Brownian Gibbs property.
\end{lemma}
\begin{proof} When $a = 0$, the result follows from \cite[Lemma 2.5]{DS25} with $N = \infty$ and $A_i = 0$ for $i \geq 1$. The general case $a \in \mathbb{R}$ then follows from \cite[Lemma 2.8]{DS25} with $N = \infty$ and $A = 0$.
\end{proof}

%
%
\subsection{Maximum and minimum estimates for Brownian line ensembles}\label{Section4.3} The goal of this section is to establish a few results about the maxima and minima of line ensembles with laws $\mathbb{P}_{\mathrm{avoid}}^{a,b, \vec{x}, \vec{y}, f,g}$ as in Definition \ref{Def.fgAvoidingBE}.

The following result shows that a line ensemble with law $\mathbb{P}_{\mathrm{avoid}}^{a,b, \vec{x}, \vec{y}}$ is likely to stay close to zero, if the boundary data $\vec{x}, \vec{y}$ are close to zero.
\begin{lemma}\label{Lem.StayInCorridor} Fix $k \in \mathbb{N}$, $a,b \in \mathbb{R}$ with $a < b$, and $M^{\mathrm{side}} > 0$. For any $\epsilon > 0$, there exists $A > 0$, depending on $k$, $a$, $b$, $M^{\mathrm{side}}$, and $\epsilon$, such that the following holds. If $\vec{x}, \vec{y} \in \weyl_k$, and $|x_i|, |y_i| \leq M^{\mathrm{side}}$ for $i \in \llbracket 1, k \rrbracket$, then 
\begin{equation}\label{Eq.StayInCorridor}
\mathbb{P}_{\mathrm{avoid}}^{a,b, \vec{x}, \vec{y}}\left( |\mathcal{Q}_i(t)| \geq A \mbox{ for some } (i,t) \in \llbracket 1,k \rrbracket \times [a,b]  \right) < \epsilon.
\end{equation}
\end{lemma}
\begin{proof} Suppose for the sake of contradiction that no such $A > 0$ exists. Then, for each $A_n = n$, there exist $\vec{x}\,^n, \vec{y}\,^n \in \weyl_k$ with $|x_i^n|, |y_i^n| \leq M^{\mathrm{side}}$, such that 
\begin{equation}\label{Eq.StayInCorridorR2}
\mathbb{P}_{\mathrm{avoid}}^{a,b, \vec{x}\,^n, \vec{y}\,^n}\left( |\mathcal{Q}_i(t)  | \geq A_n \mbox{ for some } (i,t) \in \llbracket 1,k \rrbracket \times [a,b]  \right)  \geq \epsilon.
\end{equation}

By possibly passing to a subsequence, we may assume that $\vec{x}\,^n \rightarrow \vec{x}$ and $\vec{y}\,^n \rightarrow \vec{y}$ for some $\vec{x}, \vec{y} \in \weylc_k$. Consequently, for any $A_0 > 0$
\begin{equation}\label{Eq.StayInCorridorR3}
\begin{split}
&\mathbb{P}_{\mathrm{avoid}}^{a,b, \vec{x}, \vec{y}}\left( |\mathcal{Q}_i(t)  | \geq A_0 \mbox{ for some } (i,t) \in \llbracket 1,k \rrbracket \times [a,b]  \right) \\
&\geq \limsup_{n \rightarrow \infty} \mathbb{P}_{\mathrm{avoid}}^{a,b, \vec{x}\,^n, \vec{y}\,^n}\left( |\mathcal{Q}_i(t)  | \geq A_0 \mbox{ for some } (i,t) \in \llbracket 1,k \rrbracket \times [a,b]   \right)\\
&\geq \limsup_{n \rightarrow \infty} \mathbb{P}_{\mathrm{avoid}}^{a,b, \vec{x}\,^n, \vec{y}\,^n }\left( |\mathcal{Q}_i(t)  | \geq A_n \mbox{ for some } (i,t) \in \llbracket 1,k \rrbracket \times [a,b]   \right)\geq \epsilon,
\end{split} 
\end{equation} 
where the first inequality follows from Lemma \ref{Lem.BridgeEnsemblesCty} and the Portmanteau theorem, the second one from $A_n \uparrow \infty$, and the third from (\ref{Eq.StayInCorridorR2}). By continuity of $\mathcal{Q}_i$, we have 
$$\sup_{i \in \llbracket 1, k \rrbracket, t \in [a,b]} |\mathcal{Q}_i(t)  | < \infty, \quad \mathbb{P}_{\mathrm{avoid}}^{a,b, \vec{x}, \vec{y}}\mbox{-a.s.}$$
By the bounded convergence theorem, we conclude
$$ \lim_{A_0 \rightarrow \infty} \mathbb{P}_{\mathrm{avoid}}^{a,b, \vec{x}, \vec{y}}\left( |\mathcal{Q}_i(t)  | \geq A_0 \mbox{ for some } (i,t) \in \llbracket 1,k \rrbracket \times [a,b]  \right)  = 0,$$
which gives the desired contradiction with (\ref{Eq.StayInCorridorR3}).
\end{proof}

The following result shows that a line ensemble with law $\mathbb{P}_{\mathrm{avoid}}^{a,b, \vec{x}, \vec{y}, \infty, g}$ is unlikely to be very high on $[a,b]$ if the boundary data $\vec{x}, \vec{y}$ are close to zero, and the function $g$ is not too high on $[a,b]$.
\begin{lemma}\label{Lem.NoBigMaxBLE} Fix $k \in \mathbb{N}$, $a,b \in \mathbb{R}$ with $a < b$, and $M^{\mathrm{side}}, M^{\mathrm{bot}} > 0$. For any $\epsilon > 0$, there exists $A> 0$, depending on $k$, $a$, $b$, $M^{\mathrm{side}}$, $M^{\mathrm{bot}}$ and $\epsilon$, such that the following holds. If $\vec{x}, \vec{y} \in \weyl_k$ satisfy $|x_i|, |y_i| \leq M^{\mathrm{side}}$ for $i \in \llbracket 1, k \rrbracket$, and $g: [a, b] \rightarrow [-\infty, \infty)$ is continuous and satisfies $x_k > g(a)$, $y_k > g(b)$, and $g(t) \leq M^{\mathrm{bot}}$ for all $t \in [a, b]$, then 
\begin{equation}\label{Eq.NoBigMaxBLE}
\mathbb{P}_{\mathrm{avoid}}^{a,b, \vec{x}, \vec{y}, \infty, g}\left( \mathcal{Q}_i(t)  \geq A  \mbox{ for some } (i,t) \in \llbracket 1,k \rrbracket \times [a,b]  \right) < \epsilon.
\end{equation}
\end{lemma}
\begin{proof} Let $\epsilon_0 \in (0,1)$ be sufficiently small so that $\epsilon_0(1-\epsilon_0)^{-1} < \epsilon$. From Lemma \ref{Lem.StayInCorridor} applied to $k,a,b$ as in the present setup, $M^{\mathrm{side}} = k$ and $\epsilon = \epsilon_0$, we can find $A_1 > 0$, such that if $\vec{x}, \vec{y} \in \weyl_k$, and $|x_i|, |y_i| \leq k$ for $i \in \llbracket 1, k \rrbracket$, then 
\begin{equation}\label{Eq.NBMStayInCorridor}
\mathbb{P}_{\mathrm{avoid}}^{a,b, \vec{x}, \vec{y}}\left( |\mathcal{Q}_i(t)| \geq A_1  \mbox{ for some } (i,t) \in \llbracket 1,k \rrbracket \times [a,b]  \right) < \epsilon_0.
\end{equation}

Define the function $G_{g}: C( \llbracket 1, k \rrbracket \times [a, b] ) \rightarrow [0,1]$ through
$$G_{g}\left( \{h_i\}_{i = 1}^{k} \right) = {\bf 1} \left\{ h_k(t) > g(t)  \mbox{ for all } t \in [a,b] \right\},$$ 
and for $r > 0$ the event
$$F_r = \{\mathcal{Q}_i(t) \geq r  \mbox{ for some } (i,t) \in \llbracket 1,k \rrbracket \times [a,b] \}.$$
Define $C = A_1 + k + M^{\mathrm{side}} + M^{\mathrm{bot}}$, $\vec{u}, \vec{v}, \vec{x}\,', \vec{y}\,' \in \weyl_k$ by $u_i = v_i =  C -i$ and $x_i' = y_i' = -i$ for $i \in \llbracket 1, k \rrbracket$, and note that 
\begin{equation}\label{Eq.NBMBoundaryOrdered}
x_i \leq u_i, \hspace{2mm} y_i \leq v_i \mbox{ for } i \in \llbracket 1, k \rrbracket.
\end{equation}

Observe that we have the following tower of inequalities
\begin{equation}\label{Eq.NBMTower}
\begin{split}
&\mathbb{P}_{\mathrm{avoid}}^{a,b, \vec{x}, \vec{y}, \infty, g}\left(F_{C+A_1} \right) \leq  \mathbb{P}_{\mathrm{avoid}}^{a,b, \vec{u}, \vec{v}, \infty, g}\left(F_{C+A_1} \right) = \frac{\mathbb{E}_{\mathrm{avoid}}^{a,b, \vec{u}, \vec{v}}\left[{\bf 1}_{F_{C+A_1}} \cdot G_{g}(\mathcal{Q})\right]}{\mathbb{E}_{\mathrm{avoid}}^{a,b, \vec{u}, \vec{v}}\left[ G_g(\mathcal{Q}) \right]} \\
& \leq \frac{\mathbb{E}_{\mathrm{avoid}}^{a,b, \vec{u}, \vec{v}}\left[{\bf 1}_{F_{C+A_1}} \right]}{\mathbb{P}_{\mathrm{avoid}}^{a,b, \vec{u}, \vec{v}}\left( \mathcal{Q}_k(t) > M^{\mathrm{bot}} \mbox{ for all } t \in [a,b]  \right)} = \frac{\mathbb{P}_{\mathrm{avoid}}^{a,b, \vec{x}', \vec{y}'}\left(F_{A_1} \right)}{\mathbb{P}_{\mathrm{avoid}}^{a,b, \vec{x}', \vec{y}'}\left( \mathcal{Q}_k(t) > M^{\mathrm{bot}} - C \mbox{ for all } t \in [a,b]  \right)} \\
& \leq \frac{\epsilon_0}{\mathbb{P}_{\mathrm{avoid}}^{a,b, \vec{x}', \vec{y}'}\left( \mathcal{Q}_k(t) > -A_1 \mbox{ for all } t \in [a,b]  \right)} \leq \frac{\epsilon_0}{1 - \epsilon_0} < \epsilon.
\end{split} 
\end{equation}
Indeed, the first inequality follows from the monotone coupling in Lemma \ref{Lem.MonCoup} and (\ref{Eq.NBMBoundaryOrdered}). The equality on the first line follows from Definition \ref{Def.fgAvoidingBE}. In going from the first to the second line, we used that $G_g(\mathcal{Q}) \in [0,1]$ and that $g(t) \leq M^{\mathrm{bot}}$ for all $t \in [a,b]$. The equality on the second line follows by a vertical shift of the boundary data by $C$. In going from the second to the third line, we used the definition of $A_1$ and (\ref{Eq.NBMStayInCorridor}), as well as the fact that $M^{\mathrm{bot}} - C \leq - A_1$. The second inequality on the third line follows again from (\ref{Eq.NBMStayInCorridor}), and the last one by the definition of $\epsilon_0$.

Equation (\ref{Eq.NBMTower}) implies (\ref{Eq.NoBigMaxBLE}) with $A = A_1 + C$.
\end{proof}

We end this section with a lemma, which roughly shows that a line ensemble with law $\mathbb{P}_{\mathrm{avoid}}^{a,b, \vec{x}, \vec{y}, \infty, g}$ is not too low on $[a,b]$ if its boundary data $\vec{x}, \vec{y}$ are not too low.
\begin{lemma}\label{Lem.NotTooLow} Fix $k \in \mathbb{N}$, $a,b \in \mathbb{R}$ with $a < b$, and $P_1, P_2 \in \mathbb{R}$. For any $\epsilon > 0$, there exists $A > 0$, depending on $k$ and $\epsilon$, such that the following holds. If $\vec{x}, \vec{y} \in \weyl_k$, $x_k  \geq P_1 $, $y_k \geq P_2$, and $g: [a, b] \rightarrow [-\infty, \infty)$ is continuous and satisfies $x_k > g(a)$, $y_k > g(b)$, then 
\begin{equation}\label{Eq.NotTooLow}
\mathbb{P}_{\mathrm{avoid}}^{a,b, \vec{x}, \vec{y}, \infty, g}\left( \mathcal{Q}_k(t)  \leq P_1  \cdot \frac{b - t}{b - a} + P_2  \cdot \frac{t - a}{b - a} - A \sqrt{b-a}  \mbox{ for some $t \in [a,b]$} \right) < \epsilon.
\end{equation}
\end{lemma}
\begin{proof} From Lemma \ref{Lem.MonCoup} it suffices to prove (\ref{Eq.NotTooLow}) when $g = -\infty$, which we assume in the sequel. From Definition \ref{Def.fgAvoidingBE}, we have that if we define
$$\tilde{\mathcal{Q}}_i(s) = (b-a)^{-1/2}\left(\mathcal{Q}_i(a + s(b-a)) - P_1  \cdot (1-s) - P_2  \cdot s \right) \mbox{ for } i \in \llbracket 1, k \rrbracket \mbox{ and } s \in [0,1], $$
then $\tilde{\mathcal{Q}} = \{\tilde{\mathcal{Q}}_i\}_{i = 1}^k$ has law $\mathbb{P}_{\mathrm{avoid}}^{0,1, \vec{u}, \vec{v}}$, where $u_i = (b-a)^{-1/2} (x_i-P_1)$, $v_i = (b-a)^{-1/2}(y_i-P_2)$ for $i \in \llbracket 1, k \rrbracket$. By rephrasing (\ref{Eq.NotTooLow}) in terms of $\tilde{\mathcal{Q}}$, and then dropping the tilde from the notation, we see that it suffices to show that there exists $A > 0$, such that if $u_k, v_k \geq 0$, then 
\begin{equation}\label{Eq.NotTooLowR1}
\mathbb{P}_{\mathrm{avoid}}^{0,1, \vec{u}, \vec{v}}\left( \inf_{t \in [0,1]}\mathcal{Q}_k(t) \leq   - A \right) < \epsilon.
\end{equation}

Let $\epsilon_0 \in (0,1)$ be sufficiently small so that $\epsilon_0 (1- \epsilon_0)^{-k} < \epsilon$. From Lemma \ref{Lem.StayInCorridor}, applied to $k = 1$, $M^{\mathrm{side}} = 1$, $\epsilon = \epsilon_0$, and $a = 0$, $b = 1$, we can find $A_1 > 0$, such that 
\begin{equation}\label{Eq.FreeBridgeInCorridor}
\mathbb{P}_{\mathrm{free}}^{0,1, 0, 0}\left( \sup_{t \in [0,1]}\left|B(t)\right| < A_1 \right) > 1- \epsilon_0.
\end{equation}
Note that when $k = 1$ the law $\mathbb{P}_{\mathrm{avoid}}^{0,1, 0, 0}$ is precisely that of a Brownian bridge $B$ as in (\ref{Eq.BBDef}) from $B(0) = 0$ to $B(1) = 0$. In the rest of the proof we put $A = (2k+1) A_1$ and proceed to prove (\ref{Eq.NotTooLowR1}) for this choice of $A$. \\

We define $\vec{u}\,^{\mathrm{new}}=(u^{\mathrm{new}}_1,\dots,u^{\mathrm{new}}_{k})\in\weyl_{k}$ and $\vec{v}\,^{\mathrm{new}}=(v^{\mathrm{new}}_1,\dots,v^{\mathrm{new}}_{k})\in\weyl_{k}$ through
$$u^{\mathrm{new}}_i =  - 2i A_1, \quad v^{\mathrm{new}}_i =  - 2i A_1, \mbox{ for } i \in \llbracket 1, k \rrbracket.$$
Observe that since $u_k, v_k \geq 0$, we have $u^{\mathrm{new}}_i \leq u_i$, $v^{\mathrm{new}}_i \leq v_i$ for $i \in \llbracket 1, k \rrbracket$, and so by Lemma \ref{Lem.MonCoup}
\begin{equation}\label{Eq.UVNew}
\mathbb{P}_{\mathrm{avoid}}^{0,1, \vec{u}, \vec{v}}\left( \inf_{t \in [0,1]}\mathcal{Q}_k(t) \leq   - A \right)  \leq \mathbb{P}_{\mathrm{avoid}}^{0,1, \vec{u}\,^{\mathrm{new}}, \vec{v}\,^{\mathrm{new}}}\left( \inf_{t \in [0,1]}\mathcal{Q}_k(t) \leq   - A  \right).
\end{equation}
In addition, we note that if $E_{\mathrm{avoid}}$ is as in (\ref{Eq.DefEavoid}) for $f = \infty$ and $g = -\infty$, then by Definition \ref{Def.fgAvoidingBE}
\begin{equation}\label{Eq.AvoidInFree}
\begin{split}
 &\mathbb{P}_{\mathrm{avoid}}^{0,1, \vec{u}\,^{\mathrm{new}}, \vec{v}\,^{\mathrm{new}}}\left( \inf_{t \in [0,1]} \mathcal{Q}_k(t) \leq -A \right) \leq \frac{\mathbb{P}_{\mathrm{free}}^{0,1, \vec{u}\,^{\mathrm{new}}, \vec{v}\,^{\mathrm{new}}} \left( \inf_{t \in [0,1]} B_k(t) \leq   - A \right)}{\mathbb{P}_{\mathrm{free}}^{0,1, \vec{u}\,^{\mathrm{new}}, \vec{v}\,^{\mathrm{new}}}(E_{\mathrm{avoid}})}.
\end{split}
\end{equation}

From the inclusion 
\begin{equation*}
\cap_{i = 1}^k E_i \subseteq E_{\mathrm{avoid}}, \mbox{ where } E_i = \left\{ \sup_{t \in [0,1]}\left|B_i(t) + 2i A_1 \right| < A_1 \right\},
\end{equation*}
and independence, we conclude 
\begin{equation}\label{Eq.EavoidLB}
\begin{split}
&\mathbb{P}_{\mathrm{free}}^{0,1, \vec{u}\,^{\mathrm{new}}, \vec{v}\,^{\mathrm{new}}}(E_{\mathrm{avoid}})  \geq \prod_{i = 1}^k \mathbb{P}_{\mathrm{free}}^{0,1,u^{\mathrm{new}}_i, v^{\mathrm{new}}_i } \left(\sup_{t \in [0,1]}\left|B(t)  + 2i A_1 \right| < A_1 \right) > (1- \epsilon_0)^k.
\end{split}
\end{equation}
In deriving the last inequality in (\ref{Eq.EavoidLB}), we used that each of the probabilities on the last line is equal to the one in (\ref{Eq.FreeBridgeInCorridor}), once we vertically translate by $2iA_1$. 

Lastly, we observe that since $A = (2k+1)A_1$, we have
\begin{equation}\label{Eq.LowFreeUB}
\begin{split}
&\mathbb{P}_{\mathrm{free}}^{0,1, \vec{u}\,^{\mathrm{new}}, \vec{v}\,^{\mathrm{new}}} \left(\inf_{t \in [0,1]}B_k(t) \leq  - A \right)  = \mathbb{P}_{\mathrm{free}}^{0,1,  - 2kA_1,- 2kA_1} \left(\inf_{t \in [0,1]}B(t) \leq   - (2k+1)A_1 \right) \\
& = \mathbb{P}_{\mathrm{free}}^{0,1, 0, 0} \left(\inf_{t \in [0,1]} B(t) \leq   - A_1 \right) < \epsilon_0,
\end{split}
\end{equation}
where in the last inequality we used (\ref{Eq.FreeBridgeInCorridor}). Combining (\ref{Eq.AvoidInFree}), (\ref{Eq.EavoidLB}), and (\ref{Eq.LowFreeUB}), we conclude
$$\mathbb{P}_{\mathrm{avoid}}^{0,1, \vec{u}\,^{\mathrm{new}}, \vec{v}\,^{\mathrm{new}}}\left( \inf_{t \in [0,1]}\mathcal{Q}_k(t) \leq   - A \right) \leq \epsilon_0 (1 - \epsilon_0)^{-k}.$$
Using the latter, (\ref{Eq.UVNew}) and the fact that $\epsilon_0 (1- \epsilon_0)^{-k} < \epsilon$ by construction, we conclude (\ref{Eq.NotTooLowR1}), and hence the lemma.
\end{proof}

%
%
\subsection{Modulus of continuity estimates for Brownian line ensembles}\label{Section4.4}

For $a, b \in \mathbb{R}$ with $a < b$, $f \in C\left([a,b]\right)$, and $\delta > 0$, we define the usual {\em modulus of continuity}
\begin{equation}\label{Eq.MOCDef}
w(f,\delta) = \sup_{\substack{x,y \in [a,b] \\ |x-y| \leq \delta}} |f(x) - f(y)|.
\end{equation}

The following result shows that line ensembles with laws $\mathbb{P}_{\mathrm{avoid}}^{a,b, \vec{x}, \vec{y}}$ have well-behaved moduli of continuity uniformly as the boundary data $\vec{x}, \vec{y}$ vary over a bounded window.
\begin{lemma}\label{Lem.MOCUniform} Fix $k \in \mathbb{N}$, $a,b \in \mathbb{R}$ with $a < b$, and $M^{\mathrm{side}} > 0$. For any $\epsilon, \eta > 0$, there exists $\delta > 0$, depending on $k$, $a$, $b$, $M^{\mathrm{side}}$, $\epsilon$, and $\eta$, such that the following holds. If $\vec{x}, \vec{y} \in \weyl_k$, and $|x_i|, |y_i| \leq M^{\mathrm{side}}$ for $i \in \llbracket 1, k \rrbracket$, then 
\begin{equation}\label{Eq.MOCUniform}
\mathbb{P}_{\mathrm{avoid}}^{a,b, \vec{x}, \vec{y}}\left( \max_{i \in \llbracket 1, k \rrbracket} w(\mathcal{Q}_i, \delta) \geq \eta \right) < \epsilon.
\end{equation}
\end{lemma}
\begin{proof} Suppose, for the sake of contradiction, that no such $\delta > 0$ exists. Then, for each $\delta_n = 1/n$, there exist $\vec{x}\,^n, \vec{y}\,^n \in \weyl_k$ with $|x_i^n|, |y_i^n| \leq M^{\mathrm{side}}$, such that 
\begin{equation}\label{Eq.MOCUniformR1}
\mathbb{P}_{\mathrm{avoid}}^{a,b, \vec{x}\,^n, \vec{y}\,^n}\left( \max_{i \in \llbracket 1, k \rrbracket} w(\mathcal{Q}_i, \delta_n) \geq \eta \right) \geq \epsilon.
\end{equation}

By possibly passing to a subsequence, we may assume that $\vec{x}\,^n \rightarrow \vec{x}$ and $\vec{y}\,^n \rightarrow \vec{y}$ for some $\vec{x}, \vec{y} \in \weylc_k$. Consequently, for any $\delta_0 > 0$
\begin{equation}\label{Eq.MOCUniformR2}
\begin{split}
&\mathbb{P}_{\mathrm{avoid}}^{a,b, \vec{x}, \vec{y}}\left( \max_{i \in \llbracket 1, k \rrbracket} w(\mathcal{Q}_i, \delta_0) \geq \eta \right) \geq \limsup_{n \rightarrow \infty} \mathbb{P}_{\mathrm{avoid}}^{a,b, \vec{x}\,^n, \vec{y}\,^n }\left( \max_{i \in \llbracket 1, k \rrbracket} w(\mathcal{Q}_i, \delta_0) \geq \eta \right)\\
&\geq \limsup_{n \rightarrow \infty} \mathbb{P}_{\mathrm{avoid}}^{a,b, \vec{x}\,^n, \vec{y}\,^n }\left( \max_{i \in \llbracket 1, k \rrbracket} w(\mathcal{Q}_i, \delta_n) \geq \eta \right) \geq \epsilon,
\end{split} 
\end{equation} 
where the first inequality follows from Lemma \ref{Lem.BridgeEnsemblesCty} and the Portmanteau theorem, the second one from $\delta_n \downarrow 0$, and the third from (\ref{Eq.MOCUniformR1}). By the continuity of $\mathcal{Q}_i$, we have 
$$\max_{i \in \llbracket 1, k \rrbracket} w(\mathcal{Q}_i, \delta_0) \rightarrow 0 \mbox{ as } \delta_0 \rightarrow 0, \quad \mathbb{P}_{\mathrm{avoid}}^{a,b, \vec{x}, \vec{y}}\mbox{-a.s.}$$
By the bounded convergence theorem, we conclude
$$ \lim_{\delta_0 \rightarrow 0+} \mathbb{P}_{\mathrm{avoid}}^{a,b, \vec{x}, \vec{y}}\left( \max_{i \in \llbracket 1, k \rrbracket} w(\mathcal{Q}_i, \delta_0) \geq \eta \right) = 0,$$
which gives the desired contradiction with (\ref{Eq.MOCUniformR2}).
\end{proof}

Before we state the second modulus of continuity estimate of the section, Lemma \ref{Lem.MOCInside} below, we establish an auxiliary result that we require for its proof. The following lemma states that a line ensemble with law $\mathbb{P}_{\mathrm{avoid}}^{a,b, \vec{x}, \vec{y}}$, whose boundary data $\vec{x}, \vec{y}$ are well-separated from a Lipschitz continuous function $g$, stays above $g$ on the whole interval $[a,b]$ at least with some small probability.
\begin{lemma}\label{Lem.StayAboveEnvelope} Fix $k \in \mathbb{N}$, $a,b \in \mathbb{R}$ with $a < b$, and $\delta^{\mathrm{sep}}, M^{\mathrm{side}}, L > 0$. We can find $\rho > 0$, depending on $k$, $a$, $b$, $\delta^{\mathrm{sep}}$, $M^{\mathrm{side}}$, and $L$, such that the following holds. If $\vec{x}, \vec{y} \in \weyl_k$, and $g \in C([a,b])$ altogether satisfy
\begin{itemize}
\item $|x_i|, |y_i| \leq M^{\mathrm{side}}$ for $i \in \llbracket 1, k \rrbracket$, 
\item $|g(a)|, |g(b)| \leq M^{\mathrm{side}}$,
\item $g(a) \leq x_k - \delta^{\mathrm{sep}}$, $g(b) \leq y_k - \delta^{\mathrm{sep}}$,
\item $g$ is Lipschitz continuous with parameter $L$,
\end{itemize}  
then
\begin{equation}\label{Eq.StayAboveEnvelope}
\mathbb{P}_{\mathrm{avoid}}^{a,b, \vec{x}, \vec{y}}\left( \mathcal{Q}_k(t) > g(t) \mbox{ for all } t \in [a,b] \right) > \rho.
\end{equation}
\end{lemma}
\begin{proof} Suppose, for the sake of contradiction, that no such $\rho > 0$ exists. Then, for each $\rho_n = 1/n$, there exist $\vec{x}\,^n, \vec{y}\,^n \in \weyl_k$ and functions $g^n \in C([a,b])$ satisfying the four points in the statement of the lemma, and such that 
\begin{equation}\label{Eq.StayAboveEnvelopeR1}
\mathbb{P}_{\mathrm{avoid}}^{a,b, \vec{x}\,^n, \vec{y}\,^n}\left( \mathcal{Q}_k(t) > g^n(t) \mbox{ for all } t \in [a,b] \right) \leq \rho_n.
\end{equation}

Put $\delta = (2k)^{-1} \cdot \delta^{\mathrm{sep}}$, and define $\vec{u}\,^n, \vec{v}\,^n \in \weyl_k$ by $u_i^n = g^n(a) + (2k - i) \delta $ and $v_i^n = g^n(b) +  (2k - i)\delta$. From the first three points above, we conclude that 
\begin{enumerate}
\item $x_i^n \geq u_i^n$ and $y_i^n \geq v_i^n$ for $i \in \llbracket 1, k \rrbracket$,
\item $u_i^n - u_{i+1}^n = \delta$ and $v_i^n - v_{i+1}^n = \delta$ for $i \in \llbracket 1, k-1 \rrbracket$,
\item $u_k^n = k\delta + g^n(a)$ and $v_k^n = k\delta + g^n(b)$,
\item $|u_i^n| \leq M^{\mathrm{side}}$ and $|v_i^n| \leq M^{\mathrm{side}}$ for $i \in \llbracket 1, k \rrbracket$.
\end{enumerate}

Combining condition (1) with (\ref{Eq.StayAboveEnvelopeR1}) and the monotone coupling in Lemma \ref{Lem.MonCoup}, we conclude
\begin{equation}\label{Eq.StayAboveEnvelopeR2}
\mathbb{P}_{\mathrm{avoid}}^{a,b, \vec{u}\,^n, \vec{v}\,^n}\left( \mathcal{Q}_k(t) > g^n(t) \mbox{ for all } t \in [a,b] \right) \leq \rho_n.
\end{equation}

From the second and fourth points in the lemma, we conclude that $\{g^n\}_{n \geq 1}$ satisfies the conditions of the Arzel\`a--Ascoli theorem, and so by possibly passing to a subsequence, we may assume that $g^n$ converges uniformly on $[a,b]$ to a continuous function $g^{\infty}$. Furthermore, by possibly passing to a subsequence, using condition (4) above, we may assume that $\vec{u}\,^n \rightarrow \vec{u}\,^{\infty}$ and $\vec{v}\,^n \rightarrow \vec{v}\,^{\infty}$ for some $\vec{u}\,^{\infty}, \vec{v}\,^{\infty} \in \weylc_k$. In fact, from conditions (2) and (3), we must have $u_i^\infty - u_{i+1}^\infty = \delta$ and $v_i^\infty - v_{i+1}^\infty = \delta$ for $i \in \llbracket 1, k-1 \rrbracket$, as well as $u_k^\infty = k\delta + g^\infty(a)$ and $v_k^\infty = k\delta + g^\infty(b)$, so that $\vec{u}\,^{\infty}, \vec{v}\,^{\infty} \in \weyl_k$ and $u_k^\infty > g^{\infty}(a) + \delta/2$, $v_k^{\infty}  > g^{\infty}(b) + \delta/2$.\\

As mentioned in Definition \ref{Def.fgAvoidingBE}, see \cite[Definition 2.7]{DEA21} for the details, we know that there exists some $\epsilon > 0$, depending on $a$, $b$, $k$, $\vec{u}\,^\infty$, $\vec{v}\,^\infty$, $g^{\infty}$, such that 
$$\mathbb{P}_{\mathrm{free}}^{a,b, \vec{u}\,^\infty, \vec{v}\,^\infty}\left( \mathcal{Q}_{1} (t) > \cdots >  \mathcal{Q}_k(t) > g^{\infty}(t) + \delta/2 \mbox{ for all } t \in [a,b] \right) \geq \epsilon.$$ 
The latter and Definition \ref{Def.fgAvoidingBE} imply
\begin{equation}\label{Eq.StayAboveEnvelopeR3}
\mathbb{P}_{\mathrm{avoid}}^{a,b, \vec{u}\,^\infty, \vec{v}\,^\infty}\left( \mathcal{Q}_k(t) > g^{\infty}(t) + \delta/2 \mbox{ for all } t \in [a,b] \right) \geq \epsilon.
\end{equation}

We finally observe the following tower of inequalities
\begin{equation}\label{Eq.StayAboveEnvelopeTower}
\begin{split}
&\epsilon \leq \mathbb{P}_{\mathrm{avoid}}^{a,b, \vec{u}\,^\infty, \vec{v}\,^\infty}\left( \mathcal{Q}_k(t) > g^{\infty}(t) + \delta/2 \mbox{ for all } t \in [a,b] \right) \\
&\leq \liminf_{n \rightarrow \infty} \mathbb{P}_{\mathrm{avoid}}^{a,b, \vec{u}\,^n, \vec{v}\,^n}\left( \mathcal{Q}_k(t) > g^{\infty}(t) + \delta/2 \mbox{ for all } t \in [a,b] \right) \\
&\leq \liminf_{n \rightarrow \infty} \mathbb{P}_{\mathrm{avoid}}^{a,b, \vec{u}\,^n, \vec{v}\,^n}\left( \mathcal{Q}_k(t) > g^{n}(t) \mbox{ for all } t \in [a,b] \right) \leq \liminf_{n \rightarrow \infty} \rho_n = 0.
\end{split}
\end{equation}
Indeed, the first inequality follows from (\ref{Eq.StayAboveEnvelopeR3}) and the second one follows from Lemma \ref{Lem.BridgeEnsemblesCty} and the Portmanteau theorem. The third inequality follows from the uniform convergence of $g^n$ to $g^{\infty}$, and the fourth from (\ref{Eq.StayAboveEnvelopeR2}). The inequality (\ref{Eq.StayAboveEnvelopeTower}) gives the desired contradiction.
\end{proof}

We end this section with the following result. It shows that a line ensemble with law $\mathbb{P}_{\mathrm{avoid}}^{c,d, \vec{x}, \vec{y}, \infty, g}$, whose boundary data $\vec{x}, \vec{y}$ are close to zero, and with $g$ not too high on $[c,d]$, is likely to have a well-behaved modulus of continuity on a slightly {\em smaller} interval $[a,b] \subset (c,d)$.
\begin{lemma}\label{Lem.MOCInside} Fix $k \in \mathbb{N}$, $a,b,c,d \in \mathbb{R}$ with $c< a < b < d$, $M^{\mathrm{side}} > 0$ and $M^{\mathrm{bot}} > 0$. For any $\epsilon, \eta > 0$, there exists $\delta > 0$, depending on $k$, $a$, $b$, $c$, $d$, $M^{\mathrm{side}}$, $M^{\mathrm{bot}}$, $\epsilon$, and $\eta$, such that the following holds. If $\vec{x}, \vec{y} \in \weyl_k$, and $|x_i|, |y_i| \leq M^{\mathrm{side}}$ for $i \in \llbracket 1, k \rrbracket$, and $g: [c, d] \rightarrow [-\infty, \infty)$ is continuous and satisfies $x_k > g(c)$, $y_k > g(d)$, and $g(t) \leq M^{\mathrm{bot}}$ for all $t \in [c, d]$, then 
\begin{equation}\label{Eq.MOCInside}
\mathbb{P}_{\mathrm{avoid}}^{c,d, \vec{x}, \vec{y}, \infty, g}\left( \max_{i \in \llbracket 1, k \rrbracket} w(\mathcal{Q}_i[a,b], \delta) \geq \eta \right) < \epsilon.
\end{equation}
\end{lemma}
\begin{remark} We emphasize that the modulus of continuity in (\ref{Eq.MOCInside}) is for the interval $[a,b]$ as in (\ref{Eq.MOCDef}) and not $[c,d]$.
\end{remark}
\begin{proof} Throughout the proof, all the constants we encounter depend on $k$, $a$, $b$, $c$, $d$, $M^{\mathrm{side}}$, $M^{\mathrm{bot}}$, $\epsilon$, and $\eta$. We do not mention this further. For clarity, we split the proof into three steps. In Step 1, we reduce the proof of the lemma to showing that for some $\epsilon_1 > 0$, we have
$$\mathbb{P}_{\mathrm{avoid}}^{c,d, \vec{x}, \vec{y}, \infty, g}\left( \mathbb{P}_{\mathrm{avoid}}^{a,b, \vec{X}, \vec{Y}}(\tilde{\mathcal{Q}}_k(t) > g(t) \mbox{ for all $t \in [a,b]$}) \geq \epsilon_1 \right) > 1 -\epsilon/4,$$
see (\ref{Eq.InsideAcceptanceProbability}). In the last equation $X_i = \mathcal{Q}_i(a)$, $Y_i = \mathcal{Q}_i(b)$ for $i \in \llbracket 1, k \rrbracket$ and $\tilde{\mathcal{Q}}$ has law $\mathbb{P}_{\mathrm{avoid}}^{a,b, \vec{X}, \vec{Y}}$. The probability $\mathbb{P}_{\mathrm{avoid}}^{a,b, \vec{X}, \vec{Y}}(\tilde{\mathcal{Q}}_k(t) > g(t) \mbox{ for all $t \in [a,b]$})$ is sometimes referred to as an {\em acceptance probability}, and a lower bound on this probability allows us to transfer estimates from ensembles without a lower boundary $g$ to ones with a lower boundary $g$. In particular, with a bit of effort, we will be able to show that the modulus of continuity bound in (\ref{Eq.MOCInside}) follows from the one in Lemma \ref{Lem.MOCUniform} provided the acceptance probability is lower-bounded with high probability. In Step 2, we establish the acceptance probability lower bound in (\ref{Eq.InsideAcceptanceProbability}) by showing that it holds on the event that $\mathcal{Q}_k$ is well-separated from $g$ in a small neighborhood of $a$ and $b$ -- see (\ref{Eq.InsideSeparated}) for a precise formulation of this separation. The separation bound in (\ref{Eq.InsideSeparated}) is proved in Step 3, by replacing $g$ with a spike-function as in (\ref{Eq.Spike}). For such a function the separation is less likely, but we still show its probability is uniformly lower bounded, by appropriately applying the monotone coupling Lemma \ref{Lem.MonCoup} and Lemma \ref{Lem.StayAboveEnvelope}. We now turn to providing the details of the above outline.\\

{\bf \raggedleft Step 1.} For a function $u$ on $[a,b]$, define the function $G_{u}: C( \llbracket 1, k \rrbracket \times [a, b] ) \rightarrow [0,1]$ through
$$G_{u}\left( \{h_i\}_{i = 1}^{k} \right) = {\bf 1} \left\{ h_k(t) > u(t)  \mbox{ for all } t \in [a,b] \right\}.$$ 
In addition, we let $\tilde{g}$ be the restriction of $g$ to $[a,b]$ and define the random $\vec{X}, \vec{Y} \in \weyl_k$ via 
$$X_i = \mathcal{Q}_i(a), \hspace{2mm} Y_i = \mathcal{Q}_i(b) \mbox{ for } i \in \llbracket 1, k \rrbracket.$$
We claim that there exists $\epsilon_1 \in (0,1)$ such that 
\begin{equation}\label{Eq.InsideAcceptanceProbability}
\mathbb{P}_{\mathrm{avoid}}^{c,d, \vec{x}, \vec{y}, \infty, g}\left( E^{\mathrm{ap}} \right) > 1 -\epsilon/4, \mbox{ where } E^{\mathrm{ap}} = \left\{ \mathbb{E}_{\mathrm{avoid}}^{a,b,\vec{X}, \vec{Y}}\left[G_{\tilde{g}}(\tilde{\mathcal{Q}})\right] \geq \epsilon_1 \right\},
\end{equation}
and $\tilde{\mathcal{Q}} = \{\tilde{\mathcal{Q}}_{i}\}_{i = 1}^k \in C(\llbracket 1, k \rrbracket \times [a,b])$ is distributed according to $\mathbb{P}_{\mathrm{avoid}}^{a,b,\vec{X}, \vec{Y}}$. We mention that the (random) expectation $\mathbb{E}_{\mathrm{avoid}}^{a,b,\vec{X}, \vec{Y}}\left[G_{\tilde{g}}(\tilde{\mathcal{Q}})\right]$ is a measurable function of $\vec{X}, \vec{Y}$ by \cite[Lemma 3.4]{DimMat}, and hence itself a random variable, so that $E^{\mathrm{ap}}$ is an event. We prove (\ref{Eq.InsideAcceptanceProbability}) in the steps below. Here, we assume its validity and conclude the proof of the lemma.\\

From Lemma \ref{Lem.NoBigMaxBLE} applied to $k, M^{\mathrm{side}}, M^{\mathrm{bot}}$ as in the present lemma, $a = c$, $b = d$, and $\epsilon$ replaced with $\epsilon/16$, we can find $A_1 > 0$, such that 
\begin{equation}\label{Eq.InsideNBM}
\mathbb{P}_{\mathrm{avoid}}^{c,d, \vec{x}, \vec{y}, \infty, g}\left( \max\left(X_1, Y_1\right) > A_1 \right) < \epsilon/16.
\end{equation}
On the other hand, from Lemma \ref{Lem.StayInCorridor} applied to $k, M^{\mathrm{side}}$, $a = c$, $b = d$ and $\epsilon$ replaced with $\epsilon/16$, we can find $A_2 > 0$, such that 
\begin{equation*}
\mathbb{P}_{\mathrm{avoid}}^{c,d, \vec{x}, \vec{y}}\left( \min\left(X_k, Y_k\right) < -A_2 \right) < \epsilon/16.
\end{equation*}
Combining the latter with the monotone coupling in Lemma \ref{Lem.MonCoup}, we conclude that 
\begin{equation}\label{Eq.InsideNLM}
\mathbb{P}_{\mathrm{avoid}}^{c,d, \vec{x}, \vec{y}, \infty, g}\left( \min\left(X_k, Y_k\right) < -A_2 \right) < \epsilon/16.
\end{equation}
Setting $\hat{A} = \max(A_1, A_2)$, and taking a union bound of (\ref{Eq.InsideNBM}) and (\ref{Eq.InsideNLM}), we conclude that 
\begin{equation}\label{Eq.InsideBounded}
\mathbb{P}_{\mathrm{avoid}}^{c,d, \vec{x}, \vec{y}, \infty, g}\left( E^{\mathrm{bdd}} \right) > 1- \epsilon/8, \mbox{ where } E^{\mathrm{bdd}} = \{|X_i|, |Y_i| \leq \hat{A} \mbox{ for } i \in \llbracket 1, k \rrbracket\}.
\end{equation}

From Lemma \ref{Lem.MOCUniform} applied to $k,a,b, \eta$ as in the present lemma, $M^{\mathrm{side}} = \hat{A}$, and $\epsilon$ replaced with $\epsilon \epsilon_1/2$, we can find $\delta > 0$, such that for $\vec{x}\,', \vec{y}\,' \in \weyl_k$ with $|x_i'|, |y_i'| \leq \hat{A}$ for $i \in \llbracket 1, k \rrbracket$, we have 
\begin{equation}\label{Eq.InsideMOCUniform}
\mathbb{P}_{\mathrm{avoid}}^{a,b, \vec{x}\,', \vec{y}\,'}\left( \max_{i \in \llbracket 1, k \rrbracket} w(\mathcal{Q}_i, \delta) \geq \eta \right) < \epsilon \epsilon_1/2.
\end{equation}
This specifies the $\delta$ in the statement of the lemma, and we proceed to show that (\ref{Eq.MOCInside}) holds.\\

Let us define the functions $H: C(\llbracket 1, k \rrbracket \times [c,d]) \rightarrow \mathbb{R}$, $\tilde{H}: C(\llbracket 1, k \rrbracket \times [a,b]) \rightarrow \mathbb{R}$ via
$$H\left( \{f_i\}_{i = 1}^k \right) = \max_{i \in \llbracket 1, k \rrbracket}\sup_{\substack{x,y \in [a,b] \\ |x-y| \leq \delta}} |f_i(x) - f_i(y)| \mbox{ and } \tilde{H}\left( \{f_i\}_{i = 1}^k \right) = \max_{i \in \llbracket 1, k \rrbracket}\sup_{\substack{x,y \in [a,b] \\ |x-y| \leq \delta}} |f_i(x) - f_i(y)|.$$
We then have the following identities
\begin{equation}\label{Eq.InsideTower1}
\begin{split}
&\mathbb{E}_{\mathrm{avoid}}^{c,d, \vec{x}, \vec{y}, \infty, g}\left[ {\bf 1}_{E^{\mathrm{bdd}}} \cdot {\bf 1}_{E^{\mathrm{ap}}} \cdot {\bf 1}\{H(\mathcal{Q}) \geq \eta \} \right]  \\
&= \mathbb{E}_{\mathrm{avoid}}^{c,d, \vec{x}, \vec{y}, \infty, g}\left[ {\bf 1}_{E^{\mathrm{bdd}}} \cdot {\bf 1}_{E^{\mathrm{ap}}} \cdot \mathbb{E}_{\mathrm{avoid}}^{c,d, \vec{x}, \vec{y}, \infty, g}\left[ {\bf 1}\{H(\mathcal{Q}) \geq \eta \} \vert \mathcal{F}_{\mathrm{ext}} \right] \right] \\
& = \mathbb{E}_{\mathrm{avoid}}^{c,d, \vec{x}, \vec{y}, \infty, g}\left[ {\bf 1}_{E^{\mathrm{bdd}}} \cdot {\bf 1}_{E^{\mathrm{ap}}} \cdot \mathbb{E}_{\mathrm{avoid}}^{a,b, \vec{X}, \vec{Y}, \infty, \tilde{g}}\left[ {\bf 1}\{\tilde{H}(\tilde{\mathcal{Q}}) \geq \eta \} \right] \right]
\end{split}
\end{equation}
where $\mathcal{F}_{\mathrm{ext}} = \sigma \left \{ \mathcal{Q}_i(s): (i,s) \in \llbracket 1, k \rrbracket \times [c,d] \setminus (\llbracket 1,k \rrbracket \times (a,b)) \right\}$. Indeed, the first equality follows from the tower property for conditional expectation, and the fact that $E^{\mathrm{bdd}}, E^{\mathrm{ap}} \in \mathcal{F}_{\mathrm{ext}}$. The second equality is a bit subtle. When $g = -\infty$, it follows from the fact that $\mathcal{Q}$ with law $\mathbb{P}_{\mathrm{avoid}}^{c,d, \vec{x}, \vec{y}}$ satisfies the Brownian Gibbs property as a $\llbracket 1, k \rrbracket$-indexed line ensemble on the interval $[c,d]$ in the sense of \cite[Definition 2.5]{DimMat}, which is true by \cite[Lemma 2.13]{DimMat}. When $g$ is a finite-valued continuous function, the equality follows from the fact that if $\mathcal{Q}$ has law $\mathbb{P}_{\mathrm{avoid}}^{c,d, \vec{x}, \vec{y}, \infty, g}$, and $\mathcal{L}$ is the $\llbracket 1, k+1 \rrbracket$-indexed line ensemble on $[c,d]$ given by $\mathcal{L}_i = \mathcal{Q}_i$ for $i \in \llbracket 1, k \rrbracket$ and $\mathcal{L}_{k+1} = g$, then $\mathcal{L}$ satisfies the partial Brownian Gibbs property of \cite[Definition 2.7]{DimMat}. The last statement follows from \cite[Lemma 2.6]{DS25} with $\vec{A} = 0$. \\

We now observe that on the event $E^{\mathrm{bdd}} \cap E^{\mathrm{ap}}$, we have 
\begin{equation}\label{Eq.InsideMOCAP}
\begin{split}
&\mathbb{E}_{\mathrm{avoid}}^{a,b, \vec{X}, \vec{Y}, \infty, \tilde{g}}\left[ {\bf 1}\{\tilde{H}(\tilde{\mathcal{Q}}) \geq \eta \} \right] = \frac{\mathbb{E}_{\mathrm{avoid}}^{a,b, \vec{X}, \vec{Y}}\left[ {\bf 1}\{\tilde{H}(\tilde{\mathcal{Q}}) \geq \eta \} \cdot G_{\tilde{g}}(\tilde{\mathcal{Q}}) \right]}{\mathbb{E}_{\mathrm{avoid}}^{a,b, \vec{X}, \vec{Y}}\left[  G_{\tilde{g}}(\tilde{\mathcal{Q}}) \right]} \\
& \leq \epsilon_1^{-1}\mathbb{E}_{\mathrm{avoid}}^{a,b, \vec{X}, \vec{Y}}\left[ {\bf 1}\{\tilde{H}(\tilde{\mathcal{Q}}) \geq \eta \} \right] \leq \epsilon/2.
\end{split}
\end{equation}
Indeed, the equality on the first line follows from Definition \ref{Def.fgAvoidingBE}. In going from the first to the second line we used that $G_{\tilde{g}}(\tilde{\mathcal{Q}}) \in [0,1]$ and the definition of $E^{\mathrm{ap}}$. In the last inequality we used (\ref{Eq.InsideMOCUniform}), with $\vec{X}, \vec{Y}$ playing the roles of $\vec{x}\,', \vec{y}\,'$ there, and the definitions of $E^{\mathrm{bdd}}$ and $\tilde{H}$. 

Combining (\ref{Eq.InsideTower1}) and (\ref{Eq.InsideMOCAP}), and recalling the definition of $H$, we conclude 
$$\mathbb{P}_{\mathrm{avoid}}^{c,d, \vec{x}, \vec{y}, \infty, g}\left(E^{\mathrm{bdd}} \cap E^{\mathrm{ap}} \cap \left\{ \max_{i \in \llbracket 1, k \rrbracket} w(\mathcal{Q}_i[a,b], \delta) \geq \eta \right\} \right) \leq \epsilon/2.$$
Combining the last inequality with (\ref{Eq.InsideAcceptanceProbability}) and (\ref{Eq.InsideBounded}), we conclude (\ref{Eq.MOCInside}).\\

{\bf \raggedleft Step 2.} We claim that we can find $\delta^{\mathrm{sep}} \in (0,1)$ and $\Delta^{\mathrm{sep}} \in (0, (b-a)/2)$, such that 
\begin{equation}\label{Eq.InsideSeparated}
\begin{split}
&\mathbb{P}_{\mathrm{avoid}}^{c,d, \vec{x}, \vec{y}, \infty, g} \left( E^{\mathrm{sep}}_1 \cap E^{\mathrm{sep}}_2  \right) > 1 - \epsilon/8, \mbox{ where } \\
&E_1^{\mathrm{sep}} = \left\{\mathcal{Q}_k(a) \geq \max_{t \in [a, a + \Delta^{\mathrm{sep}}]} g(t) + \delta^{\mathrm{sep}} \right\} \mbox{ and } E_2^{\mathrm{sep}} = \left\{\mathcal{Q}_k(b) \geq \max_{t \in [b - \Delta^{\mathrm{sep}}, b]} g(t) + \delta^{\mathrm{sep}} \right\}.
\end{split}
\end{equation}
We prove (\ref{Eq.InsideSeparated}) in the next step. Here, we assume its validity and conclude the proof of (\ref{Eq.InsideAcceptanceProbability}).\\

Let $\hat{A}$ be as in Step 1, see (\ref{Eq.InsideBounded}), and let $\rho$ be as in Lemma \ref{Lem.StayAboveEnvelope} for $k,a,b, \delta^{\mathrm{sep}}$ as in the present setup, $M^{\mathrm{side}} = \hat{A} + 1$ and $L = (2\hat{A} + M^{\mathrm{bot}} +\delta^{\mathrm{sep}})/\Delta^{\mathrm{sep}}$. Below we show (\ref{Eq.InsideAcceptanceProbability}) for $\epsilon_1 = \rho$.

Consider the (random) function $\hat{g} \in C([a,b])$ given by 
$$\hat{g}(a) = X_k - \delta^{\mathrm{sep}}, \hspace{2mm} \hat{g}(b) = Y_k - \delta^{\mathrm{sep}}, \hspace{2mm} \hat{g}(t) = \hat{A}+ M^{\mathrm{bot}} \mbox{ for } t \in [a + \Delta^{\mathrm{sep}}, b - \Delta^{\mathrm{sep}}],$$
and $\hat{g}$ is linear on $[a, a + \Delta^{\mathrm{sep}}]$ and $[b - \Delta^{\mathrm{sep}}, b]$. Note that the following hold on $E^{\mathrm{bdd}} \cap E^{\mathrm{sep}}_1 \cap E^{\mathrm{sep}}_2$
\begin{equation}\label{Eq.InsideProp1}
\hat{g}(t) \geq g(t) \mbox{ for all } t\in [a,b], \mbox{ and $\hat{g}$ is Lipschitz continuous with parameter $L$}. 
\end{equation}
Indeed, since $X_k \in [-\hat{A}, \hat{A}]$ on $E^{\mathrm{bdd}}$, we know 
$$\hat{g}(a + \Delta^{\mathrm{sep}}) - \hat{g}(a) \in [0,2\hat{A} + M^{\mathrm{bot}} +\delta^{\mathrm{sep}}].$$
The last equation shows that the slope of $\hat{g}$ on $[a, a + \Delta^{\mathrm{sep}}]$ is at most $L$. An analogous argument works for $[b - \Delta^{\mathrm{sep}}, b]$ and $\hat{g}$ is constant on $[a + \Delta^{\mathrm{sep}}, b - \Delta^{\mathrm{sep}}]$, which verifies the second statement in (\ref{Eq.InsideProp1}). In addition, the last displayed equation shows that $\hat{g}$ is increasing on $[a, a + \Delta^{\mathrm{sep}}]$, which by the definition of $E_1^{\mathrm{sep}}$ gives for $t \in [a, a + \Delta^{\mathrm{sep}}]$
$$\hat{g}(t) \geq \hat{g}(a) = X_k - \delta^{\mathrm{sep}} \geq \max_{s \in [a, a + \Delta^{\mathrm{sep}}]} g(s) \geq g(t).$$
This verifies the inequality in (\ref{Eq.InsideProp1}) when $t \in [a, a + \Delta^{\mathrm{sep}}]$, and an analogous argument shows it when $t \in [b - \Delta^{\mathrm{sep}}, b]$. For $t \in [a + \Delta^{\mathrm{sep}}, b - \Delta^{\mathrm{sep}}]$ the inequality is immediate from our assumption that $g(t) \leq M^{\mathrm{bot}}$ and $\hat{g}(t) = \hat{A} + M^{\mathrm{bot}}$.\\

The above work shows that on $E =E^{\mathrm{bdd}} \cap E^{\mathrm{sep}}_1 \cap E^{\mathrm{sep}}_2$ the conditions of Lemma \ref{Lem.StayAboveEnvelope} are satisfied by $\vec{x} = \vec{X}$, $\vec{y} = \vec{Y}$ and $g = \hat{g}$ (with our earlier choice of parameters), and so on $E$
\begin{equation*}
\mathbb{E}_{\mathrm{avoid}}^{a,b,\vec{X}, \vec{Y}}\left[G_{\hat{g}}(\tilde{\mathcal{Q}})\right] \geq \epsilon_1. 
\end{equation*}
Since by (\ref{Eq.InsideProp1}), we have $\hat{g}(t) \geq \tilde{g}(t) = g(t)$, and hence $G_{\hat{g}}(\tilde{\mathcal{Q}}) \leq G_{\tilde{g}}(\tilde{\mathcal{Q}})$, we conclude that on $E$
$$\mathbb{E}_{\mathrm{avoid}}^{a,b,\vec{X}, \vec{Y}}\left[G_{\tilde{g}}(\tilde{\mathcal{Q}})\right] \geq \epsilon_1.$$
The last equation implies
$$ \mathbb{P}_{\mathrm{avoid}}^{c,d, \vec{x}, \vec{y}, \infty, g}\left( \mathbb{E}_{\mathrm{avoid}}^{a,b,\vec{X}, \vec{Y}}\left[G_{\tilde{g}}(\tilde{\mathcal{Q}})\right] \geq \epsilon_1\right) \geq \mathbb{P}_{\mathrm{avoid}}^{c,d, \vec{x}, \vec{y}, \infty, g}(E) > 1 - \epsilon/4,$$
where the last inequality used (\ref{Eq.InsideBounded}) and (\ref{Eq.InsideSeparated}). This proves (\ref{Eq.InsideAcceptanceProbability}).\\

{\bf \raggedleft Step 3.} In this step, we prove (\ref{Eq.InsideSeparated}). We will show that there exist $\delta^{\mathrm{sep}}_i \in (0,1)$ and $\Delta^{\mathrm{sep}}_i \in (0, (b-a)/2)$, such that 
\begin{equation}\label{Eq.InsideSeparated2}
\begin{split}
&\mathbb{P}_{\mathrm{avoid}}^{c,d, \vec{x}, \vec{y}, \infty, g} \left( \mathcal{Q}_k(a) < \max_{t \in [a, a + \Delta^{\mathrm{sep}}_1]} g(t) + \delta^{\mathrm{sep}}_1   \right) \leq \epsilon/16,  \\
&\mathbb{P}_{\mathrm{avoid}}^{c,d, \vec{x}, \vec{y}, \infty, g} \left( \mathcal{Q}_k(b) < \max_{t \in [b- \Delta^{\mathrm{sep}}_2, b]} g(t) + \delta^{\mathrm{sep}}_2   \right) \leq \epsilon/16.
\end{split}
\end{equation}
A union bound of the two lines in (\ref{Eq.InsideSeparated2}) gives (\ref{Eq.InsideSeparated}) with $\delta^{\mathrm{sep}} = \min(\delta^{\mathrm{sep}}_1, \delta^{\mathrm{sep}}_2)$ and $\Delta^{\mathrm{sep}} = \min (\Delta^{\mathrm{sep}}_1, \Delta^{\mathrm{sep}}_2)$. As the proofs are analogous, we only show the first line in (\ref{Eq.InsideSeparated2}).\\

Suppose for the sake of contradiction that no such $\delta^{\mathrm{sep}}_1, \Delta^{\mathrm{sep}}_1$ exist. Then, we can find sequences $\delta^{\mathrm{sep}}_{1,n} \in (0,1)$, $\Delta^{\mathrm{sep}}_{1,n} \in (0, (b-a)/2)$, $\vec{x}^n, \vec{y}^n \in \weyl_k$ with $|x_i^n|, |y_i^n| \leq M^{\mathrm{side}}$ and continuous $g^n: [c,d] \rightarrow [-\infty, \infty)$ with $g^n(t) \leq M^{\mathrm{bot}}$ for $t \in [c,d]$ and $g^n(c) < x_k^n$, $g^n(d) < y_k^n$, such that 
\begin{equation}\label{Eq.InsideFS1}
\mathbb{P}_{\mathrm{avoid}}^{c,d, \vec{x}^n, \vec{y}^n, \infty, g^n} \left( \mathcal{Q}_k(a) < \max_{t \in [a, a + \Delta^{\mathrm{sep}}_{1,n}]} g^n(t) + \delta^{\mathrm{sep}}_{1,n}   \right) > \epsilon/16.
\end{equation}
Since $g^n$ are continuous, we can find $t_n \in [a, a + \Delta^{\mathrm{sep}}_{1,n}]$, such that 
$$g^n(t_n) = \max_{t \in [a, a + \Delta^{\mathrm{sep}}_{1,n}]} g^n(t).$$
Let $q^n$ be the function such that $q^n(t_n) = g^n(t_n)$ and $q^n(t) = -\infty$ for $t \in [c,d]\setminus \{t_n\}$. In addition, let $\vec{u}, \vec{v} \in \weyl_k$ be defined by $u_i = v_i = -M^{\mathrm{side}} - i$, and observe that 
$$u_i \leq x^n_i, \hspace{2mm} v_i \leq y_i^n \mbox{ for } i \in \llbracket 1, k \rrbracket \mbox{, and }q^n(t) \leq g^n(t) \mbox{ for }t \in [c,d].$$
By the monotone coupling in Lemma \ref{Lem.MonCoup}, and (\ref{Eq.InsideFS1}), we conclude
\begin{equation}\label{Eq.InsideFS2}
\mathbb{P}_{\mathrm{avoid}}^{c,d, \vec{u}, \vec{v}, \infty, q^n} \left( \mathcal{Q}_k(a) < q^n(t_n) + \delta^{\mathrm{sep}}_{1,n}  \right) > \epsilon/16.
\end{equation}

Let $q \in C([c,d])$ be linear on $[c,a]$, $[a,b]$ and $[b,d]$ with 
$$q(c) = - M^{\mathrm{side}} - k - 1 = q(d), \mbox{ and } q(a) = q(b) = M^{\mathrm{bot}}.$$
We observe that $q$ is Lipschitz continuous with constant 
$$L_q = (M^{\mathrm{bot}} + M^{\mathrm{side}} + k + 1) \cdot \max\left((a-c)^{-1}, (d-b)^{-1}\right).$$
In addition, we notice that $q(t) \geq q^n(t)$ for $t \in [c,d]$, and $q(c) = q(d) = u_k - 1 = v_k - 1$. Consequently,
\begin{equation}\label{Eq.InsideFS3}
\mathbb{P}_{\mathrm{avoid}}^{c,d, \vec{u}, \vec{v}} \left( \mathcal{Q}_k(t_n) > q^n(t_n)  \right) \geq \mathbb{P}_{\mathrm{avoid}}^{c,d, \vec{u}, \vec{v}} \left( \mathcal{Q}_k(t) > q(t) \mbox{ for all $t \in [c,d]$ } \right) \geq \rho_q,
\end{equation}
where $\rho_q$ is the $\rho$ from Lemma \ref{Lem.StayAboveEnvelope} applied to $k$ as in the present setup, $a = c$, $b = d$, $M^{\mathrm{side}}$ set to $M^{\mathrm{side}} + k + 1$, $\delta^{\mathrm{sep}} = 1$, and $L = L_q$. 

Combining (\ref{Eq.InsideFS2}) and (\ref{Eq.InsideFS3}) with Definition \ref{Def.fgAvoidingBE} we conclude
\begin{equation}\label{Eq.InsideFS4}
\begin{split}
&\mathbb{P}_{\mathrm{avoid}}^{c,d, \vec{u}, \vec{v}} \left( \mathcal{Q}_k(a) < q^n(t_n) + \delta^{\mathrm{sep}}_{1,n} \mbox{ and } \mathcal{Q}_k(t_n) > q^n(t_n) \right) \\
&= \mathbb{P}_{\mathrm{avoid}}^{c,d, \vec{u}, \vec{v}, \infty, q^n} \left( \mathcal{Q}_k(a) < q^n(t_n) + \delta^{\mathrm{sep}}_{1,n}  \right) \cdot  \mathbb{P}_{\mathrm{avoid}}^{c,d, \vec{u}, \vec{v}} \left( \mathcal{Q}_k(t_n) > q^n(t_n)  \right) > \rho_q\epsilon/16.
\end{split}
\end{equation}
Since $q^n(t_n) \leq M^{\mathrm{bot}}$, by possibly passing to a subsequence, we may assume $q^n(t_n) \rightarrow \alpha \in [-\infty, M^{\mathrm{bot}}]$. Using that under $\mathbb{P}_{\mathrm{avoid}}^{c,d, \vec{u}, \vec{v}}$ the curve $\mathcal{Q}_k$ is a.s. continuous, $t_n \rightarrow a$, and $\delta^{\mathrm{sep}}_{1,n} \rightarrow 0$, we conclude that
\begin{equation}\label{Eq.InsideFS5}
\mathbb{P}_{\mathrm{avoid}}^{c,d, \vec{u}, \vec{v}} \left( \mathcal{Q}_k(a) =\alpha \right) \geq \limsup_{n \rightarrow \infty}  \mathbb{P}_{\mathrm{avoid}}^{c,d, \vec{u}, \vec{v}} \left( \mathcal{Q}_k(a) \leq q^n(t_n) + \delta^{\mathrm{sep}}_{1,n} \mbox{ and } \mathcal{Q}_k(t_n) \geq q^n(t_n) \right) \geq \frac{\rho_q\epsilon}{16},
\end{equation}
where the first inequality follows by Fatou's lemma and the second by (\ref{Eq.InsideFS4}). 

If $E_{\mathrm{avoid}}$ is as in (\ref{Eq.DefEavoid}) with $f = \infty$, and $g = -\infty$, we have from Definition \ref{Def.fgAvoidingBE} 
$$\mathbb{P}_{\mathrm{avoid}}^{c,d, \vec{u}, \vec{v}} \left( \mathcal{Q}_k(a) =\alpha \right) = \frac{\mathbb{P}_{\mathrm{free}}^{c,d, \vec{u}, \vec{v}}  (\{B_k(a) = \alpha\} \cap E_{\mathrm{avoid}})}{\mathbb{P}_{\mathrm{free}}^{c,d, \vec{u}, \vec{v}}  (E_{\mathrm{avoid}})} = 0,$$
where in the last equality we used that the law of $B_k(a)$ is diffuse, as this is a normal variable with a non-zero variance, due to $a \in (c,d)$. 

The last displayed equation and (\ref{Eq.InsideFS5}) show $0 \geq \rho_q\epsilon/16$, which is our desired contradiction. The contradiction arose from our assumption that there are no $\delta^{\mathrm{sep}}_1 \in (0,1)$ and $\Delta^{\mathrm{sep}}_1 \in (0, (b-a)/2)$ satisfying the first line in (\ref{Eq.InsideSeparated2}). This completes the proof of (\ref{Eq.InsideSeparated2}) and hence the lemma.
\end{proof}

%
%
\section{Convergence to Dyson Brownian motion}\label{Section5} The goal of this section is to prove Theorem \ref{Thm.ConvDBM}. We do so in Section \ref{Section5.2} after establishing two technical propositions in Section \ref{Section5.1}. Throughout this section we continue with the same notation as in Sections \ref{Section1} and \ref{Section4}.

%
%
\subsection{Approximate curve location}\label{Section5.1} In this section, we establish two technical results which we require for the proof of Theorem \ref{Thm.ConvDBM}. Assume the same notation as in Theorem \ref{Thm.ConvDBM}(a), and introduce the auxiliary line ensemble $\tilde{\mathcal{L}}^N = \{\tilde{\mathcal{L}}^N_i\}_{i \geq 1}$ by
\begin{equation}\label{Eq.LNTilde}
\tilde{\mathcal{L}}^N_i(t) = (2T_N)^{-1/2} \cdot \left( \mathcal{A}^{a,b,c}_{i}(tT_N) + 2tT_N/v_k^a - t^2T_N^2 \right), \mbox{ for } (i,t) \in \mathbb{N} \times \mathbb{R}.
\end{equation}
By Proposition \ref{Prop.AWLE}, we know that $\mathcal{L}^{a,b,c}$ satisfies the Brownian Gibbs property on $\mathbb{R}$, and then from Lemma \ref{Lem.Affine} we conclude the same for $\tilde{\mathcal{L}}^N$. In addition, we have by (\ref{Eq.ScaledSlopeA}) that 
\begin{equation}\label{Eq.LNMatch}
\mathcal{L}_i^N(t)  = \tilde{\mathcal{L}}^N_{\mathsf{M}_{k-1}^a+ i}(t) \mbox{ for } (i,t) \in \llbracket 1, \mathsf{m}_k^a \rrbracket \times (0, \infty).
\end{equation}

The following result shows that the curves $\mathcal{L}_i^N$ from Theorem \ref{Thm.ConvDBM} are likely close (on scale $T_N^{1/2}$) to zero uniformly on compact intervals.
\begin{proposition}\label{Prop.CloseToZero} Assume the same notation as in Theorem \ref{Thm.ConvDBM}(a). For any $0 < \alpha < \beta$, and $\epsilon > 0$, there exists $M > 0$, depending on $\alpha, \beta, \epsilon$, the parameters $(a,b,c)$, and the sequence $\{T_N\}_{N \geq 1}$, such that for all $N \in \mathbb{N}$
\begin{equation}\label{Eq.LNStayInCorridor}
\mathbb{P}\left( \sup_{(i,t) \in \llbracket 1, \mathsf{m}_k^a \rrbracket \times [\alpha, \beta]} |\mathcal{L}_i^N(t)| > M \right) < \epsilon.
\end{equation} 
\end{proposition}
\begin{proof} Throughout the proof all the constants we encounter depend on $\alpha, \beta, \epsilon$, the parameters $(a,b,c)$, and the sequence $\{T_N\}_{N \geq 1}$. We do not mention this further. For clarity, we split the proof into three steps.\\

{\bf \raggedleft Step 1.} Let $\tilde{\mathcal{L}}^N$ be as in (\ref{Eq.LNTilde}). We claim that we can find $M_1 > 0$, such that  
\begin{equation}\label{Eq.LNNotTooHigh}
\mathbb{P}\left( \sup_{t \in [\alpha, \beta]} \tilde{\mathcal{L}}_{\mathsf{M}^a_{k-1}+ 1}^N(t) > M_1 \right) < \epsilon/2.
\end{equation} 
In addition, we claim that we can find $M_2 > 0$, such that  
\begin{equation}\label{Eq.LNNotTooLow}
\mathbb{P}\left( \inf_{t \in [\alpha, \beta]} \tilde{\mathcal{L}}_{\mathsf{M}_k^a}^N(t) < -M_2 \right) < \epsilon/2.
\end{equation}

As $\tilde{\mathcal{L}}^N$ satisfies the Brownian Gibbs property, we have for all $t \in \mathbb{R}$ 
\begin{equation}\label{Eq.LNTildeNonIntersecting}
\tilde{\mathcal{L}}^N_1(t) > \tilde{\mathcal{L}}^N_2(t) > \cdots > \tilde{\mathcal{L}}^N_{\mathsf{M}_k^a}(t).
\end{equation}
Consequently, a union bound of (\ref{Eq.LNNotTooHigh}) and (\ref{Eq.LNNotTooLow}) implies for $M = \max(M_1, M_2)$ that
$$
\mathbb{P}\left( \sup_{(i,t) \in \llbracket 1, \mathsf{m}_k^a \rrbracket \times [\alpha, \beta]} |\tilde{\mathcal{L}}_{i+ \mathsf{M}_{k-1}^a}^N(t)| > M \right) < \epsilon.$$
The last inequality and (\ref{Eq.LNMatch}) imply (\ref{Eq.LNStayInCorridor}).

We have thus reduced the proof of the proposition to establishing (\ref{Eq.LNNotTooHigh}) and (\ref{Eq.LNNotTooLow}).\\

{\bf \raggedleft Step 2.} In this step, we prove (\ref{Eq.LNNotTooHigh}).

Set for convenience $K_1 = \mathsf{M}_{k-1}^a + 1$, $K_2 = \mathsf{M}_{k}^a$, $\Delta = \beta - \alpha$, $\gamma = \alpha + 4\Delta$. From Proposition \ref{Prop.FDSlope} and (\ref{Eq.LNMatch}), we can find $\hat{M}$, sufficiently large, so that for all $N \in \mathbb{N}$ and $r \in \{\alpha, \beta, \gamma\}$
\begin{equation}\label{Eq.LNGamma}
\mathbb{P}\left(H^r_N\right) < \epsilon/8, \mbox{ where } H^r_N = \left\{|\tilde{\mathcal{L}}^N_i(r)| \geq \hat{M} \mbox{ for some } i \in \llbracket K_1, K_2 \rrbracket \right\}.
\end{equation} 
We also define
\begin{equation}\label{Eq.LNGamma2}
\tilde{H}^{\beta}_N = \left\{\tilde{\mathcal{L}}^N_{K_1}(\beta) \geq \hat{M} \right\}, \mbox{ and note } \mathbb{P}\left(\tilde{H}^{\beta}_N\right) < \epsilon/8, \mbox{ as } \tilde{H}^{\beta}_N \subseteq H^\beta_N.
\end{equation}

Let $A$ be as in Lemma \ref{Lem.NotTooLow} for $k = K_1$ and $\epsilon$ replaced with $\epsilon/4$. From (\ref{Eq.NotTooLow}), we conclude that if $\vec{x}, \vec{y} \in \weyl_{K_1}$ satisfy $x_{K_1} > \max(P_1, g(a))$, $y_{K_1} > \max(P_2,g(b))$ for some $P_1, P_2 \in \mathbb{R}$, and continuous $g$, then for each $t \in [a,b]$
\begin{equation}\label{Eq.CorrDBM}
\mathbb{P}_{\mathrm{avoid}}^{a,b, \vec{x}, \vec{y}, \infty, g}\left( \mathcal{Q}_{K_1}(t)  \leq P_1  \cdot \frac{b - t}{b - a} + P_2 \cdot \frac{t - a}{b - a} - A \sqrt{b-a}   \right) < \epsilon/4.
\end{equation}
Pick $M_1 > 0$ sufficiently large, so that $3M_1/4 \geq 2 \hat{M} + A \sqrt{\gamma - \alpha}$, and observe that for $s \in [\alpha,\beta]$
\begin{equation}\label{Eq.LNNotTooHighM1}
M_1 \cdot \frac{\gamma-\beta}{\gamma-s} - \hat{M} \cdot \frac{\beta-s}{\gamma-s} - A \sqrt{\gamma-s} \geq \hat{M}.
\end{equation}
Below, we proceed to prove (\ref{Eq.LNNotTooHigh}) for this choice of $M_1$.\\

Fix a finite set $\mathsf{S} = \{s_1, \dots, s_m\} \subset [\alpha, \beta]$ with $s_1 < s_2  < \cdots < s_m$, and define the events 
$$E_N^v = \left\{ \tilde{\mathcal{L}}^N_{K_1}(s_v) > M_1 \mbox{ and } \tilde{\mathcal{L}}^N_{K_1}(s_i) \leq M_1 \mbox{ for } i \in \llbracket 1, v-1 \rrbracket \right\}.$$
We also set $E_N^{\mathsf{S}} = \bigsqcup_{v = 1}^m E_N^v$. We further define $\vec{x}^v, \vec{y} \in \weyl_{K_1}$ via
$$x_i^v = \tilde{\mathcal{L}}^N_{i}(s_v), \hspace{2mm} y_i = \tilde{\mathcal{L}}^N_{i}(\gamma) \mbox{ for } i\in \llbracket 1, K_1 \rrbracket,$$
and observe that on $E_N^v \cap (H^{\gamma}_N)^c$, we have 
\begin{equation*}
x_{K_1}^v > M_1, \hspace{2mm} y_{K_1} > - \hat{M}.
\end{equation*}
Combining the latter with (\ref{Eq.CorrDBM}), applied to $a = s_v$, $b = \gamma$, $t = \beta$, $P_1 = M_1$, $P_2 = - \hat{M}$, $g = \tilde{\mathcal{L}}_{K_1+1}[s_v, \gamma]$, and (\ref{Eq.LNNotTooHighM1}), we conclude that on $E_N^v \cap (H^{\gamma}_N)^c$, we have
\begin{equation}\label{Eq.IntermDBM}
\mathbb{P}^{s_v, \gamma, \vec{x}^v, \vec{y}, \infty, \tilde{\mathcal{L}}_{K_1+1}[s_v, \gamma]}_{\mathrm{avoid}}\left(  \mathcal{Q}_{K_1}(\beta) < \hat{M} \right) < \epsilon/4.
\end{equation}

We observe that we have the following tower of inequalities
\begin{equation}\label{Eq.LNNotTooHighTower}
\begin{split}
&\mathbb{P}\left( E_N^{v} \cap (\tilde{H}^{\beta}_N)^c \cap (H^{\gamma}_N)^c   \right) = \mathbb{E} \left[ {\bf 1}_{E_N^v \cap (H^{\gamma}_N)^c} \cdot \mathbb{E}\left[ {\bf 1}_{(\tilde{H}^{\beta}_N)^c} {\big \vert} \mathcal{F}_{\mathrm{ext}} ( \llbracket 1, K_1 \rrbracket \times (s_v,\gamma))\right] \right] \\
& = \mathbb{E} \left[ {\bf 1}_{E_N^v \cap (H^{\gamma}_N)^c} \cdot \mathbb{P}^{s_v, \gamma, \vec{x}^v, \vec{y}, \infty, \tilde{\mathcal{L}}_{K_1+1}[s_v, \gamma]}_{\mathrm{avoid}}\left(  \mathcal{Q}_{K_1}(\beta) < \hat{M} \right) \right]\\
& \leq \mathbb{E} \left[ {\bf 1}_{E_N^v \cap (H^{\gamma}_N)^c} \cdot \epsilon/4 \right] = \mathbb{P}\left(E_N^v \cap (H^{\gamma}_N)^c \right) \cdot \epsilon/4.
\end{split}
\end{equation}
Indeed, the first equality follows from the tower property for conditional expectation and the fact that $\tilde{H}^{\beta}_N, E_N^v \in \mathcal{F}_{\mathrm{ext}} ( \llbracket 1, K_1 \rrbracket \times (s_v,\gamma))$, where the latter $\sigma$-algebra is as in Definition \ref{Def.BGPVanilla}. The second equality follows from the Brownian Gibbs property of $\tilde{\mathcal{L}}^N$. The inequality on the third line follows from (\ref{Eq.IntermDBM}).

Summing (\ref{Eq.LNNotTooHighTower}) over $v \in \llbracket 1, m \rrbracket$ and using that $E_N^v$ are pairwise disjoint, we conclude
$$\mathbb{P}\left(E_N^{\mathsf{S}} \cap (\tilde{H}^{\beta}_N)^c \cap (H^{\gamma}_N)^c   \right) \leq \mathbb{P}\left(E_N^{\mathsf{S}} \cap (H^{\gamma}_N)^c \right) \cdot \epsilon/4 \leq \epsilon/4.$$
Taking suprema over $\mathsf{S}$ and using that $\tilde{\mathcal{L}}^N_{K_1}$ is continuous, we conclude 
\begin{equation*}
\mathbb{P}\left(\left\{ \sup_{s \in [\alpha, \beta]} \tilde{\mathcal{L}}^N_{K_1}(s) > M_1   \right\} \cap (\tilde{H}^{\beta}_N)^c \cap (H^{\gamma}_N)^c   \right) \leq \epsilon/4,
\end{equation*}
which together with (\ref{Eq.LNGamma}) and (\ref{Eq.LNGamma2}) implies (\ref{Eq.LNNotTooHigh}).\\

{\bf \raggedleft Step 3.} In this step, we prove (\ref{Eq.LNNotTooLow}).

Let $H^r_N$ be as in (\ref{Eq.LNGamma}), and note that we have
\begin{equation}\label{Eq.LNSidesOK}
\mathbb{P}\left(E_N^{\mathrm{side}}\right) < \epsilon/4, \mbox{ where } E_N^{\mathrm{side}} = H^{\alpha}_N \cup H^{\beta}_N.
\end{equation} 
Let $\vec{x}, \vec{y}, \vec{u}, \vec{v} \in \weyl_{K_2}$ be defined by
$$x_i = \tilde{\mathcal{L}}^N_{i}(\alpha), \hspace{2mm} y_i = \tilde{\mathcal{L}}^N_{i}(\beta), \hspace{2mm} u_i = v_i = - \hat{M} - i \mbox{, for } i\in \llbracket 1, K_2 \rrbracket.$$
From (\ref{Eq.LNTildeNonIntersecting}), we have on $(E_N^{\mathrm{side}})^c$ that 
\begin{equation}\label{Eq.LNSideOrdered}
x_i \geq u_i \mbox{ and } y_i \geq v_i \mbox{ for } i \in \llbracket 1, K_2 \rrbracket.
\end{equation}

Let $A$ be as in Lemma \ref{Lem.StayInCorridor} for $a = \alpha$, $b = \beta$, $k = K_2 (=\mathsf{M}_{k}^a)$, $M^{\mathrm{side}} = \hat{M} + K_2$, $\epsilon$ replaced with $\epsilon/4$, and introduce the event 
$$F_N = \left\{\inf_{t \in [\alpha, \beta]} \tilde{\mathcal{L}}^N_{K_2}(t) < -A \right\}.$$
We observe that we have the following tower of inequalities
\begin{equation}\label{Eq.LNNotTooLowTower}
\begin{split}
&\mathbb{P}\left(F_N \cap (E_N^{\mathrm{side}})^c \right) = \mathbb{E} \left[ {\bf 1}_{(E_N^{\mathrm{side}})^c} \cdot \mathbb{E}\left[ {\bf 1}_{F_N} {\big \vert} \mathcal{F}_{\mathrm{ext}} ( \llbracket 1, K_2 \rrbracket \times (\alpha,\beta))\right] \right] \\
& = \mathbb{E} \left[ {\bf 1}_{(E_N^{\mathrm{side}})^c} \cdot \mathbb{P}^{\alpha, \beta, \vec{x}, \vec{y}, \infty, \tilde{\mathcal{L}}_{K_2+1}[\alpha, \beta]}_{\mathrm{avoid}}\left( \inf_{t \in [\alpha, \beta]} \mathcal{Q}_{K_2}(t) < -A \right) \right]\\
& \leq \mathbb{E} \left[ {\bf 1}_{(E_N^{\mathrm{side}})^c} \cdot \mathbb{P}^{\alpha, \beta, \vec{u}, \vec{v}}_{\mathrm{avoid}}\left( \inf_{t \in [\alpha, \beta]} \mathcal{Q}_{K_2}(t) < -A \right) \right] \leq \mathbb{E} \left[ {\bf 1}_{(E_N^{\mathrm{side}})^c} \cdot \epsilon/4 \right] \leq \epsilon/4.  
\end{split}
\end{equation}
Indeed, the first equality follows from the tower property of conditional expectation and the fact that $E_N^{\mathrm{side}} \in \mathcal{F}_{\mathrm{ext}} ( \llbracket 1, K_2 \rrbracket \times (\alpha,\beta))$ (this $\sigma$-algebra is as in Definition \ref{Def.BGPVanilla}). The second equality follows from the Brownian Gibbs property satisfied by $\tilde{\mathcal{L}}^N$. In going from the second to the third line we used (\ref{Eq.LNSideOrdered}) and the monotone coupling in Lemma \ref{Lem.MonCoup}. The middle inequality on the third line used the definition of $A$, and Lemma \ref{Lem.StayInCorridor} with the same parameters as above (we note that $\vec{x}, \vec{y}$ in that lemma are given by $\vec{u}, \vec{v}$, which satisfy the conditions within that lemma by construction). The last inequality is trivial.

From (\ref{Eq.LNSidesOK}) and (\ref{Eq.LNNotTooLowTower}), we conclude
$$\mathbb{P}(F_N) = \mathbb{P}\left(F_N \cap (E_N^{\mathrm{side}})^c \right) + \mathbb{P}\left(F_N \cap E_N^{\mathrm{side}} \right) \leq \epsilon/4 + \mathbb{P}\left( E_N^{\mathrm{side}} \right) < \epsilon/2,$$
which establishes (\ref{Eq.LNNotTooLow}) with $M_2 = A$.
\end{proof}

The following result shows that if $J_a < \infty$, the curve $\mathcal{L}^{a,b,c}_{J_a + 1}$ is with high probability under a slightly elevated inverted parabola $-2^{-1/2} t^2$ on a large interval.  
\begin{proposition}\label{Prop.UnderParabola} Assume the same notation as in Definitions \ref{Def.DLP}, \ref{Def.ParMultiplicities}, and Proposition \ref{Prop.AWLE}, and suppose $J_a < \infty$. Fix $0< \alpha < \beta$, $\epsilon, \delta > 0$, and a sequence $T_N > 0$ with $T_N \uparrow \infty$. We can find $N_0 > 0$, depending on $\alpha, \beta, \epsilon, \delta$, the parameters $(a,b,c)$, and the sequence $\{T_N\}_{N \geq 1}$, such that for $N \geq N_0$, we have
\begin{equation}\label{Eq.UnderParabola}
\mathbb{P}\left( \sup_{t \in [\alpha, \beta]} \left( \mathcal{L}^{a,b,c}_{J_a + 1}(tT_N) + 2^{-1/2} t^2 T_N^2 \right) > \delta T_N^2    \right) < \epsilon.
\end{equation} 
\end{proposition}
\begin{proof} Throughout the proof all the constants we encounter depend on $\alpha, \beta, \epsilon, \delta$, the parameters $(a,b,c)$, and the sequence $\{T_N\}_{N \geq 1}$, and all inequalities hold provided that $N$ is sufficiently large, depending on the same set of parameters. We do not mention this further. For clarity, we split the proof into two steps.\\

{\bf \raggedleft Step 1.} Put $\gamma = \beta + 1$, and split the interval $[\alpha, \gamma]$ into $n \geq 1$ intervals $ [t_{i-1}, t_i]$ of equal length $\rho_n = (\gamma - \alpha)/n$, where $ \alpha = t_0 < t_1 < \cdots < t_n = \gamma$. Set $f(t) = -2^{-1/2} t^2$, and let $n \geq 1$ be sufficiently large, depending on $\alpha, \beta$, so that $\rho_n < 1$ and for each $i \in \llbracket 1, n-1 \rrbracket$, $t \in [t_{i-1}, t_i]$, we have 
\begin{equation}\label{Eq.ConvexHigh}
\left(f(t) + \delta\right) \cdot \frac{t_{i+1} - t_i}{t_{i+1} - t} + f(t_{i+1}) \cdot \frac{t_i - t}{t_{i+1} - t} \geq f(t_i) + \frac{\delta}{4}.   
\end{equation}
Let us briefly explain why such a choice of $n$ is possible. Put $\rho = t_i - t \in [0, \rho_n]$ and note that (\ref{Eq.ConvexHigh}) is equivalent to
$$ -2^{-1/2} (t_i - \rho)^2 \cdot \frac{\rho_n}{\rho + \rho_n} + \delta \cdot \frac{\rho_n}{\rho + \rho_n} - 2^{-1/2} (t_i + \rho_n)^2 \cdot \frac{\rho}{\rho + \rho_n} \geq - 2^{-1/2} t_i^2 + \frac{\delta}{4},$$
which upon rearranging is equivalent to
$$ -2^{-1/2} \rho \rho_n + \delta \cdot \frac{\rho_n}{\rho + \rho_n} \geq \frac{\delta}{4}.$$
The latter clearly holds if $n$ is large enough so that $2^{-1/2} \rho_n^2 < \delta/4$, as $\rho \in [0, \rho_n]$.

We claim that for each $i_0 \in \llbracket 1, n-1 \rrbracket$, we can find $W_{i_0} \in \mathbb{N}$, such that for $N \geq W_{i_0}$, we have
\begin{equation}\label{Eq.UnderParabolaI}
\mathbb{P}\left( \sup_{t \in [t_{i_0-1}, t_{i_0}]} \left( \mathcal{L}^{a,b,c}_{J_a + 1}(tT_N) + 2^{-1/2} t^2 T_N^2 \right) > \delta T_N^2 \right) < \frac{\epsilon}{n}.
\end{equation} 
Since by construction $[\alpha, \beta] \subseteq \cup_{i_0 = 1}^{n-1} [t_{i_0-1}, t_{i_0}]$, we see that (\ref{Eq.UnderParabola}) holds with $N_0 = \max_{1 \leq i_0 \leq n-1} W_{i_0}$, by taking a union bound of (\ref{Eq.UnderParabolaI}). We have thus reduced the proof to establishing (\ref{Eq.UnderParabolaI}).\\

In the remainder of the proof we fix $i_0 \in \llbracket 1, n-1 \rrbracket$. By Proposition \ref{Prop.Slopes}(c), we can find $M_1 > 0$ sufficiently large, so that for $i \in \{i_0, i_0 + 1 \}$, we have
\begin{equation}\label{Eq.NotTooHighGrid}
\mathbb{P}(F_N^i) < (4n)^{-1} \cdot \epsilon, \mbox{ where } F_N^i = \left\{ \left|\mathcal{L}^{a,b,c}_{J_a + 1}(t_iT_N) + 2^{-1/2} t_i^2 T_N^2 \right| >  M_1 \right\}.
\end{equation}

Let $A$ be as in Lemma \ref{Lem.NotTooLow} for $k = J_{a} + 1$ and $\epsilon$ replaced with $(2n)^{-1} \cdot \epsilon$. From (\ref{Eq.NotTooLow}), we conclude that if $\vec{x}, \vec{y} \in \weyl_{J_a+1}$ satisfy $x_{J_a+ 1} > \max(P_1, g(a))$, $y_{J_a+1} > \max(P_2,g(b))$ for some $P_1, P_2 \in \mathbb{R}$, and continuous $g$, then for each $t \in [a,b]$
\begin{equation}\label{Eq.GridBrownianEnsemble}
\mathbb{P}_{\mathrm{avoid}}^{a,b, \vec{x}, \vec{y}, \infty, g}\left( \mathcal{Q}_{J_a+1}(t)  \leq P_1  \cdot \frac{b - t}{b - a} + P_2 \cdot \frac{t - a}{b - a} - A \sqrt{b-a}   \right) < \frac{\epsilon}{2n}.
\end{equation}
We suppose that $W_{i_0}$ is sufficiently large, so that for $N \geq W_{i_0}$ and all $s\in [t_{i_0-1}, t_{i_0}]$, we have 
\begin{equation}\label{Eq.GridLargeN}
A\sqrt{T_N(t_{i_0+1}-s)} + M_1 \leq (\delta/4)T_N^2. 
\end{equation}
In the next step, we prove (\ref{Eq.UnderParabolaI}) for this choice of $W_{i_0}$.\\

{\bf \raggedleft Step 2.} The argument we present here is similar to Step 2 of the proof of Proposition \ref{Prop.CloseToZero}.

Fix $\mathsf{S} = \{s_1, \dots, s_m\} \subset [t_{i_0-1}, t_{i_0}]$ with $s_1 < s_2  < \cdots < s_m$, and define the events 
$$E_N^v = \left\{ \mathcal{L}^{a,b,c}_{J_a + 1}(s_vT_N) + 2^{-1/2} s_v^2 T_N^2 > \delta T_N^2 \mbox{ and } \mathcal{L}^{a,b,c}_{J_a + 1}(s_iT_N) + 2^{-1/2} s_i^2 T_N^2 \leq \delta T_N^2 \mbox{ for } i \in \llbracket 1, v-1 \rrbracket \right\}.$$
We also set $E_N^{\mathsf{S}} = \bigsqcup_{v = 1}^m E_N^v$. We further define $\vec{x}^v, \vec{y} \in \weyl_{J_a + 1}$ via
$$x_i^v = \mathcal{L}^{a,b,c}_{i}(s_vT_N), \hspace{2mm} y_i = \mathcal{L}^{a,b,c}_{i}(t_{i_0+1} T_N) \mbox{ for } i\in \llbracket 1, J_a + 1 \rrbracket,$$
and observe that on $E_N^v \cap (F_N^{i_0+1})^c $, we have 
\begin{equation}\label{Eq.GridBoundaryLB}
x_{J_a+1}^v > - 2^{-1/2} s_v^2 T_N^2 + \delta T_N^2, \hspace{2mm} y_{J_a+1} \geq - 2^{-1/2} t_{i_0+1}^2 T_N^2 -  M_1.
\end{equation}

If we set 
$$P_1 = - 2^{-1/2} s_v^2 T_N^2 + \delta T_N^2, \hspace{2mm} P_2 = - 2^{-1/2} t_{i_0+1}^2 T_N^2 -  M_1,$$
$$a_v = s_v T_N, \hspace{2mm} b = t_{i_0 + 1} T_N, \hspace{2mm} t = t_{i_0} T_N, \hspace{2mm} g_v(s) = \mathcal{L}^{a,b,c}_{J_a + 2}(s) \mbox{ for } s \in [s_v T_N, t_{i_0+1} T_N],$$
we conclude that on $E_N^v \cap (F_N^{i_0+1})^c$, we have for $N \geq W_{i_0}$ the following tower of inequalities
\begin{equation}\label{Eq.GridMiniTower}
\begin{split}
&\mathbb{P}_{\mathrm{avoid}}^{a_v,b, \vec{x}^v, \vec{y}, \infty, g_v}\left( \mathcal{Q}_{J_a+1}(t)  \leq -2^{-1/2}t_{i_0}^2T_N^2 + M_1\right) \\
& = \mathbb{P}_{\mathrm{avoid}}^{a_v,b, \vec{x}^v, \vec{y}, \infty, g_v}\left( \mathcal{Q}_{J_a+1}(t)  \leq f(t_{i_0}) T_N^2 + M_1\right) \\
& \leq \mathbb{P}_{\mathrm{avoid}}^{a_v,b, \vec{x}^v, \vec{y}, \infty, g_v}\left( \mathcal{Q}_{J_a+1}(t)  \leq \left[\left(f(s_v) + \delta\right)  \frac{t_{i_0+1} - t_{i_0}}{t_{i_0+1} - s_v} + f(t_{i_0+1}) \frac{t_{i_0} - s_v}{t_{i_0+1} - s_v} - \frac{\delta}{4}\right]\hspace{-1mm}T_N^2 +  M_1\right)\\
& = \mathbb{P}_{\mathrm{avoid}}^{a_v,b, \vec{x}^v, \vec{y}, \infty, g_v}\left( \mathcal{Q}_{J_a+1}(t)  \leq P_1  \frac{t_{i_0+1} - t_{i_0}}{t_{i_0+1} - s_v} + P_2 \frac{t_{i_0} - s_v}{t_{i_0+1} - s_v} - \frac{\delta T_N^2}{4} +M_1 \right) \\
& \leq \mathbb{P}_{\mathrm{avoid}}^{a_v,b, \vec{x}^v, \vec{y}, \infty, g_v}\left( \mathcal{Q}_{J_a+1}(t)  \leq P_1  \frac{t_{i_0+1} - t_{i_0}}{t_{i_0+1} - s_v} + P_2 \frac{t_{i_0} - s_v}{t_{i_0+1} - s_v} - A\sqrt{T_N(t_{i_0+1}-s_v)} \right)< \frac{\epsilon}{2n}.
\end{split}
\end{equation}
In going from the first to the second line we used that $f(t) = -2^{-1/2}t^2$, and in going from the second to the third line we used (\ref{Eq.ConvexHigh}). In going from the third to the fourth line, we used the definitions of $P_1, P_2$, and in going from the fourth line to the fifth line, we used that $N \geq W_{i_0}$ and (\ref{Eq.GridLargeN}). In the last inequality we used (\ref{Eq.GridBrownianEnsemble}) and (\ref{Eq.GridBoundaryLB}).  \\

With the same notation as above, we have the following tower of inequalities for $N \geq W_{i_0}$ 
\begin{equation}\label{Eq.GridMainTower}
\begin{split}
&\mathbb{P}\left( E_N^{v} \cap (F_N^{i_0})^c \cap (F_N^{i_0+1})^c   \right) = \mathbb{E} \left[ {\bf 1}_{E_N^{v} \cap (F_N^{i_0+1})^c} \cdot \mathbb{E}\left[ {\bf 1}_{(F_N^{i_0})^c} {\Big \vert} \mathcal{F}_{\mathrm{ext}} ( \llbracket 1, J_a+1 \rrbracket \times (a_v,b))\right] \right] \\
& = \mathbb{E} \left[ {\bf 1}_{E_N^{v} \cap (F_N^{i_0+1})^c} \cdot \mathbb{P}^{a_v, b, \vec{x}^v, \vec{y}, \infty, g_v}_{\mathrm{avoid}}\left(  \left|\mathcal{Q}_{J_a+1}(t_{i_0}T_N) + 2^{-1/2}t_{i_0}^2T_N^2 \right| \leq  M_1  \right) \right]\\
& \leq \mathbb{E} \left[ {\bf 1}_{E_N^{v} \cap (F_N^{i_0+1})^c} \cdot \frac{\epsilon}{2n} \right] \leq \frac{\epsilon}{2n} \cdot \mathbb{P}\left( E_N^{v} \cap (F_N^{i_0+1})^c \right). 
\end{split}
\end{equation}
The first equality follows from the tower property of conditional expectation and the fact that $E_N^{v}, F_N^{i_0+1} \in \mathcal{F}_{\mathrm{ext}} ( \llbracket 1, J_a+1 \rrbracket \times (a_v,b))$ (this $\sigma$-algebra is as in Definition \ref{Def.BGPVanilla}). The second equality follows from the Brownian Gibbs property for $\mathcal{L}^{a,b,c}$, see Proposition \ref{Prop.AWLE}. In going from the second to the third line, we used (\ref{Eq.GridMiniTower}).

Summing (\ref{Eq.GridMainTower}) over $v \in \llbracket 1, m \rrbracket$ and using that $E_N^v$ are pairwise disjoint, we conclude
$$\mathbb{P}\left( E_N^{\mathsf{S}} \cap (F_N^{i_0})^c \cap (F_N^{i_0+1})^c   \right) \leq \mathbb{P}\left( E_N^{\mathsf{S}}\cap (F_N^{i_0+1})^c \right) \cdot \frac{\epsilon}{2n} \leq \frac{\epsilon}{2n}.$$
Taking suprema over $\mathsf{S}$ and using that $\mathcal{L}^{a,b,c}_{J_a+1}$ is continuous, we conclude 
\begin{equation*}
\mathbb{P}\left(\left\{ \sup_{t \in [t_{i_0-1}, t_{i_0}]} \left( \mathcal{L}^{a,b,c}_{J_a + 1}(tT_N) + 2^{-1/2} t^2 T_N^2 \right) > \delta T_N^2  \right\} \cap (F_N^{i_0})^c \cap (F_N^{i_0+1})^c   \right) \leq \frac{\epsilon}{2n},
\end{equation*}
which together with (\ref{Eq.NotTooHighGrid}) implies (\ref{Eq.UnderParabolaI}).
\end{proof}

%
%
\subsection{Proof of Theorem \ref{Thm.ConvDBM}}\label{Section5.2} We continue with the same notation as in the statement of the theorem. For clarity, we split the proof into three steps. In Step 1, we reduce the proof of the theorem to showing that the modulus of continuity of $\mathcal{L}_i^N$ is well-behaved with high probability for each $i \in \llbracket 1, \mathsf{m}_k^a \rrbracket$, see (\ref{Eq.MOCDBMRed}). In Step 2, we summarize several key consequences of our results in Sections \ref{Section4} and \ref{Section5.1}. In particular, we show that if $\tilde{\mathcal{L}}^N = \{\tilde{\mathcal{L}}^N_i\}_{i \geq 1}$ is as in (\ref{Eq.LNTilde}), then on any compact interval $\tilde{\mathcal{L}}^N_{\mathsf{M}_{k-1}^a}$ is very high and $\tilde{\mathcal{L}}^N_{\mathsf{M}_{k}^a+1}$ is very low with high probability, see (\ref{Eq.TopBotDBM}). The latter implies that the curves $\{\tilde{\mathcal{L}}^N_i: i \in \llbracket \mathsf{M}_{k-1}^a + 1, \mathsf{M}_{k}^a\rrbracket \}$ are too far away from the other curves of the ensemble $\tilde{\mathcal{L}}^N$ and behave like a finite Brownian ensemble. As such ensembles have a well-behaved modulus of continuity by Lemma \ref{Lem.MOCUniform}, we are able to infer the same for $\{\tilde{\mathcal{L}}^N_i: i \in \llbracket \mathsf{M}_{k-1}^a + 1, \mathsf{M}_{k}^a\rrbracket \}$, and hence for $\{\mathcal{L}^N_i : i \in \llbracket 1, \mathsf{m}_{k}^a\rrbracket\}$ in view of (\ref{Eq.LNMatch}). The details behind the argument we outlined in the last sentence are provided in Step 3, where we ultimately prove (\ref{Eq.MOCDBMRed}).\\

{\bf \raggedleft Step 1.} We observe that part (b) of the theorem follows from part (a) and Proposition \ref{Prop.BasicProperties}(b). Consequently, we only need to establish part (a). In what follows we fix $k \in \mathbb{N}$ and $(a,b,c) \in \parP$, such that $|\mathsf{V}_a| \geq k$, and proceed to prove (\ref{Eq.ConvDBMA}).

We claim that for each $\epsilon, \eta > 0$, and $\beta > \alpha > 0$, we can find $N_0 \in \mathbb{N}$ and $\delta > 0$, depending on $\epsilon, \eta, \beta, \alpha$, the parameters $(a,b,c)$, and the sequence $\{T_N\}_{N \geq 1}$, such that for $N \geq N_0$ and $i \in \llbracket 1, \mathsf{m}_k^a \rrbracket$
\begin{equation}\label{Eq.MOCDBMRed}
\mathbb{P}\left(  w(\mathcal{L}^N_i[\alpha,\beta], \delta) \geq \eta \right) < \epsilon,
\end{equation}
where we recall that the modulus of continuity $w(\mathcal{L}^N_i[\alpha,\beta], \delta)$ is as in (\ref{Eq.MOCDef}) with $a = \alpha$ and $b = \beta$. We prove (\ref{Eq.MOCDBMRed}) in the steps below. Here, we assume its validity and finish the proof of part (a).

From Proposition \ref{Prop.FDSlope}, we know
\begin{equation}\label{Eq.FDConvDBMRed}
\left(\mathcal{L}^N_i(t): t > 0 \mbox{, }i = 1, \dots, \mathsf{m}_k^a \right) \overset{f.d.}{\rightarrow}\left(\lambda^{\mathsf{m}_k^a}_i(t): t > 0 \mbox{, }i = 1, \dots, \mathsf{m}_k^a\right).
\end{equation}
Equation (\ref{Eq.FDConvDBMRed}) in particular shows that the sequences $\{\mathcal{L}^N_i(1)\}_{N \geq 1}$ for each $i \in \llbracket 1, \mathsf{m}_k^a \rrbracket$ are tight, which together with (\ref{Eq.MOCDBMRed}) shows that the tightness criteria of \cite[Lemma 2.4]{DEA21} are satisfied by $\mathcal{L}^N = \{\mathcal{L}_i^N\}_{i = 1}^{\mathsf{m}_k^a}$. 

From \cite[Lemma 2.4]{DEA21}, we conclude that $\{\mathcal{L}^N\}_{N \geq 1}$ is a tight sequence of random elements in $C(\llbracket 1,\mathsf{m}_k^a\rrbracket \times (0, \infty) )$. In addition, (\ref{Eq.FDConvDBMRed}) shows that any subsequential limit has the same finite-dimensional distributions as $\{\lambda^{\mathsf{m}_k^a}_i\}_{i = 1}^{\mathsf{m}_k^a}$. As finite-dimensional sets form a separating class, see \cite[Lemma 3.1]{DimMat}, we conclude that $\mathcal{L}^N \Rightarrow \{\lambda^{\mathsf{m}_k^a}_i\}_{i = 1}^{\mathsf{m}_k^a}$, which is precisely (\ref{Eq.ConvDBMA}).\\

{\bf \raggedleft Step 2.} In this step, we specify the choices of $N_0$ and $\delta$, for which we will prove (\ref{Eq.MOCDBMRed}). 

From (\ref{Eq.FDConvDBMRed}), we can find $M_1 > 0$, depending on $\epsilon$, such that for all $N \geq 1$, 
\begin{equation}\label{Eq.EsideDBM}
\mathbb{P}\left(E^{\mathrm{side}}_N\right) < \epsilon/4, \mbox{ where } E^{\mathrm{side}}_N = \{|\mathcal{L}_i^N(\gamma)| > M_1 \mbox{ for some } (i,\gamma) \in \llbracket 1, \mathsf{m}_k^a \rrbracket \times \{\alpha, \beta\}\}.
\end{equation}
From Lemma \ref{Lem.StayInCorridor} with $a = \alpha$, $b = \beta$, $k = \mathsf{m}_k^a$, $M^{\mathsf{side}} = M_1$ and $\epsilon = 1/2$, we can find $A > 0$, such that for $\vec{x}, \vec{y} \in \weyl_{\mathsf{m}_k^a}$ with $|x_i|, |y_i| \leq M^{\mathsf{side}}$ for $i \in \llbracket 1, \mathsf{m}_k^a\rrbracket$, we have 
\begin{equation}\label{Eq.StayInCorridorDBMRed}
\mathbb{P}_{\mathrm{avoid}}^{\alpha,\beta, \vec{x}, \vec{y}}\left( |\mathcal{Q}_i(t) | \geq A \mbox{ for some } (i,t) \in \llbracket 1,\mathsf{m}_k^a \rrbracket \times [\alpha,\beta]  \right) < 1/2.
\end{equation}  

From Lemma \ref{Lem.MOCUniform} with $a = \alpha$, $b = \beta$, $k = \mathsf{m}_k^a$, $M^{\mathsf{side}} = M_1$, $\eta$ as above and $\epsilon$ replaced with $\epsilon/4$, we can find $\delta > 0$, such that for $\vec{x}, \vec{y} \in \weyl_{\mathsf{m}_k^a}$ with $|x_i|, |y_i| \leq M^{\mathsf{side}}$ for $i \in \llbracket 1, \mathsf{m}_k^a\rrbracket$, we have
\begin{equation}\label{Eq.MOCUniformDBM}
\mathbb{P}_{\mathrm{avoid}}^{\alpha,\beta, \vec{x}, \vec{y}}\left( \max_{i \in \llbracket 1, \mathsf{m}_k^a \rrbracket} w(\mathcal{Q}_i[\alpha, \beta], \delta) \geq \eta \right) < \epsilon/4.
\end{equation}
This specifies our choice of $\delta$. \\

We next let $\tilde{\mathcal{L}}^N = \{\tilde{\mathcal{L}}^N_i\}_{i \geq 1}$ be as in (\ref{Eq.LNTilde}), and we adopt the convention $\tilde{\mathcal{L}}^N_0 = \infty$. As we explain below, we have from the results in Section \ref{Section5.1}, that
\begin{equation}\label{Eq.NoCeilingDBM}
\lim_{N \rightarrow \infty} \mathbb{P} \left( \inf_{t \in [\alpha, \beta]}\tilde{\mathcal{L}}^N_{\mathsf{M}_{k-1}^a}(t) > A  \right) = 1.
\end{equation}
and
\begin{equation}\label{Eq.NoFloorDBM}
\lim_{N \rightarrow \infty} \mathbb{P} \left( \sup_{t \in [\alpha, \beta]}\tilde{\mathcal{L}}^N_{\mathsf{M}_{k}^a+1}(t) < - A  \right) = 1.
\end{equation}

Indeed, if $k = 1$ and hence $\mathsf{M}_{k-1}^a = 0$, we have $\tilde{\mathcal{L}}_0^N = \infty$ and (\ref{Eq.NoCeilingDBM}) is automatic. If $k \geq 2$, then from Proposition \ref{Prop.CloseToZero} applied to $k-1$, we have that the sequence 
$$\sup_{t \in [\alpha, \beta]}\left|\tilde{\mathcal{L}}^N_{\mathsf{M}_{k-1}^a}(t) + t (2T_N)^{1/2} \cdot (1/v_{k-1}^a - 1/v_{k}^a) \right|$$
is tight. Since $v_{k-1}^a > v_k^a$, and $T_N \uparrow \infty$, the latter implies (\ref{Eq.NoCeilingDBM}).

To see why (\ref{Eq.NoFloorDBM}) holds, we again consider two cases. If $|\mathsf{V}_a| > k$, then from Proposition \ref{Prop.CloseToZero} applied to $k+1$, we have that the sequence 
$$\sup_{t \in [\alpha, \beta]}\left|\tilde{\mathcal{L}}^N_{\mathsf{M}_{k}^a + 1}(t) + t (2T_N)^{1/2} \cdot (1/v_{k+1}^a - 1/v_{k}^a) \right|,$$
is tight. Since $v_{k+1}^a < v_k^a$, and $T_N \uparrow \infty$, the latter implies (\ref{Eq.NoFloorDBM}). If $|\mathsf{V}_a| = k$, then $\mathsf{M}_{k}^a = J_a$, and by Proposition \ref{Prop.UnderParabola}, we have for any $\delta_0 > 0$
$$ \lim_{N \rightarrow \infty}\mathbb{P}\left( \sup_{t \in [\alpha, \beta]} \left( \mathcal{L}^{a,b,c}_{J_a + 1}(tT_N) + 2^{-1/2} t^2 T_N^2 \right) > \delta_0 T_N^2    \right) = 0.
$$
Combining the latter with (\ref{Eq.LNTilde}), we obtain for any $\delta_0 > 0$
$$ \lim_{N \rightarrow \infty}\mathbb{P}\left( \sup_{t \in [\alpha, \beta]} \left(  \tilde{\mathcal{L}}^N_{\mathsf{M}_{k}^a+1}(t)-2^{1/2}tT_N^{1/2}/v_k^a  + 2^{-1/2} t^2 T_N^{3/2} \right) > \delta_0 T_N^{3/2}    \right) = 0.
$$
Picking $\delta_0 < 2^{-1/2} \alpha^2$, we see that the last equation implies (\ref{Eq.NoFloorDBM}).

By combining (\ref{Eq.NoCeilingDBM}) and (\ref{Eq.NoFloorDBM}), we see that we can find $N_0 \in \mathbb{N}$, such that for all $N \geq N_0$, 
\begin{equation}\label{Eq.TopBotDBM}
\mathbb{P}\left(F_N^{\mathrm{tb}}\right) < \epsilon/4, \mbox{ where }F_N^{\mathrm{tb}} = \left\{ \inf_{t \in [\alpha, \beta]}\tilde{\mathcal{L}}^N_{\mathsf{M}_{k-1}^a}(t) \leq A \right\} \cup \left\{ \sup_{t \in [\alpha, \beta]}\tilde{\mathcal{L}}^N_{\mathsf{M}_{k}^a+1}(t) \geq - A\right\}.
\end{equation}
This specifies our choice of $N_0$.\\

{\bf \raggedleft Step 3.} In this final step, we prove (\ref{Eq.MOCDBMRed}) for $N_0, \delta$ as in Step 2.

We define $\vec{x}^N, \vec{y}^N \in \weyl_{\mathsf{m}_k^a}$, $f^N: [\alpha, \beta] \rightarrow (-\infty, \infty]$ and $g^N: [\alpha, \beta] \rightarrow [-\infty, \infty)$ through
$$x_i^N = \tilde{\mathcal{L}}_{\mathsf{M}_{k-1}^a + i}^N(\alpha), \hspace{2mm} y_i^N = \tilde{\mathcal{L}}_{\mathsf{M}_{k-1}^a + i}^N(\beta), \mbox{ for $i \in \llbracket 1, \mathsf{m}_k^a \rrbracket$}, \hspace{2mm} f^N = \tilde{\mathcal{L}}_{\mathsf{M}_{k-1}^a}^N[\alpha, \beta], \hspace{2mm} g^N = \tilde{\mathcal{L}}_{\mathsf{M}_{k}^a + 1}^N[\alpha, \beta],$$
and observe that on $(E_N^{\mathrm{side}})^c \cap (F_N^{\mathrm{tb}})^c$, we have 
\begin{equation}\label{Eq.BoundaryDBM}
\left|x_i^N\right|, \left|y_i^N\right| \leq M_1, \mbox{ for $i \in \llbracket 1, \mathsf{m}_k^a \rrbracket$}, \hspace{2mm} f^N(t) > A, \hspace{2mm} g^N(t) < - A, \mbox{ for $t \in [\alpha, \beta]$}.  
\end{equation}

We also define for any $f: [\alpha, \beta] \rightarrow (-\infty, \infty]$ and $g: [\alpha, \beta] \rightarrow [-\infty, \infty)$, the function $G_{f,g}: C( \llbracket 1, \mathsf{m}_k^a \rrbracket \times [\alpha, \beta] ) \rightarrow [0,1]$ via
$$G_{f,g}\left( \{h_i\}_{i = 1}^{\mathsf{m}_k^a} \right) = {\bf 1} \left\{f(t) > h_1(t) \mbox{ and } h_{\mathsf{m}_k^a}(t) > g(t) \mbox{ for all } t \in [\alpha, \beta] \right\}.$$

We observe that we have the following tower of inequalities on $(E_N^{\mathrm{side}})^c \cap (F_N^{\mathrm{tb}})^c$ for $N \geq N_0$
\begin{equation}\label{Eq.DBMConvMiniTower}
\begin{split}
&\mathbb{P}_{\mathrm{avoid}}^{\alpha,\beta, \vec{x}^N, \vec{y}^N, f^N, g^N}\left( \max_{i \in \llbracket 1, \mathsf{m}_k^a \rrbracket} w(\mathcal{Q}_i[\alpha, \beta], \delta) \geq \eta \right) \\
& = \frac{\mathbb{E}_{\mathrm{avoid}}^{\alpha,\beta, \vec{x}^N, \vec{y}^N}\left[ {\bf 1}\left\{\max_{i \in \llbracket 1, \mathsf{m}_k^a \rrbracket} w(\mathcal{Q}_i[\alpha, \beta], \delta) \geq \eta \right\} \cdot G_{f^N, g^N} (\mathcal{Q}) \right]}{\mathbb{E}_{\mathrm{avoid}}^{\alpha,\beta, \vec{x}^N, \vec{y}^N}\left[ G_{f^N, g^N} (\mathcal{Q}) \right]} \\
& \leq 2 \mathbb{E}_{\mathrm{avoid}}^{\alpha,\beta, \vec{x}^N, \vec{y}^N}\left[ {\bf 1}\left\{\max_{i \in \llbracket 1, \mathsf{m}_k^a \rrbracket} w(\mathcal{Q}_i[\alpha, \beta], \delta) \geq \eta \right\} \cdot G_{f^N, g^N} (\mathcal{Q})  \right] < \epsilon/2.
\end{split}
\end{equation}
Indeed, in going from the first to the second line, we used Definition \ref{Def.fgAvoidingBE}, and in going from the second to the third line, we used 
$$\mathbb{E}_{\mathrm{avoid}}^{\alpha,\beta, \vec{x}^N, \vec{y}^N}\left[ G_{f^N, g^N} (\mathcal{Q}) \right] \geq \mathbb{E}_{\mathrm{avoid}}^{\alpha,\beta, \vec{x}^N, \vec{y}^N}\left[ G_{A, -A} (\mathcal{Q}) \right] > 1/2,$$
where the first inequality follows from (\ref{Eq.BoundaryDBM}), while the second from (\ref{Eq.StayInCorridorDBMRed}). The last inequality in (\ref{Eq.DBMConvMiniTower}) follows from (\ref{Eq.MOCUniformDBM}) and the fact that $G_{f^N, g^N}(\mathcal{Q}) \in [0,1]$.

With the same notation as above, and $E_N = E_N^{\mathrm{side}} \cup F_N^{\mathrm{tb}}$, we have the following tower of inequalities for $N \geq N_0$ 
\begin{equation}\label{Eq.DBMConvTower}
\begin{split}
&\mathbb{P}\left( E_N^c \cap \left\{ \max_{i \in \llbracket 1, \mathsf{m}_k^a \rrbracket} w(\tilde{\mathcal{L}}^N_{i + \mathsf{M}_{k-1}^a}[\alpha,\beta], \delta) \geq \eta \right\} \right)  \\
& = \mathbb{E}\left[{\bf 1}_{E_N^c }\cdot \mathbb{E}\left[ \left\{ \max_{i \in \llbracket 1, \mathsf{m}_k^a \rrbracket} w(\tilde{\mathcal{L}}^N_{i+ \mathsf{M}_{k-1}^a}[\alpha,\beta], \delta) \geq \eta \right\} \bigg{\vert} \mathcal{F}_{\mathrm{ext}} ( \llbracket \mathsf{M}^a_{k-1}+1, \mathsf{M}^a_{k} \rrbracket \times (\alpha, \beta)) \right] \right] \\
& = \mathbb{E}\left[{\bf 1}_{E_N^c} \cdot \mathbb{P}_{\mathrm{avoid}}^{\alpha,\beta, \vec{x}^N, \vec{y}^N, f^N, g^N}\left( \max_{i \in \llbracket 1, \mathsf{m}_k^a \rrbracket} w(\mathcal{Q}_i[\alpha, \beta], \delta) \geq \eta \right) \right] \leq \mathbb{E}\left[{\bf 1}_{E_N^c} \cdot (\epsilon/2)\right] \leq \epsilon/2.
\end{split}
\end{equation}
Indeed, the first equality follows from the tower property for conditional expectation and the fact that $E_N \in \mathcal{F}_{\mathrm{ext}} ( \llbracket \mathsf{M}^a_{k-1}+1, \mathsf{M}^a_{k} \rrbracket \times (\alpha, \beta)) $ (this $\sigma$-algebra is as in Definition \ref{Def.BGPVanilla}). The second equality follows from the Brownian Gibbs property for $\tilde{\mathcal{L}}^N$, and the first inequality on the third line follows from (\ref{Eq.DBMConvMiniTower}).

Combining (\ref{Eq.EsideDBM}), (\ref{Eq.TopBotDBM}) and (\ref{Eq.DBMConvTower}), we conclude for $N \geq N_0$ that 
$$ \mathbb{P}\left( \max_{i \in \llbracket 1, \mathsf{m}_k^a \rrbracket} w(\tilde{\mathcal{L}}^N_{i + \mathsf{M}_{k-1}^a}[\alpha,\beta], \delta) \geq \eta \right) \leq \epsilon/2 + \mathbb{P}(E_N^{\mathrm{side}}) + \mathbb{P}(F_N^{\mathrm{tb}}) < \epsilon.$$
The last inequality and $\mathcal{L}_i^N(t)  = \tilde{\mathcal{L}}^N_{\mathsf{M}_{k-1}^a+ i}(t) \mbox{ for } (i,t) \in \llbracket 1, \mathsf{m}_k^a \rrbracket \times (0, \infty)$, see (\ref{Eq.LNMatch}), imply (\ref{Eq.MOCDBMRed}).

%
%
\section{Convergence to the Airy line ensemble}\label{Section6} The goal of this section is to prove Theorem \ref{Thm.ConvAiry}. We do so in Section \ref{Section6.2} after establishing one technical proposition in Section \ref{Section6.1}. Throughout this section we continue with the same notation as in Sections \ref{Section1} and \ref{Section4}.

%
%
\subsection{No big max}\label{Section6.1} In this section, we establish one technical result which we require for the proof of Theorem \ref{Thm.ConvAiry}. Assume the same notation as in Theorem \ref{Thm.ConvAiry}(a), and introduce the auxiliary line ensemble $\hat{\mathcal{L}}^N = \{\hat{\mathcal{L}}^N_i\}_{i \geq 1}$ by
\begin{equation}\label{Eq.LNHatDef}
\hat{\mathcal{L}}^N_i(t) = 2^{-1/2} \left(\mathcal{A}^{a,b,c}_{i}(t + T_N) - t^2\right) \mbox{ for } (i,t) \in \mathbb{N} \times \mathbb{R}.
\end{equation}
By Proposition \ref{Prop.AWLE}, we know that $\mathcal{L}^{a,b,c}$ satisfies the Brownian Gibbs property on $\mathbb{R}$, and then from Lemma \ref{Lem.Affine} we conclude the same for $\hat{\mathcal{L}}^N$. In addition, we have from (\ref{Eq.ScaledFlatA}) that 
\begin{equation}\label{Eq.LNHatMatch}
\mathcal{A}_i^N(t) = 2^{1/2} \cdot \hat{\mathcal{L}}^N_{J_a+ i}(t) + t^2  \mbox{ for } (i,t) \in \mathbb{N} \times \mathbb{R}.
\end{equation}

The following result shows that the curves $\hat{\mathcal{L}}_{J_a + k}^N$ are likely not too high on any fixed interval.
\begin{proposition}\label{Prop.NoBigMax} Assume the same notation as in Theorem \ref{Thm.ConvAiry}(a) and let $\hat{\mathcal{L}}^N = \{\hat{\mathcal{L}}^N_i\}_{i \geq 1}$  be as in (\ref{Eq.LNHatDef}). Fix $ \beta, \epsilon > 0$, and $k \in \mathbb{N}$. We can find $M > 0$, depending on $\beta$, $\epsilon$, the parameters $(a,b,c)$, and the sequence $\{T_N\}_{N \geq 1}$, such that for all $N \in \mathbb{N}$, we have
\begin{equation}\label{Eq.NoBigMax}
\mathbb{P}\left( \sup_{t \in [-\beta, \beta]} \hat{\mathcal{L}}_{J_a + k}^N(t) > M \right) < \epsilon.
\end{equation} 
\end{proposition}
\begin{proof} Throughout the proof, all the constants we encounter depend on $ \beta, \epsilon$, the parameters $(a,b,c)$, and the sequence $\{T_N\}_{N \geq 1}$. For clarity, we split the proof into two steps.\\

{\bf \raggedleft Step 1.} As $\hat{\mathcal{L}}^N$ satisfies the Brownian Gibbs property, we have for all $t \in \mathbb{R}$ 
$$\hat{\mathcal{L}}^N_1(t) > \hat{\mathcal{L}}^N_2(t) > \cdots.$$
Consequently, it suffices to find $M > 0$, such that 
\begin{equation}\label{Eq.NBMRed1}
\mathbb{P}\left( \sup_{t \in [-\beta, \beta]} \hat{\mathcal{L}}_{J_a + 1}^N(t) > M \right) < \epsilon.
\end{equation} 

Put $\gamma_1 = 2\beta$, $\gamma_2 = 4 \beta$ and $K = J_a+1$. From Proposition \ref{Prop.FDFlat} and (\ref{Eq.LNHatMatch}), we can find $M_1 > 0$, sufficiently large, so that for all $N \in \mathbb{N}$ and $r \in \{\gamma_1, \gamma_2\}$
\begin{equation}\label{Eq.LNHatGamma}
\mathbb{P}\left(H^r_N\right) < \epsilon/4, \mbox{ where } H^r_N = \left\{|\hat{\mathcal{L}}^N_K(r)| \geq M_1 \right\}.
\end{equation} 
We also define
\begin{equation}\label{Eq.LNHatGamma2}
\tilde{H}^{\gamma_1}_N = \left\{\hat{\mathcal{L}}^N_{K}(\gamma_1) \geq M_1 \right\}, \mbox{ and note } \mathbb{P}\left(\tilde{H}^{\gamma_1}_N\right) < \epsilon/4, \mbox{ as } \tilde{H}^{\gamma_1}_N \subseteq H^{\gamma_1}_N.
\end{equation}

Let $A$ be as in Lemma \ref{Lem.NotTooLow} for $k = J_{a} + 1$ and $\epsilon$ replaced with $\epsilon/4$. From (\ref{Eq.NotTooLow}), we conclude that if $\vec{x}, \vec{y} \in \weyl_{J_a+1}$ satisfy $x_{J_a+ 1} > \max(P_1, g(a))$, $y_{J_a+1} > \max(P_2,g(b))$ for some $P_1, P_2 \in \mathbb{R}$, and continuous $g$, then for each $t \in [a,b]$
\begin{equation}\label{Eq.GridBrownianEnsembleNBM}
\mathbb{P}_{\mathrm{avoid}}^{a,b, \vec{x}, \vec{y}, \infty, g}\left( \mathcal{Q}_{J_a+1}(t)  \leq P_1  \cdot \frac{b - t}{b - a} + P_2 \cdot \frac{t - a}{b - a} - A \sqrt{b-a}   \right) < \frac{\epsilon}{4}.
\end{equation}
We suppose that $M$ is sufficiently large, so that $2M/5 -  A \sqrt{5 \beta} \geq 2M_1$. In particular, we note that if $s \in [-\beta, \beta]$, we have 
\begin{equation}\label{Eq.GridLargeNNBM}
M \cdot \frac{\gamma_2 - \gamma_1}{\gamma_2 - s} - M_1 \cdot \frac{\gamma_1 - s}{\gamma_2 - s} -  A \sqrt{\gamma_2 - s} \geq M_1.  
\end{equation}
We proceed to prove (\ref{Eq.NBMRed1}) for this choice of $M$.\\

{\bf \raggedleft Step 2.} The argument we present here is similar to Step 2 of the proof of Proposition \ref{Prop.CloseToZero}.

Fix a finite set $\mathsf{S} = \{s_1, \dots, s_m\} \subset [-\beta, \beta]$ with $s_1 < s_2  < \cdots < s_m$, and define the events 
$$E_N^v = \left\{ \hat{\mathcal{L}}^{N}_{J_a + 1}(s_v) > M \mbox{ and } \hat{\mathcal{L}}^{N}_{J_a + 1}(s_i) \leq M \mbox{ for } i \in \llbracket 1, v-1 \rrbracket \right\}.$$
We also set $E_N^{\mathsf{S}} = \bigsqcup_{v = 1}^m E_N^v$. We further define $\vec{x}^v, \vec{y} \in \weyl_{J_a + 1}$ via
$$x_i^v = \hat{\mathcal{L}}^{N}_{i}(s_v), \hspace{2mm} y_i = \hat{\mathcal{L}}^{N}_{i}(\gamma_2) \mbox{ for } i\in \llbracket 1, J_a + 1 \rrbracket,$$
and observe that on $(H^{\gamma_2}_N)^c \cap E_N^v$, we have 
\begin{equation}\label{Eq.GridBoundaryLBNBM}
x_{J_a+1}^v > M, \hspace{2mm} y_{J_a+1} > -  M_1.
\end{equation}

Notice that on $(H^{\gamma_2}_N)^c \cap E_N^v$, we have
\begin{equation}\label{Eq.GridMiniTowerNBM}
\begin{split}
&\mathbb{P}_{\mathrm{avoid}}^{s_v,\gamma_2, \vec{x}^v, \vec{y}, \infty, g_v}\left( \mathcal{Q}_{J_a+1}(\gamma_1)  \leq  M_1\right)  \\
& \leq \mathbb{P}_{\mathrm{avoid}}^{s_v,\gamma_2, \vec{x}^v, \vec{y}, \infty, g_v}\left( \mathcal{Q}_{J_a+1}(\gamma_1)  \leq M \cdot \frac{\gamma_2 - \gamma_1}{\gamma_2 - s_v} - M_1 \cdot \frac{\gamma_1 - s_v}{\gamma_2 - s_v} -  A \sqrt{\gamma_2 - s_v}\right) < \frac{\epsilon}{4}.
\end{split}
\end{equation}
The first inequality follows from (\ref{Eq.GridLargeNNBM}), and the second from (\ref{Eq.GridBoundaryLBNBM}) and (\ref{Eq.GridBrownianEnsembleNBM}) with 
$$P_1 = M, \hspace{2mm} P_2 = -  M_1, \hspace{2mm} a = s_v, \hspace{2mm} b = \gamma_2, \hspace{2mm} t = \gamma_1, \hspace{2mm} g_v(s) = \hat{\mathcal{L}}^N_{J_a+2}(s) \mbox{ for } s \in [s_v, \gamma_2].$$

With the same notation as above, we have the following tower of inequalities  
\begin{equation}\label{Eq.GridMainTowerNBM}
\begin{split}
&\mathbb{P}\left( E_N^v \cap (\tilde{H}^{\gamma_1}_N)^c \cap (H^{\gamma_2}_N)^c  \right) = \mathbb{E} \left[ {\bf 1}_{E_N^v \cap (H^{\gamma_2}_N)^c } \cdot \mathbb{E}\left[ {\bf 1}_{(\tilde{H}^{\gamma_1}_N)^c} {\big \vert} \mathcal{F}_{\mathrm{ext}} ( \llbracket 1, J_a+1 \rrbracket \times (s_v,\gamma_2))\right] \right] \\
& = \mathbb{E} \left[ {\bf 1}_{E_N^v \cap (H^{\gamma_2}_N)^c } \cdot \mathbb{P}^{s_v, \gamma_2, \vec{x}^v, \vec{y}, \infty, g_v}_{\mathrm{avoid}}\left(  \left|\mathcal{Q}_{J_a+1}(\gamma_1) \right| \leq  M_1  \right) \right]\\
& \leq \mathbb{E} \left[ {\bf 1}_{E_N^v \cap (H^{\gamma_2}_N)^c} \cdot \frac{\epsilon}{4} \right] = \frac{\epsilon}{4} \cdot \mathbb{P}\left(E_N^v \cap (H^{\gamma_2}_N)^c \right). 
\end{split}
\end{equation}
Indeed, the first equality follows from the tower property of conditional expectation and the fact that $E_N^v,  H^{\gamma_2}_N \in \mathcal{F}_{\mathrm{ext}} ( \llbracket 1, J_a+1 \rrbracket \times (s_v,\gamma_2))$ (this $\sigma$-algebra is as in Definition \ref{Def.BGPVanilla}). The second equality follows from the Brownian Gibbs property for $\hat{\mathcal{L}}^{N}$. In going from the second to the third line, we used (\ref{Eq.GridMiniTowerNBM}).

Summing (\ref{Eq.GridMainTowerNBM}) over $v \in \llbracket 1, m \rrbracket$ and using that $E_N^v$ are pairwise disjoint, we conclude
$$\mathbb{P}\left( E_N^{\mathsf{S}} \cap (\tilde{H}^{\gamma_1}_N)^c  \cap (H^{\gamma_2}_N)^c    \right) \leq \mathbb{P}\left( E_N^{\mathsf{S}}\cap (H^{\gamma_2}_N)^c \right) \cdot \frac{\epsilon}{4} \leq \frac{\epsilon}{4}.$$
Taking suprema over $\mathsf{S}$ and using that $\hat{\mathcal{L}}^{N}_{J_a+1}$ is continuous, we conclude 
\begin{equation*}
\mathbb{P}\left(\left\{ \sup_{t \in [-\beta, \beta]} \hat{\mathcal{L}}_{J_a + 1}^N(t) > M \right\} \cap (\tilde{H}^{\gamma_1}_N)^c \cap (H^{\gamma_2}_N)^c \right) \leq \frac{\epsilon}{4},
\end{equation*}
which together with (\ref{Eq.LNHatGamma}) and (\ref{Eq.LNHatGamma2}) imply (\ref{Eq.NBMRed1}).
\end{proof}

%
%
\subsection{Proof of Theorem \ref{Thm.ConvAiry}}\label{Section6.2} We continue with the same notation as in the statement of the theorem. For clarity, we split the proof into three steps. In Step 1, we reduce the proof of the theorem to showing that the modulus of continuity of $\hat{\mathcal{L}}^N_{J_a + k}$, where $\hat{\mathcal{L}}^N = \{\hat{\mathcal{L}}^N_i\}_{i \geq 1}$ is as in (\ref{Eq.LNHatDef}), is well-behaved with high probability for each $k \in \mathbb{N}$, see (\ref{Eq.MOCAiryRed}). In Step 2, we summarize several key consequences of our results in Sections \ref{Section4}, \ref{Section5.1} and \ref{Section6.1}. In particular, we show that on any compact interval $\hat{\mathcal{L}}^N_{J_a}$ is very high, see (\ref{Eq.TopAiry}), and $\hat{\mathcal{L}}^N_{J_a + k+1}$ is not too high with high probability, see (\ref{Eq.NoBigMaxLNHat}). The latter implies that the curves $\{\hat{\mathcal{L}}^N_{J_a + i}: i \in \llbracket 1, k \rrbracket \}$ are too far away from $\hat{\mathcal{L}}^N_{J_a}$ and behave like a finite Brownian ensemble with a well-behaved bottom boundary $g = \hat{\mathcal{L}}^N_{J_a + k+1}$. As such ensembles have a well-behaved modulus of continuity by Lemma \ref{Lem.MOCInside}, we are able to infer the same for $\{\hat{\mathcal{L}}^N_{J_a +i}: i \in \llbracket 1, k \rrbracket \}$, and hence for $\hat{\mathcal{L}}^N_{J_a +k}$. The details behind the argument we outlined in the last sentence are provided in Step 3, where we ultimately prove (\ref{Eq.MOCAiryRed}).\\

{\bf \raggedleft Step 1.} We observe that part (b) of the theorem follows from part (a) and Proposition \ref{Prop.BasicProperties}(b). Consequently, we only need to establish part (a). In what follows we fix $(a,b,c) \in \parP$, such that $J_a < \infty$, and proceed to prove (\ref{Eq.ConvAiryA}).

We claim that for each $k \in \mathbb{N}$ and $\epsilon, \eta, \alpha > 0$, we can find $N_0 \in \mathbb{N}$ and $\delta > 0$, depending on $k, \epsilon, \eta, \alpha$, the parameters $(a,b,c)$, and the sequence $\{T_N\}_{N \geq 1}$, such that for $N \geq N_0$
\begin{equation}\label{Eq.MOCAiryRed}
\mathbb{P}\left(  w(\hat{\mathcal{L}}^N_{J_a + k}[-\alpha, \alpha], \delta) \geq \eta \right) < \epsilon,
\end{equation}
where the modulus of continuity $w(\hat{\mathcal{L}}^N_{J_a + k}[-\alpha, \alpha], \delta)$ is as in (\ref{Eq.MOCDef}) with $a = -\alpha$ and $b = \alpha$. We prove (\ref{Eq.MOCAiryRed}) in the steps below. Here, we assume its validity and finish the proof of part (a). \\

From (\ref{Eq.LNHatMatch}) and (\ref{Eq.MOCAiryRed}), we conclude for any $\alpha, \eta > 0$ and $k \in \mathbb{N}$ that 
\begin{equation}\label{Eq.MOCAiryRed2}
\lim_{\delta \rightarrow 0+} \limsup_{N \rightarrow \infty}\mathbb{P}\left(  w(\mathcal{A}^N_{k}[-\alpha, \alpha], \delta) \geq \eta \right) = 0.
\end{equation}
From Proposition \ref{Prop.FDFlat}, we know
\begin{equation}\label{Eq.FDConvAiryRed}
\left(\mathcal{A}^N_i(t): t \in \mathbb{R} \mbox{, }i  \in \mathbb{N} \right) \overset{f.d.}{\rightarrow}\left( \mathcal{A}^{0,0,0}_i(t): t \in \mathbb{R} \mbox{, }i \in \mathbb{N} \right).
\end{equation}
Equation (\ref{Eq.FDConvAiryRed}) in particular shows that the sequences $\{\mathcal{A}^N_k(0)\}_{N \geq 1}$ for each $k \in \mathbb{N}$ are tight, which together with (\ref{Eq.MOCAiryRed2}) shows that the tightness criteria of \cite[Lemma 2.4]{DEA21} are satisfied by $\mathcal{A}^N$.

From \cite[Lemma 2.4]{DEA21}, we conclude that $\{\mathcal{A}^N\}_{N \geq 1}$ is a tight sequence of random elements in $C(\mathbb{N} \times \mathbb{R} )$. In addition, (\ref{Eq.FDConvAiryRed}) shows that any subsequential limit has the same finite-dimensional distributions as $\mathcal{A}^{0,0,0}$. As finite-dimensional sets form a separating class, see \cite[Lemma 3.1]{DimMat}, we conclude that $\mathcal{A}^N \Rightarrow \mathcal{A}^{0,0,0}$, which is precisely (\ref{Eq.ConvAiryA}).\\

{\bf \raggedleft Step 2.} In this step, we specify the choices of $N_0$ and $\delta$, for which we will prove (\ref{Eq.MOCAiryRed}). In what follows we fix $\beta = \alpha + 1$ and adopt the convention $\hat{\mathcal{L}}^N_0 = \infty$.

From Proposition \ref{Prop.NoBigMax}, with $k$ replaced with $k+1$ and $\epsilon$ replaced with $\epsilon/4$, we can find $M > 0$, such that 
\begin{equation}\label{Eq.NoBigMaxLNHat}
\mathbb{P}\left( F_N^{\mathrm{bot}}\right) < \epsilon/4, \mbox{ where }F_N^{\mathrm{bot}} = \left\{\sup_{t \in [-\beta, \beta]} \hat{\mathcal{L}}^N_{J_a + k + 1}(t) > M  \right\}.
\end{equation}
From (\ref{Eq.LNHatMatch}) and (\ref{Eq.FDConvAiryRed}), we can find $M_1 > 0$, such that
\begin{equation}\label{Eq.EsideAiry}
\mathbb{P}\left(E^{\mathrm{side}}_N\right) < \epsilon/4, \mbox{ where } E^{\mathrm{side}}_N = \left\{\left|\hat{\mathcal{L}}_{J_a + i}^N(\gamma)\right| > M_1 \mbox{ for some } (i,\gamma) \in \llbracket 1, k \rrbracket \times \{-\beta, \beta\} \right\}.
\end{equation}

From Lemma \ref{Lem.NoBigMaxBLE} applied to $k$ as in the present setup, $a = -\beta$, $b = \beta$, $M^{\mathrm{side}} = M_1$, $M^{\mathrm{bot}} = M$, and $\epsilon = 1/2$, we can find $A > 0$, such that the following holds. If 
\begin{enumerate}
\item[(i)] $\vec{x}, \vec{y} \in \weyl_k$ with $|x_i|, |y_i| \leq M_1$ for $i \in \llbracket 1, k \rrbracket$,
\item[(ii)] $g: [-\beta, \beta] \rightarrow [-\infty, \infty)$ is continuous and satisfies $x_k > g(a)$, $y_k > g(b)$, and $g(t) \leq M$ for all $t \in [-\beta, \beta]$,
\end{enumerate}
then
\begin{equation}\label{Eq.NoBigMaxBLEAiry}
\mathbb{P}_{\mathrm{avoid}}^{-\beta,\beta, \vec{x}, \vec{y}, \infty, g}\left( \mathcal{Q}_i(t)  \geq A  \mbox{ for some } (i,t) \in \llbracket 1,k \rrbracket \times [-\beta, \beta]  \right) < 1/2.
\end{equation}

From Lemma \ref{Lem.MOCInside} applied to $k, \eta$ as in the present setup, $c = -\beta$, $a = -\alpha$, $b = \alpha$, $d = \beta$, $M^{\mathrm{side}} = M_1$, $M^{\mathrm{bot}} = M$, and $\epsilon$ replaced with $\epsilon/8$, we can find $\delta > 0$, such that if $\vec{x}, \vec{y}, g$ satisfy conditions (i) and (ii) above, then
\begin{equation}\label{Eq.MOCInsideAiry}
\mathbb{P}_{\mathrm{avoid}}^{-\beta,\beta, \vec{x}, \vec{y}, \infty, g}\left( \max_{i \in \llbracket 1, k \rrbracket} w(\mathcal{Q}_i[-\alpha,\alpha], \delta) \geq \eta \right) < \epsilon/8. 
\end{equation}
This specifies our choice of $\delta$. 

Finally, we observe that we have 
\begin{equation}\label{Eq.NoCeilingAiry}
\lim_{N \rightarrow \infty} \mathbb{P} \left( \inf_{t \in [-\beta, \beta]}\hat{\mathcal{L}}^N_{J_a}(t) > A  \right) = 1.
\end{equation}
Indeed, if $J_a =0$, we have $\hat{\mathcal{L}}_0^N = \infty$ and (\ref{Eq.NoCeilingAiry}) is automatic. If $J_a > 0$, then $J_a = \mathsf{M}_p^a$ for some $p \in \mathbb{N}$ and from Proposition \ref{Prop.CloseToZero} applied to $\alpha = 1/2$, $\beta = 2$ and $k = p$, we have that the sequence 
$$\sup_{s \in [1/2, 2]}\left|\mathcal{L}^N_{\mathsf{m}_p^a}(s) \right| =  \sup_{s \in [1/2, 2]} (2T_N)^{-1/2} \cdot \left| \mathcal{A}^{a,b,c}_{J_a}(sT_N) + 2sT_N/v_p^a - s^2T_N^2 \right|$$
is tight. The above equality follows from (\ref{Eq.LNTilde}) and (\ref{Eq.LNMatch}). Combining the last statement with (\ref{Eq.LNHatDef}) and the fact that $T_N + t \in [T_N/2, 2T_N]$ for all $t \in [-\beta, \beta]$ and large $N$, we conclude that the sequence
$$ \sup_{t \in [-\beta, \beta]} (2T_N)^{-1/2} \cdot \left| \left(2^{1/2}\hat{\mathcal{L}}^N_{J_a}(t) + t^2 \right) + 2(t+T_N)/v_p^a - (t +T_N)^2 \right|$$
is tight, which implies (\ref{Eq.NoCeilingAiry}). 

In view of (\ref{Eq.NoCeilingAiry}), we can find $N_0 \in \mathbb{N}$, such that for all $N \geq N_0$
\begin{equation}\label{Eq.TopAiry}
\mathbb{P} \left(F_N^{\mathrm{top}} \right) < \epsilon/4, \mbox{ where } F_N^{\mathrm{top}} = \left\{ \inf_{t \in [-\beta, \beta]}\hat{\mathcal{L}}^N_{J_a}(t) \leq A  \right\}.
\end{equation}
This specifies our choice of $N_0$.\\

{\bf \raggedleft Step 3.} In this final step, we prove (\ref{Eq.MOCAiryRed}) for $N_0, \delta$ as in Step 2.

We define $\vec{x}^N, \vec{y}^N \in \weyl_{k}$, $f^N: [-\beta, \beta] \rightarrow (-\infty, \infty]$ and $g^N: [-\beta, \beta] \rightarrow [-\infty, \infty)$ through
$$x_i^N = \hat{\mathcal{L}}_{J_a + i}^N(-\beta), \hspace{2mm} y_i^N = \hat{\mathcal{L}}_{J_a + i}^N(\beta), \mbox{ for $i \in \llbracket 1, k \rrbracket$}, \hspace{2mm} f^N = \hat{\mathcal{L}}_{J_a}^N[-\beta, \beta], \hspace{2mm} g^N = \hat{\mathcal{L}}_{J_a + k + 1}^N[-\beta, \beta],$$
and observe that on $(E_N^{\mathrm{side}})^c \cap (F_N^{\mathrm{top}})^c \cap (F_N^{\mathrm{bot}})^c$, we have 
\begin{equation}\label{Eq.BoundaryAiry}
\left|x_i^N\right|, \left|y_i^N\right| \leq M_1, \mbox{ for $i \in \llbracket 1, k \rrbracket$}, \hspace{2mm} f^N(t) > A, \hspace{2mm} g^N(t) \leq M, \mbox{ for $t \in [-\beta, \beta]$}.  
\end{equation}

We also define for any $f: [-\beta, \beta] \rightarrow (-\infty, \infty]$, the function $G_{f}: C( \llbracket 1, k \rrbracket \times [-\beta, \beta] ) \rightarrow [0,1]$ via
$$G_{f}\left( \{h_i\}_{i = 1}^{k} \right) = {\bf 1} \left\{f(t) > h_1(t) \mbox{ for all } t \in [-\beta, \beta] \right\}.$$

We have the following tower of inequalities on $(E_N^{\mathrm{side}})^c \cap (F_N^{\mathrm{top}})^c \cap (F_N^{\mathrm{bot}})^c$ for $N \geq N_0$:
\begin{equation}\label{Eq.AiryConvMiniTower}
\begin{split}
&\mathbb{P}_{\mathrm{avoid}}^{-\beta,\beta, \vec{x}^N, \vec{y}^N, f^N, g^N}\left( \max_{i \in \llbracket 1,k \rrbracket} w(\mathcal{Q}_i[-\alpha, \alpha], \delta) \geq \eta \right) \\
& = \frac{\mathbb{E}_{\mathrm{avoid}}^{-\beta,\beta, \vec{x}^N, \vec{y}^N, \infty, g^N}\left[ {\bf 1}\left\{\max_{i \in \llbracket 1, k \rrbracket} w(\mathcal{Q}_i[-\alpha, \alpha], \delta) \geq \eta \right\} \cdot G_{f^N} (\mathcal{Q}) \right]}{\mathbb{E}_{\mathrm{avoid}}^{-\beta,\beta, \vec{x}^N, \vec{y}^N, \infty, g^N}\left[ G_{f^N} (\mathcal{Q}) \right]} \\
& \leq 2 \mathbb{E}_{\mathrm{avoid}}^{-\beta,\beta, \vec{x}^N, \vec{y}^N, \infty, g^N}\left[ {\bf 1}\left\{\max_{i \in \llbracket 1, k \rrbracket} w(\mathcal{Q}_i[-\alpha, \alpha], \delta) \geq \eta \right\} \cdot G_{f^N} (\mathcal{Q})  \right] < \epsilon/4.
\end{split}
\end{equation}
Indeed, in going from the first to the second line, we used Definition \ref{Def.fgAvoidingBE}, and in going from the second to the third line, we used 
$$\mathbb{E}_{\mathrm{avoid}}^{-\beta,\beta, \vec{x}^N, \vec{y}^N, \infty, g^N}\left[ G_{f^N} (\mathcal{Q}) \right] \geq \mathbb{E}_{\mathrm{avoid}}^{-\beta,\beta, \vec{x}^N, \vec{y}^N, \infty, g^N}\left[ G_{A} (\mathcal{Q}) \right] > 1/2,$$
where the first inequality follows from (\ref{Eq.BoundaryAiry}), while the second from (\ref{Eq.NoBigMaxBLEAiry}). The last inequality in (\ref{Eq.AiryConvMiniTower}) follows from (\ref{Eq.MOCInsideAiry}) and the fact that $G_{f^N}(\mathcal{Q}) \in [0,1]$.

With the same notation as above, and $E_N = E_N^{\mathrm{side}} \cup F_N^{\mathrm{top}} \cup F_N^{\mathrm{bot}}$, we have the following tower of inequalities for $N \geq N_0$ 
\begin{equation}\label{Eq.AiryConvTower}
\begin{split}
&\mathbb{P}\left( E_N^c  \cap \left\{ \max_{i \in \llbracket 1, k \rrbracket} w(\hat{\mathcal{L}}^N_{J_a + i}[-\alpha,\alpha], \delta) \geq \eta \right\} \right)  \\
& = \mathbb{E}\left[ {\bf 1}_{E_N^c} \cdot \mathbb{E}\left[ {\bf 1}\left\{ \max_{i \in \llbracket 1, k \rrbracket} w(\hat{\mathcal{L}}^N_{J_a + i}[-\alpha,\alpha], \delta) \geq \eta \right\} \bigg{\vert} \mathcal{F}_{\mathrm{ext}} ( \llbracket J_a + 1, J_a + k \rrbracket \times (-\beta, \beta)) \right] \right] \\
& = \mathbb{E}\left[{\bf 1}_{E_N^c} \cdot \mathbb{P}_{\mathrm{avoid}}^{-\beta,\beta, \vec{x}^N, \vec{y}^N, f^N, g^N}\left( \max_{i \in \llbracket 1, k \rrbracket} w(\mathcal{Q}_i[-\alpha, \alpha], \delta) \geq \eta \right) \right] \leq \mathbb{E}\left[{\bf 1}_{E_N^c} \cdot (\epsilon/4)\right] \leq \epsilon/4.
\end{split}
\end{equation}
Indeed, the first equality follows by the tower property for conditional expectation and the fact that $E_N \in \mathcal{F}_{\mathrm{ext}} ( \llbracket J_a + 1, J_a + k \rrbracket \times (-\beta, \beta))$ (this $\sigma$-algebra is as in Definition \ref{Def.BGPVanilla}). The second equality follows from the Brownian Gibbs property for $\hat{\mathcal{L}}^{N}$, and the first inequality on the third line follows by (\ref{Eq.AiryConvMiniTower}).

Combining (\ref{Eq.NoBigMaxLNHat}), (\ref{Eq.EsideAiry}), (\ref{Eq.TopAiry}), and (\ref{Eq.AiryConvTower}), we conclude for $N \geq N_0$ that 
$$ \mathbb{P}\left( \max_{i \in \llbracket 1, k \rrbracket} w(\hat{\mathcal{L}}^N_{J_a + i}[-\alpha,\alpha], \delta) \geq \eta \right) \leq \epsilon/4 + \mathbb{P}(E_N^{\mathrm{side}}) + \mathbb{P}(F_N^{\mathrm{top}}) + \mathbb{P}(F_N^{\mathrm{bot}}) < \epsilon.$$
The last inequality implies (\ref{Eq.MOCAiryRed}).

\bibliographystyle{amsplain}
\bibliography{PD}

\end{document}